\documentclass[oneside,english]{amsart}
\usepackage[T1]{fontenc}
\usepackage[latin9]{inputenc}
\usepackage{units}
\usepackage{amstext}
\usepackage{amsthm}
\usepackage{amssymb}

\makeatletter
\numberwithin{equation}{section}
\numberwithin{figure}{section}
  \theoremstyle{remark}
  \newtheorem*{rem*}{\protect\remarkname}
  \theoremstyle{plain}
  \newtheorem*{conjecture*}{\protect\conjecturename}
\theoremstyle{plain}
\newtheorem{thm}{\protect\theoremname}[section]
  \theoremstyle{plain}
  \newtheorem{prop}[thm]{\protect\propositionname}
  \theoremstyle{plain}
  \newtheorem{lem}[thm]{\protect\lemmaname}
  \theoremstyle{plain}
  \newtheorem{cor}[thm]{\protect\corollaryname}
  \theoremstyle{definition}
  \newtheorem{defn}[thm]{\protect\definitionname}
  \theoremstyle{plain}
  \newtheorem{conjecture}[thm]{\protect\conjecturename}

\makeatother

\usepackage{babel}
  \providecommand{\conjecturename}{Conjecture}
  \providecommand{\corollaryname}{Corollary}
  \providecommand{\definitionname}{Definition}
  \providecommand{\lemmaname}{Lemma}
  \providecommand{\propositionname}{Proposition}
  \providecommand{\remarkname}{Remark}
\providecommand{\theoremname}{Theorem}

\begin{document}

\title[Competition in periodic media: II \textendash{} Segregation of pulsating
fronts]{Competition in periodic media: II \textendash{} Segregative limit
of pulsating fronts and \textquotedblleft Unity is not Strength\textquotedblright -type
result}

\author{Léo Girardin$^{1}$ \and Grégoire Nadin$^{2}$}

\thanks{The research leading to these results has received funding from the
European Research Council under the European Union's Seventh Framework
Programme (FP/2007-2013) / ERC Grant Agreement n.321186 \textendash{}
ReaDi \textendash{} Reaction\textendash Diffusion Equations, Propagation
and Modelling held by Henri Berestycki. \\
$^{1,2}$ Laboratoire Jacques-Louis Lions, CNRS UMR 7598, Université
Pierre et Marie Curie, 4 place Jussieu, 75005 Paris, France}

\email{$^{1}$ girardin@ljll.math.upmc.fr}

\email{$^{2}$ nadin@ljll.math.upmc.fr}
\begin{abstract}
This paper is concerned with the limit, as the interspecific competition
rate goes to infinity, of pulsating front solutions in space-periodic
media for a bistable two-species competition\textendash diffusion
Lotka\textendash Volterra system. We distinguish two important cases:
null asymptotic speed and non-null asymptotic speed. In the former
case, we show the existence of a segregated stationary equilibrium.
In the latter case, we are able to uniquely characterize the segregated
pulsating front, and thus full convergence is proved. The segregated
pulsating front solves an interesting free boundary problem. We also
investigate the sign of the speed as a function of the parameters
of the competitive system. We are able to determine it in full generality,
with explicit conditions depending on the various parameters of the
problem. In particular, if one species is sufficiently more motile
or competitive than the other, then it is the invader. This is an
extension of our previous work in space-homogeneous media. 
\end{abstract}



\maketitle
\tableofcontents{}

\section*{Introduction}

This is the second part of a sequel to our previous article \cite{Girardin_Nadin_2015}.
In the prequel, we studied the sign of the speed of bistable traveling
wave solutions of the following competition\textendash diffusion problem:
\[
\left\{ \begin{matrix}\partial_{t}u_{1}-\partial_{xx}u_{1}=u_{1}\left(1-u_{1}\right)-ku_{1}u_{2} & \mbox{ in }\left(0,+\infty\right)\times\mathbb{R}\,\\
\partial_{t}u_{2}-d\partial_{xx}u_{2}=ru_{2}\left(1-u_{2}\right)-\alpha ku_{1}u_{2} & \mbox{ in }\left(0,+\infty\right)\times\mathbb{R}.
\end{matrix}\right.
\]

We proved that, as $k\to+\infty$, the speed of the traveling wave
connecting $\left(1,0\right)$ to $\left(0,1\right)$ converges to
a limit which has exactly the sign of $\alpha^{2}-rd$. In particular,
if $\alpha=r=1$ and if $k$ is large enough, the more motile species
is the invader: this is what we called the \textquotedblleft Unity
is not strength\textquotedblright{} result.

In view of this result, it would seem natural to try to generalize
it in heterogeneous spaces, that is to systems with non-constant coefficients.
Is the more motile species still the invading one?

Competition\textendash diffusion problems in bounded heterogeneous
spaces with various boundary conditions have been widely studied during
the past decades. Dockery, Hutson, Mischaikow and Pernarowski \cite{Dockery_1998}
showed (in particular) that for the heterogeneous system: 
\[
\left\{ \begin{matrix}\partial_{t}u_{1}-d_{1}\Delta_{x}u_{1}=a_{1}\left(x\right)u_{1}-u_{1}^{2}-u_{1}u_{2} & \mbox{ in }\left(0,+\infty\right)\times\Omega\\
\partial_{t}u_{2}-d_{2}\Delta_{x}u_{2}=a_{2}\left(x\right)u_{2}-u_{2}^{2}-u_{1}u_{2} & \mbox{ in }\left(0,+\infty\right)\times\Omega
\end{matrix}\right.
\]
with $a_{1}$ and $a_{2}$ non-constant functions, $d_{1}$ and $d_{2}$
constant, $\Omega$ a bounded open subset of some Euclidean space
and homogeneous Neumann boundary conditions, the persistent species
is actually the less motile one. The interspecific competition rate
of this system is equal to $1$ and the system is therefore monostable.
On the contrary, as soon as the competition rate is large enough,
the system is bistable. We wonder whether this qualitative change
might be sufficient to reverse their conclusion. If we are able to
extend in some satisfying way our space-homogeneous result, then the
conclusion will be reversed indeed.

In the first part \cite{Girardin_2016} of this sequel, the first
author studied the existence of bistable pulsating front solutions
for the following problem:
\[
\left\{ \begin{matrix}\partial_{t}u_{1}=\partial_{xx}u_{1}+u_{1}f_{1}\left(u_{1},x\right)-ku_{1}u_{2} & \mbox{ in }\left(0,+\infty\right)\times\mathbb{R}\,\\
\partial_{t}u_{2}=d\partial_{xx}u_{2}+u_{2}f_{2}\left(u_{2},x\right)-\alpha ku_{1}u_{2} & \mbox{ in }\left(0,+\infty\right)\times\mathbb{R}.
\end{matrix}\right.
\]

Here, the non-linearities $\left(u,x\right)\mapsto uf_{i}\left(u,x\right)$,
$i\in\left\{ 1,2\right\} $, are of ``KPP\textquotedblright -type
and, most importantly, are spatially periodic. Thanks to Fang\textendash Zhao\textquoteright s
theorem \cite{Fang_Zhao_2011}, it was showed that, provided $k$
is large enough and $\left(f_{1},f_{2}\right)$ satisfies a high-frequency
algebraic hypothesis (we highlight that the condition was algebraic
and not asymptotic), there exists indeed such a pulsating front. 

While the forthcoming main ideas might be generalizable to systems
with periodic diffusion and interspecific competition rates, an existence
result is lacking. Therefore we naturally stick with the aforementioned
system. Let us recall moreover that the fully heterogeneous problem
(non-periodic non-constant coefficients) is, as far as we know, still
completely open at this time. 

Let us recall as well that several important results about scalar
reaction\textendash diffusion equations in periodic media have been
established recently (about ``KPP\textquotedblright -type, see \cite{Berestycki_Ham_1,Berestycki_Ham_2,Nadin_2009,Nadin_2011,Nolen_Rudd_Xin};
about ``ignition\textquotedblright -type and monostable non-linearities,
see \cite{Berestycki_Ham_3}; about bistable non-linearities, see
\cite{Ding_Hamel_Zhao,Ding_Hamel_Zhao2,Xin_1991}). The first author
used extensively the results about ``KPP\textquotedblright -type
equations in \cite{Girardin_2016}. In the forthcoming work, we will
use the whole collection of results. Especially, we will use several
times, in slightly different contexts, the sliding method of Berestycki\textendash Hamel
\cite{Berestycki_Ham_3}.

Integration over a bounded domain with Neumann boundary conditions
and over a periodicity cell are somehow similar operations and thus
Neumann and periodic boundary conditions yield in general analogous
results. The periodic extension of the persistence result by Dockery
and his collaborators seems in fact quite straightforward and, conversely,
it should be possible to adapt the forthcoming ideas to determine
the persistent species in a bistable space-heterogeneous Neumann problem
with large competition rate. The comparison is therefore even more
meaningful.

The competition-induced segregation phenomenon highlighted by Dancer,
Terracini and others (see for instance \cite{Conti_Terracin,Dancer_Du_1994,Dancer_Hilhors,Dancer_2010,Dancer06})
has been one of our main tools in the preceding pair of articles \cite{Girardin_2016,Girardin_Nadin_2015}
and will still be a cornerstone here. In particular, segregation in
two or more dimensions generically yields free boundary problems and
this will be a major difference between the space-homogeneous case
and this study: here, we will need to dedicate a few pages to the
natural free boundary problem induced by the segregation of pulsating
fronts. Thanks to the specific setting of pulsating fronts (monotonicity
in time, spatial periodicity of the profile, limiting conditions,
etc.), we will be able to prove that the free boundary is the graph
of a strictly monotonic, bijective and continuous function without
resorting to blow-up arguments or monotonicity formulas. We believe
that our approach of the free boundary has interest of its own and
that the ideas presented here might fond applications in other frameworks.

The following pages will be organized as follows: in the first section,
the core hypotheses and framework will be precisely formulated and
the main results stated. The second section will focus on the so-called
``segregative limit\textquotedblright{} and will finally lead us
to the third section and the statement of the periodic extension of
the \textquotedblleft Unity is not strength\textquotedblright{} theorem.

\section{Preliminaries and main results}
\begin{rem*}
Subsections \ref{subsec:Preliminaries} and \ref{subsec:A-few-more-prelim}
are mostly a repetition of the preliminaries of the first author\textquoteright s
article \cite{Girardin_2016} where the existence of competitive pulsating
fronts was investigated. A reader well aware of this article may safely
skip these. On the contrary, Subsections \ref{subsec:UINS_Theorem}
and \ref{subsec:Comparison-first-second-parts} respectively state
the main results of this article and highlight the differences between
the present set of technical hypotheses and that of the first author\textquoteright s
article \cite{Girardin_2016}. 
\end{rem*}
Let $d,k,\alpha,L>0$, $C=\left(0,L\right)\subset\mathbb{R}$ and
$\left(f_{1},f_{2}\right):[0,+\infty)\times\mathbb{R}\to\mathbb{R}^{2}$
$L$-periodic with respect to its second variable. For any $u:\mathbb{R}^{2}\to[0,+\infty)$
and $i\in\left\{ 1,2\right\} $, we refer to $\left(t,x\right)\mapsto f_{i}\left(u\left(t,x\right),x\right)$
as $f_{i}\left[u\right]$. Our interest lies in the following competition\textendash diffusion
problem:
\[
\left\{ \begin{matrix}\partial_{t}u_{1}=\partial_{xx}u_{1}+u_{1}f_{1}\left[u_{1}\right]-ku_{1}u_{2},\\
\partial_{t}u_{2}=d\partial_{xx}u_{2}+u_{2}f_{2}\left[u_{2}\right]-\alpha ku_{1}u_{2}.
\end{matrix}\right.\quad\left(\mathcal{P}_{k}\right)
\]

\subsection{Preliminaries\label{subsec:Preliminaries}}

\subsubsection{Redaction conventions. }
\begin{itemize}
\item Mirroring the definition of $f_{1}\left[u\right]$ and $f_{2}\left[u\right]$,
for any function of two real variables $f$ and any real-valued function
$u$ of two real variables, $f\left[u\right]$ will refer to $\left(t,x\right)\mapsto f\left(u\left(t,x\right),x\right)$.
For any real-valued function $u$ of one real variable, $f\left[u\right]$
will refer to $x\mapsto f\left(u\left(x\right),x\right)$. For any
function $f$ of one real variable and any real-valued function $u$
of one or two real variables, $f\left[u\right]$ will simply refer
to $f\circ u$.
\item For the sake of brevity, although we could index everything ($\left(\mathcal{P}\right)$,
$u_{1}$, $u_{2}$\dots ) on $k$ and $d$, the dependencies on $k$
or $d$ will mostly be implicit and will only be made explicit when
it definitely facilitates the reading.
\item Since we consider the limit of this system when $k\to+\infty$, many
(but finitely many) results will only be true when \textquotedblleft $k$
is large enough\textquotedblright . Hence, we define by induction
the positive number $k^{\star}$, whose value is initially $1$ and
is updated each time a statement is only true when \textquotedblleft $k$
is large enough\textquotedblright{} in the following way: if the statement
is true for any $k\geq k^{\star}$, the value of $k^{\star}$ is unchanged;
if, conversely, there exists $K>k^{\star}$ such that the statement
is true for any $k\geq K$ but false for any $k\in[k^{\star},K)$,
the value of $k^{\star}$ becomes that of $K$. In the text, we will
indifferently write \textquotedblleft for $k$ large enough\textquotedblright{}
or \textquotedblleft provided $k^{\star}$ is large enough\textquotedblright .
Moreover, when $k$ indexes appear, they \textit{a priori} indicate
that we are considering families indexed on $[k^{\star},+\infty)$,
but for the sake of brevity, when sequential arguments involve sequences
indexed themselves on increasing elements of $[k^{\star},+\infty)^{\mathbb{N}}$,
we will not explicitly define these sequences of indexes and will
simply stick with the indexes $k$, reindexing along the course of
the proof the considered objects. In such a situation, the statement
``as $k\to+\infty$'' should be understood unambiguously.
\item Periodicity will always implicitly mean $L$-periodicity (unless explicitly
stated otherwise). For any functional space $X$ on $\mathbb{R}$,
$X_{per}$ denotes the subset of $L$-periodic elements of $X$.
\item We will use the classical partial order on the space of functions
from any $\Omega\subset\mathbb{R}^{N}$ to $\mathbb{R}$: $g\leq h$
if for any $x\in\Omega$ $g\left(x\right)\leq h\left(x\right)$ and
$g<h$ if $g\leq h$ and $g\neq h$. We recall that when $g<h$, there
might still exists $x\in\Omega$ such that $g\left(x\right)=h\left(x\right)$.
If, for any $x\in\Omega$, $g\left(x\right)<h\left(x\right)$, we
use the notation $g\ll h$. In particular, if $g\geq0$, we say that
$g$ is non-negative, if $g>0$, we say that $g$ is non-negative
non-zero, and if $g\gg0$, we say that $g$ is positive. Finally,
if $g_{1}\leq h\leq g_{2}$, we write $h\in\left[g_{1},g_{2}\right]$,
if $g_{1}<h<g_{2}$, we write $h\in\left(g_{1},g_{2}\right)$, and
if $g_{1}\ll h\ll g_{2}$, we write $h\in\left\langle g_{1},g_{2}\right\rangle $. 
\item We will also use the partial order on the space of vector functions
$\Omega\to\mathbb{R}^{N'}$ naturally derived from the preceding partial
order. It will involve similar notations.
\item Functions $f$ of two or more real variables will sometimes be identified
with the maps $t\mapsto\left(x\mapsto f\left(t,x\right)\right)$.
This is quite standard in parabolic theory but we stress that the
variable of the map will always be the first variable of $f$, even
if this variable is not called $t$: we will use indeed functions
of the pair of variables $\left(\xi,x\right)\in\mathbb{R}^{2}$ and
then the maps will be $\xi\mapsto\left(x\mapsto f\left(\xi,x\right)\right)$.
So for instance if we say that a function $f$ of $\left(\xi,x\right)$
is an element of a functional space $X\left(\mathbb{R},Y\right)$,
the latter should be understood unambiguously. 
\end{itemize}

\subsubsection{Hypotheses on the reaction.}

For any $i\in\left\{ 1,2\right\} $, we have in mind functions $f_{i}$
such that the reaction term $uf_{i}\left[u\right]$ is of logistic
type (also known as \textquotedblleft KPP\textquotedblright -type).
At least, we want to cover the largest possible class of $\left(u,x\right)\mapsto\mu\left(x\right)\left(a-u\right)$.
This is made precise by the following assumptions.

{ \renewcommand\labelenumi{($\mathcal{H}_\theenumi$)}
\begin{enumerate}
\item $f_{i}$ is in $\mathcal{C}^{1}\left([0,+\infty)\times\mathbb{R}\right)$.
\item There exists a constant $m_{i}>0$ such that $f_{i}\left[0\right]\geq m_{i}$.
\item $f_{i}$ is decreasing with respect to its first variable and there
exists $a_{i}>0$ such that, for any $x\in\mathbb{R}$, $f_{i}\left(a_{i},x\right)=0$.
\end{enumerate}
}
\begin{rem*}
If $f_{i}$ is in the class of all $\left(u,x\right)\mapsto\mu\left(x\right)\left(a-u\right)$,
then $\mu\in\mathcal{C}_{per}^{1}\left(\mathbb{R}\right)$, $\mu\gg0$
and $a>0$. More generally, from $\left(\mathcal{H}_{1}\right)$,
$\left(\mathcal{H}_{2}\right)$ and the periodicity of $f_{i}\left[0\right]$,
it follows immediately that there exists a constant $M_{i}>m_{i}$
such that $f_{i}\left[0\right]\leq M_{i}$. Without loss of generality,
we assume that $m_{i}$ and $M_{i}$ are optimal, that is $m_{i}=\min\limits _{\overline{C}}f_{i}\left[0\right]$
and $M_{i}=\max\limits _{\overline{C}}f_{i}\left[0\right]$. 
\end{rem*}

\subsubsection{Extinction states}

The periodic principal eigenvalues of $\frac{\mbox{d}^{2}}{\mbox{d}x^{2}}+f_{1}\left[0\right]$
and $d\frac{\mbox{d}^{2}}{\mbox{d}x^{2}}+f_{2}\left[0\right]$ are
negative (as proved by the first author in \cite{Girardin_2016}).
Recall (from Berestycki\textendash Hamel\textendash Roques \cite{Berestycki_Ham_1}
for instance) that the periodic principal eigenvalue of $\mathcal{L}$
is the unique real number $\lambda$ such that there exists a periodic
function $\varphi\gg0$ satisfying: 
\[
\left\{ \begin{matrix}-\mathcal{L}\varphi=\lambda\varphi\mbox{ in }\mathbb{R}\\
\|\varphi\|_{L^{\infty}\left(C\right)}=1
\end{matrix}\right.
\]

From this observation, it follows from Berestycki\textendash Hamel\textendash Roques
\cite{Berestycki_Ham_1} that $a_{1}$ (respectively $a_{2}$) is
the unique periodic non-negative non-zero solution of $-z''=zf_{1}\left[z\right]$
(resp. $-dz''=zf_{2}\left[z\right]$). 

The states $\left(a_{1},0\right)$ and $\left(0,a_{2}\right)$ are
clearly periodic stationary states of $\left(\mathcal{P}_{k}\right)$
(for any $k>k^{\star}$) and are referred to as the \textit{extinction
states} of $\left(\mathcal{P}_{k}\right)$ (remark that they are the
unique periodic stationary states with one null component and the
other one positive, so that it makes sense to call them ``the''
extinction states). Provided $k^{\star}$ is large enough, they are
moreover locally asymptotically stable (again, as proved in \cite{Girardin_2016}). 

We recall also that, for any $k>k^{\star}$, by virtue of the scalar
parabolic comparison principle, any solution $\left(u_{1},u_{2}\right)$
of $\left(\mathcal{P}_{k}\right)$ with initial condition $\left(0,0\right)<\left(u_{1,0},u_{2,0}\right)<\left(a_{1},a_{2}\right)$
satisfies $\left(0,0\right)\ll\left(u_{1},u_{2}\right)\ll\left(a_{1},a_{2}\right)$.

\subsubsection{Pulsating front solutions of $\left(\mathcal{P}\right)$}

Let us add a necessary existence hypothesis. 

{ \renewcommand\labelenumi{($\mathcal{H}_{exis}$)}
\begin{enumerate}
\item There exists $k^{\star}>0$ such that, for any $k>k^{\star}$, there
exists $c_{k}\in\mathbb{R}$ and $\left(\varphi_{1,k},\varphi_{2,k}\right)\in\mathcal{C}^{2}\left(\mathbb{R}^{2}\right)^{2}$
such that the following properties hold.
\begin{itemize}
\item $\left(u_{1,k},u_{2,k}\right):\left(t,x\right)\mapsto\left(\varphi_{1,k},\varphi_{2,k}\right)\left(x-c_{k}t,x\right)$
is a classical solution of $\left(\mathcal{P}_{k}\right)$.
\item $\varphi_{1,k}$ and $\varphi_{2,k}$ are respectively non-increasing
and non-decreasing with respect to their first variable, generically
noted $\xi$. 
\item $\varphi_{1,k}$ and $\varphi_{2,k}$ are periodic with respect to
their second variable, generically noted $x$.
\item As $\xi\to-\infty$, 
\[
\max_{x\in\left[0,L\right]}\left|\left(\varphi_{1,k},\varphi_{2,k}\right)\left(\xi,x\right)-\left(a_{1},0\right)\right|\to0.
\]
\item As $\xi\to+\infty$,
\[
\max_{x\in\left[0,L\right]}\left|\left(\varphi_{1,k},\varphi_{2,k}\right)\left(\xi,x\right)-\left(0,a_{2}\right)\right|\to0.
\]
\end{itemize}
\end{enumerate}
}

The pair $\left(u_{1,k},u_{2,k}\right)$ is referred to as a \textit{pulsating
front solution} of $\left(\mathcal{P}_{k}\right)$ with \textit{speed}
$c_{k}$ and \textit{profile} $\left(\varphi_{1,k},\varphi_{2,k}\right)$.

Before going any further, it is natural to wonder if such a solution
is unique.
\begin{conjecture*}
Let $k>k^{\star}$. Let $\left(\hat{\varphi}_{1},\hat{\varphi}_{2}\right)$
and $\hat{c}$ be respectively the profile and the speed of a pulsating
front solution $\left(\hat{u}_{1},\hat{u}_{2}\right)$ of $\left(\mathcal{P}\right)$.
Then $\hat{c}=c_{k}$ and there exists $\hat{\xi}\in\mathbb{R}$ such
that $\left(\hat{\varphi}_{1},\hat{\varphi}_{2}\right)$ coincides
with:
\[
\left(\xi,x\right)\mapsto\left(\varphi_{1,k},\varphi_{2,k}\right)\left(\xi-\hat{\xi},x\right).
\]
\end{conjecture*}
This conjecture is due to the following observation: in most (if not
all) problems concerned with bistable traveling or pulsating fronts,
the front is unique (in the same sense as above: two fronts have the
same speed and have the same profile up to translation). 

We refer to Gardner \cite{Gardner_1982}, Kan-On \cite{Kan_on_1995},
Berestycki\textendash Hamel \cite{Berestycki_Ham_3} or Ding\textendash Hamel\textendash Zhao
\cite{Ding_Hamel_Zhao} for proofs of this type of result in slightly
different settings. 

Because the proof of such a result:
\begin{itemize}
\item would involve precise estimates of the exponential decay of the profiles
as $\xi\to\pm\infty$ that cannot be obtained briefly (in the scalar
case, see Hamel \cite{Hamel_2008}) and have no additional interest
in the forthcoming work, 
\item would be strongly analogous to the proofs of the preceding collection
of references,
\end{itemize}
we choose to leave this as an open question here for the sake of brevity.
We might address this question in a future sequel.

Still, it is useful to have this uniqueness in mind because it clearly
motivates our study of $\lim\limits _{k\to+\infty}c_{k}$.

\subsection{\textquotedblleft Unity is not strength\textquotedblright{} theorem
for periodic media\label{subsec:UINS_Theorem}}

In the forthcoming theorem, the parameters $d$, $\alpha$, $f_{1}$
and $f_{2}$ may vary (in some sense which is made precise), but immediately
after that they are fixed again (at least up to Section \ref{sec:Sign}).
\begin{thm}
{[}\textquotedblleft Unity is not strength\textquotedblright , periodic
case{]} Assume that there exists an open connected set $\mathfrak{P}$
of parameters:
\[
\left(d,\alpha,f_{1},f_{2}\right)\in\left(0,+\infty\right)^{2}\cap\mathcal{C}\left([0,+\infty),\mathcal{C}_{per}\left(\mathbb{R}\right)\right)^{2}
\]
 in which $\left(\mathcal{H}_{1}\right)$, $\left(\mathcal{H}_{2}\right)$,
$\left(\mathcal{H}_{3}\right)$ and $\left(\mathcal{H}_{exis}\right)$
are satisfied. 

The sequence $\left(\left(d,\alpha,f_{1},f_{2}\right)\in\mathfrak{P}\mapsto c_{k}\right)_{k>k^{\star}}$
converges pointwise as $k\to+\infty$ to some continuous function
$\left(d,\alpha,f_{1},f_{2}\right)\in\mathfrak{P}\mapsto c_{\infty}$.
If the function $\left(d,\alpha,f_{1},f_{2}\right)\in\mathfrak{P}\mapsto k^{\star}$
is locally bounded, then this convergence is in fact locally uniform
in $\mathfrak{P}$. 

Furthermore, for any $\left(d,\alpha,f_{1},f_{2}\right)\in\mathfrak{P}$,
there exist $\overline{r}>0$, $\underline{r}\in(0,\overline{r}]$
(both dependent on $\left(f_{1},f_{2}\right)$ only) and a non-empty
closed interval $\mathcal{R}^{0}\subset\left[\underline{r},\overline{r}\right]$
(dependent on $\left(d,f_{1},f_{2}\right)$ only) such that the sign
of $c_{\infty}$ satisfies the following properties.
\begin{enumerate}
\item $c_{\infty}>0$ if and only if $\frac{\alpha^{2}}{d}>\max\mathcal{R}^{0}$.
\item $c_{\infty}<0$ if and only if $\frac{\alpha^{2}}{d}<\min\mathcal{R}^{0}$.
\item If, for any $i\in\left\{ 1,2\right\} $, $f_{i}$ has the particular
form $\left(u,x\right)\mapsto\mu_{i}\left(x\right)\left(1-u\right)$,
then:
\begin{enumerate}
\item $c_{\infty}$ is null or has the sign of: 
\[
\alpha^{2}-d\frac{\|\mu_{2}\|_{L^{1}\left(C\right)}}{\|\mu_{1}\|_{L^{1}\left(C\right)}};
\]
\item $\left(\underline{r},\overline{r}\right)$ satisfies:
\[
\frac{\min\limits _{\overline{C}}\left(\mu_{2}\right)}{\max\limits _{\overline{C}}\left(\mu_{1}\right)}\leq\underline{r}\leq\overline{r}\leq\frac{\max\limits _{\overline{C}}\left(\mu_{2}\right)}{\min\limits _{\overline{C}}\left(\mu_{1}\right)}.
\]
\end{enumerate}
\end{enumerate}
The objects $\overline{r}$, $\underline{r}$ and $\mathcal{R}^{0}$
are respectively defined by formulas $\left(\mathfrak{F}_{\overline{r}}\right)$,
$\left(\mathfrak{F}_{\underline{r}}\right)$ and $\left(\mathfrak{F}_{\mathcal{R}^{0}}\right)$
(see page \pageref{formulas}). 
\end{thm}

\begin{rem*}
We emphasize the interest of $\underline{r}$ and $\overline{r}$,
which are upper and lower bounds for $\mathcal{R}^{0}$ which are
uniform with respect to $d$.

We will explain in Section \ref{sec:Sign} that if $\left(\mathcal{H}_{exis}\right)$
is derived from the existence result of the first author \cite{Girardin_2016},
then a set $\mathfrak{P}$ exists: the main assumption of our theorem
makes sense indeed. 
\end{rem*}
The strategy of the proof is as follows. 

We will begin with some compactness estimates uniform with respect
to $k$ so that a limiting speed and an associated limiting solution,
possibly non-unique at this point, can be extracted. This will require
a crucial distinction between two cases: limiting speed null or not. 

Regarding the first case, we will give some regularity properties
of the corresponding solution, that will be called a \textit{segregated
stationary equilibrium}. It is unclear whether the segregated stationary
equilibrium is unique but this is not surprising: the null speed case
is known to be quite degenerate (see for instance Ding\textendash Hamel\textendash Zhao
\cite{Ding_Hamel_Zhao}).

On the contrary, the second case will be fully characterized: the
corresponding solution, the \textit{segregated pulsating front}, is
actually unique (up to translation). Such a uniqueness result will
require several intermediary results and in particular a (possibly
not complete but already quite thorough) study of its intrinsic free
boundary problem. 

Subsequently, the uniqueness of the segregated pulsating front will
follow from a sliding argument which will also provide us with an
exclusion result: there exists a segregated stationary equilibrium
for a particular choice of parameters $\left(d,\alpha,f_{1},f_{2}\right)$
if and only if there does not exist a segregated pulsating front.
Thanks to this result, the uniqueness of the limiting speed will be
deduced even though the null case is still degenerate. 

We will then obtain a necessary and sufficient condition on $\left(d,\alpha,f_{1},f_{2}\right)$
for the existence of a segregated stationary equilibrium thanks to
its regularity at the interface (which is, in some sense, the counterpart
to the free boundary problem leading to the uniqueness of the segregated
pulsating front) and finally, thanks to a classical integration by
parts, obtain the sign of the speed provided it is already known to
be non-zero.

\subsection{A few more preliminaries\label{subsec:A-few-more-prelim}}

\subsubsection{Compact embeddings of Hölder spaces}
\begin{prop}
Let $\left(a,a'\right)\in\left(0,+\infty\right)^{2}$ and $n,n',\beta,\beta'$
such that $\left(a,a'\right)=\left(n+\beta,n'+\beta'\right)$, $n$
and $n'$ are non-negative integers and $\beta$ and $\beta'$ are
in $(0,1]$.

If $a\leq a'$, then the canonical embedding $i:\mathcal{C}^{n',\beta'}\left(C\right)\hookrightarrow\mathcal{C}^{n,\beta}\left(C\right)$
is continuous and compact. 
\end{prop}

It will be clear later on that this problem naturally involves uniform
bounds in $\mathcal{C}^{0,\nicefrac{1}{2}}$. Therefore, we fix once
and for all $\beta\in\left(0,\frac{1}{2}\right)$ and we will use
systematically the compact embeddings $\mathcal{C}^{n,\nicefrac{1}{2}}\hookrightarrow\mathcal{C}^{n,\beta}$,
meaning that uniform bounds in $\mathcal{C}^{n,\nicefrac{1}{2}}$
yield relative compactness in $\mathcal{C}^{n,\beta}$.

\subsubsection{Additional notations regarding the pulsating fronts}

Let $E=\left(\begin{matrix}1 & 1\\
1 & 1
\end{matrix}\right)$. For any $k>k^{\star}$, $\left(c_{k},\varphi_{1,k},\varphi_{2,k}\right)$
satisfies the following system:
\[
\left\{ \begin{matrix}-\mbox{div}\left(E\nabla\varphi_{1,k}\right)-c_{k}\partial_{\xi}\varphi_{1,k}=\varphi_{1,k}f_{1}\left[\varphi_{1,k}\right]-k\varphi_{1,k}\varphi_{2,k}\\
-d\mbox{div}\left(E\nabla\varphi_{2,k}\right)-c_{k}\partial_{\xi}\varphi_{2,k}=\varphi_{2,k}f_{2}\left[\varphi_{2,k}\right]-\alpha k\varphi_{1,k}\varphi_{2,k}.
\end{matrix}\right.\quad\left(\mathcal{PF}_{sys,k}\right)
\]
\begin{rem*}
Be aware that, since $\mbox{sp}E=\left\{ 0,2\right\} $, the differential
operator: 
\[
\mbox{div}\left(E\nabla\right)=\partial_{\xi\xi}+\partial_{xx}+2\partial_{\xi x}
\]
is only degenerate elliptic. This will trigger difficulties unknown
in the space-homogeneous case. Most regularity results will come from
the parabolic system $\left(\mathcal{P}\right)$ and we will need
to go back and forth a lot between the so-called \textquotedblleft parabolic
coordinates\textquotedblright{} $\left(t,x\right)$ and the so-called
\textquotedblleft traveling coordinates\textquotedblright{} $\left(\xi,x\right)$.
This will be possible if and only if the propagation speed is non-zero,
whence a necessary distinction of cases.
\end{rem*}
For any $k>k^{\star}$, let: 
\[
\psi_{d,k}=\alpha\varphi_{1,k}-d\varphi_{2,k},
\]
\[
\psi_{1,k}=\alpha\varphi_{1,k}-\varphi_{2,k},
\]
\[
v_{d,k}=\alpha u_{1,k}-du_{2,k},
\]
\[
v_{1,k}=\alpha u_{1,k}-u_{2,k}.
\]

A linear combination of the equations of $\left(\mathcal{PF}_{sys,k}\right)$
yields:
\[
-\mbox{div}\left(E\nabla\psi_{d,k}\right)-c_{k}\partial_{\xi}\psi_{1,k}=\alpha\varphi_{1,k}f_{1}\left[\varphi_{1,k}\right]-\varphi_{2,k}f_{2}\left[\varphi_{2,k}\right]\quad\left(\mathcal{PF}_{k}\right).
\]

$\left(\mathcal{PF}_{k}\right)$ does not depend explicitly on $k$.

$\left(u_{1,k},u_{2,k},v_{d,k},v_{1,k}\right)$ is isomorphic to $\left(\varphi_{1,k},\varphi_{2,k},\psi_{d,k},\psi_{1,k}\right)$
if and only if $c_{k}\neq0$. In parabolic coordinates, $\left(\mathcal{PF}_{k}\right)$
becomes:
\[
\partial_{t}v_{1,k}-\partial_{xx}v_{d,k}=\alpha u_{1,k}f_{1}\left[u_{1,k}\right]-u_{2,k}f_{2}\left[u_{2,k}\right].
\]

As $k\to+\infty$, the following function will naturally appear:

\[
\eta:\left(z,x\right)\mapsto f_{1}\left(\frac{z}{\alpha},x\right)z^{+}-\frac{1}{d}f_{2}\left(-\frac{z}{d},x\right)z^{-},
\]
 where $z^{+}=\max\left(z,0\right)$ and $z^{-}=-\min\left(z,0\right)$
so that $z=z^{+}-z^{-}$.

We will also denote $g_{i}$ the partial derivative of $\left(u,x\right)\mapsto uf_{i}\left(u,x\right)$
with respect to $u$:
\[
g_{i}:\left(u,x\right)\mapsto f_{i}\left(u,x\right)+u\partial_{1}f_{i}\left(u,x\right)\text{ for all }i\in\left\{ 1,2\right\} .
\]

\subsection{Comparison between the first and the second part\label{subsec:Comparison-first-second-parts}}

In addition to the new notations introduced in the preceding subsection
( $\left(\mathcal{PF}_{sys}\right)$, $\left(\mathcal{PF}\right)$,
``parabolic coordinates'', ``traveling coordinates'', $\psi_{d}$,
$\psi_{1}$, $v_{d}$, $v_{1}$), the following differences are pointed
out.
\begin{itemize}
\item In the first part \cite{Girardin_2016}, $f_{1}$ and $f_{2}$ were
only assumed to be Hölder-continuous with respect to $x$, whereas
here we need them to be at least continuously differentiable. Thanks
to this technical hypothesis, it is then possible to differentiate
with respect to $x$ the various equations and systems involved. In
particular, continuous pulsating front solutions of $\left(\mathcal{P}\right)$
are in fact in $\mathcal{C}_{loc}^{2}\left(\mathbb{R}^{2}\right)$.
This will similarly yield a stronger regularity at the limit. Nevertheless,
we think that Hölder-continuity might actually suffice to obtain most
of the forthcoming results.
\item The positive zero of $u\mapsto f_{i}\left(u,x\right)$ cannot depend
on $x$ anymore. Consequently, while, in the first part \cite{Girardin_2016},
the unique positive solution of $-z''=zf_{1}\left[z\right]$, $\tilde{u}_{1}$,
and the unique positive solution of $-dz''=zf_{2}\left[z\right]$,
$\tilde{u}_{2}$, were periodic functions of $x$, here they are the
constants $a_{1}$ and $a_{2}$. This restriction is standard in bistable
pulsating front problems (see for instance \cite{Ding_Hamel_Zhao,Ding_Hamel_Zhao2,Zlatos_2015})
and is especially related to the method generically used to determine
the sign of the speed of the pulsating fronts. Still, most of the
forthcoming pages is easily generalized (actually, many results need
no adaptation at all). We will highlight where this hypothesis is
truly needed and will give some indications regarding the non-constant
case. In the end, it should be clear why we conjecture that \textquotedblleft Unity
is not strength\textquotedblright{} holds true even in the non-constant
case.
\item A trade-off to these more restrictive assumptions is that here we
do not assume \textit{a priori} the high-frequency hypothesis:
\[
L<\pi\left(\frac{1}{\sqrt{M_{1}}}+\sqrt{\frac{d}{M_{2}}}\right).\quad\left(\mathcal{H}_{freq}\right)
\]
We merely assume existence of pulsating fronts, this hypothesis being
referred to as $\left(\mathcal{H}_{exis}\right)$. It was proved in
the first part that if $\left(\mathcal{H}_{freq}\right)$ is satisfied,
then so is $\left(\mathcal{H}_{exis}\right)$.
\end{itemize}

\section{Asymptotic behavior: the infinite competition limit}

\subsection{Existence of a limiting speed}

In order to prove that $\left(c_{k}\right)_{k>k^{\star}}$ has at
least one limit point, we recall an important result from the Fisher\textendash KPP
scalar case (see Berestycki\textendash Hamel\textendash Roques \cite{Berestycki_Ham_2}).
\begin{thm}
\label{thm:BHR2} For any $\delta\in\left\{ 1,d\right\} $ and $i\in\left\{ 1,2\right\} $,
there exists $c^{\star}\left[\delta,i\right]>0$ such that, for any
$s\in\mathbb{R}$, there exists in $\mathcal{C}^{2}\left(\mathbb{R}^{2}\right)$
a pulsating front solution of: 
\[
\partial_{t}z-\delta\partial_{xx}z=zf_{i}\left[z\right]
\]
 connecting $a_{i}$ to $0$ at speed $s$ if and only if $s\geq c^{\star}\left[\delta,i\right]$.
\end{thm}

\begin{lem}
\label{lem:bounds_speed_pf} Provided $k^{\star}$ is large enough,
for any $k>k^{\star}$ and any pulsating front solution of $\left(\mathcal{P}_{k}\right)$,
its speed $c$ satisfies:
\[
-c^{\star}\left[d,2\right]<c<c^{\star}\left[1,1\right].
\]

In particular, the family $\left(c_{k}\right)_{k>k^{\star}}$ is uniformly
bounded with respect to $k$.
\end{lem}

\begin{rem*}
Here, the assumption that $k$ is large enough might in fact be redundant
with the underlying assumption of bistability. Indeed, this proof
does not use any limiting behavior but only requires that: 
\[
k>\max\left\{ \frac{1}{a_{2}}\max_{\overline{C}}\left(f_{1}\left[0\right]\right),\frac{1}{\alpha a_{1}}\max_{\overline{C}}\left(f_{2}\left[0\right]\right)\right\} .
\]

In the space-homogeneous logistic case, this condition reduces to
$k>\max\left\{ 1,\alpha^{-1}\right\} $, that is precisely the necessary
and sufficient condition for the system to be bistable. In the space-periodic
case, according to the proof of \cite[Proposition 2.1]{Girardin_2016},
both $a_{i}$ are stable if the condition above is satisfied. Yet
an optimal threshold should involve periodic principal eigenvalues
instead of these maxima. Furthermore, the instability of any other
periodic steady state has only been established for (really) large
$k$ (see \cite[Theorem 1.2]{Girardin_2016}) and when $\left(\mathcal{H}_{freq}\right)$
holds true. Even for arbitrarily large $k$, it is unclear whether
stable coexistence periodic steady states might exist when $\left(\mathcal{H}_{freq}\right)$
does not hold.

We point out that the following proof provides us with an instance
of a detailed proof using the sliding method \cite{Berestycki_Ham_3}
that will be referred to later on.
\end{rem*}
\begin{proof}
Assume by contradiction that there exists $k>0$ such that there exists
a pulsating front solution $\left(z_{1},z_{2}\right)$ of $\left(\mathcal{P}_{k}\right)$
with a speed $c\notin\left(-c^{\star}\left[d,2\right],c^{\star}\left[1,1\right]\right)$
and a profile $\left(\varphi_{1},\varphi_{2}\right)$. For instance,
assume $c\geq c^{\star}\left[1,1\right]$ (the other case being obviously
symmetric), and let $\underline{c}=c^{\star}\left[1,1\right]\leq c$.
By virtue of Theorem \ref{thm:BHR2}, $\underline{c}>0$ and there
exists a pulsating front solution $z$ of : 
\[
\partial_{t}z-\partial_{xx}z=zf_{1}\left[z\right]
\]
 with speed $\underline{c}$ and profile $\varphi$.

Now we are in position to use the sliding method to compare $z$ and
$z_{1}$. This will finally lead to a contradiction. 

\textbf{Step 1: existence of a translation of the profile associated
with the higher speed such that it is locally below the other profile.}

Fix $\zeta\in\mathbb{R}$. Then let $\zeta_{1}\in\mathbb{R}$ such
that: 
\[
\max_{x\in\overline{C}}\varphi_{1}\left(\zeta_{1},x\right)<\min_{x\in\overline{C}}\varphi\left(\zeta,x\right).
\]

Let:
\[
\tau=\zeta-\zeta_{1},
\]
\[
\varphi_{1}^{\tau}:\left(\xi,x\right)\mapsto\varphi_{1}\left(\xi-\tau,x\right),
\]
\[
\Phi^{\tau}=\varphi-\varphi_{1}^{\tau},
\]
so that: 
\[
\min_{x\in\overline{C}}\Phi^{\tau}\left(\zeta,x\right)=\min_{x\in\overline{C}}\left(\varphi\left(\zeta,x\right)-\varphi_{1}\left(\zeta_{1},x\right)\right)>0.
\]

\textbf{Step 2: up to some extra term, this ordering is global on
the left.}

Let $\mathcal{U}=\left(-\infty,\zeta\right)\times\overline{C}$. Since
$\varphi\gg0$ in $\mathcal{U}$ and $\varphi_{1}^{\tau}\in L^{\infty}\left(\mathcal{U}\right)$,
there exists $\kappa>0$ such that: 
\[
\kappa\varphi-\varphi_{1}^{\tau}\geq0\mbox{ in }\mathcal{U}.
\]

Notice that, since $\Phi^{\tau}\left(\xi,x\right)\to0$ as $\xi\to\pm\infty$
(uniformly with respect to $x$), any such $\kappa$ is larger than
or equal to $1$. 

\textbf{Step 3: this extra term is actually unnecessary, thanks to
the maximum principle.}

Let:
\[
\kappa^{\star}=\inf\left\{ \kappa>1\ |\ \inf_{\mathcal{U}}\left(\kappa\varphi-\varphi_{1}^{\tau}\right)>0\right\} 
\]
 and let us prove that $\kappa^{\star}=1$. We assume by contradiction
that $\kappa^{\star}>1$ and we take a sequence $\left(\kappa_{n}\right)_{n\in\mathbb{N}}\in\left(1,\kappa^{\star}\right)^{\mathbb{N}}$
which converges to $\kappa^{\star}$ from below. 

There exists a sequence $\left(\left(\xi_{n},x_{n}\right)\right)\in\mathcal{U}^{\mathbb{N}}$
such that for any $n\in\mathbb{N}$, 
\[
\kappa_{n}\varphi\left(\xi_{n},x_{n}\right)<\varphi_{1}^{\tau}\left(\xi_{n},x_{n}\right).
\]

Since $\kappa_{n}>1$, the limits when $\xi\to-\infty$ prove that
$\left(\xi_{n}\right)$ is bounded from below, and since it is also
bounded from above by $\zeta$, we can extract a convergent subsequence
with limit $\xi^{\star}\in(-\infty,\zeta]$. Similarly, we can extract
a convergent subsequence of $\left(x_{n}\right)\in\overline{C}^{\mathbb{N}}$
with limit $x^{\star}\in\overline{C}$. By continuity, $\left(\kappa^{\star}\varphi-\varphi_{1}^{\tau}\right)\left(\xi^{\star},x^{\star}\right)=0$
and, necessarily, $\xi^{\star}<\zeta$. 

Back to parabolic variables, recall that $\underline{c}>0$ and let:
\[
t^{\star}=\frac{x^{\star}-\xi^{\star}}{\underline{c}},
\]
\[
\hat{z}_{i}^{\tau}:\left(t,x\right)\mapsto\varphi_{i}\left(x-\underline{c}t-\tau,x\right)\mbox{ for any }i\in\left\{ 1,2\right\} ,
\]
\[
v^{\star}=\kappa^{\star}z-\hat{z}_{1}^{\tau},
\]
\[
f:\left(t,x\right)\mapsto-\left(c-\underline{c}\right)\left(\partial_{\xi}\varphi_{1}^{\tau}\right)\left(x-\underline{c}t,x\right)
\]
\[
E=\left\{ \left(t,x\right)\in\mathbb{R}^{2}\ |\ x-\underline{c}t<\zeta\right\} .
\]

By virtue of $\left(\mathcal{H}_{3}\right)$ and $\kappa^{\star}>1$:
\[
\kappa^{\star}zf_{1}\left[z\right]>\kappa^{\star}zf_{1}\left[\kappa^{\star}z\right]\mbox{ in }E,
\]
 and moreover:
\[
\partial_{t}v^{\star}-\partial_{xx}v^{\star}=\kappa^{\star}zf_{1}\left[z\right]-\hat{z}_{1}^{\tau}f_{1}\left[\hat{z}_{1}^{\tau}\right]+k\hat{z}_{1}^{\tau}\hat{z}_{2}^{\tau}+f\mbox{ in }E,
\]
\[
f\geq0\mbox{ in }E.
\]

Now, from the Lipschitz-continuity of $f_{1}$ with respect to its
first variable, it follows that of $\left(u,x\right)\mapsto uf_{1}\left(u,x\right)$,
whence there exists $q\in L^{\infty}\left(E\right)$ such that:
\[
\partial_{t}v^{\star}-\partial_{xx}v^{\star}\geq qv^{\star}\mbox{ in }E.
\]

In the end, $v^{\star}$ is a non-negative super-solution which vanishes
at some interior point: by virtue of the parabolic strong minimum
principle, it is identically null in $\left((-\infty,t^{\star}]\times\mathbb{R}\right)\cap E$. 

But in such an unbounded set, it is always possible to construct an
element of $\left\{ \zeta\right\} \times\overline{C}$, which contradicts:
\[
\min_{x\in\overline{C}}\left(\kappa^{\star}\varphi-\varphi_{1}^{\tau}\right)\left(\zeta,x\right)>0.
\]

Therefore $\kappa^{\star}=1$, 
\[
\kappa^{\star}\varphi-\varphi_{1}^{\tau}=\Phi^{\tau}\geq0\mbox{ in }\mathcal{U}
\]
and then by periodicity and, once more, by virtue of the parabolic
strong minimum principle:
\[
\Phi^{\tau}\gg0\mbox{ in }\left(-\infty,\zeta\right)\times\mathbb{R}.
\]

\textbf{Step 4: up to some (possibly different) extra term, this ordering
is global on the right.}

Near $+\infty$ (in $\left(\zeta,+\infty\right)\times\mathbb{R}$),
on the contrary, multiplying $\varphi$ by some $\kappa\gg1$ is not
going to yield a clear ordering anymore since we are interested in
the behavior as $\varphi\sim0$ and $\varphi_{1}\sim0$ (and replacing
$\varphi$ and $\varphi_{1}^{\tau}$ by respectively $a_{1}-\varphi$
and $a_{1}-\varphi_{1}^{\tau}$ will not suffice since the monostability
has no underlying symmetry). 

But it is natural, for instance, to replace this multiplication by
the addition of some $\varepsilon\geq0$ and to prove in the next
step that $\varepsilon^{\star}=0$. This is actually what was done
originally by Berestycki\textendash Hamel \cite{Berestycki_Ham_3}.

\textbf{Step 5: this (possibly different) extra term is also unnecessary.}

We define $\varepsilon^{\star}$ as the following quantity: 
\[
\varepsilon^{\star}=\inf\left\{ \varepsilon>0\ |\ \inf_{\left(\zeta,+\infty\right)\times\overline{C}}\left(\varphi-\varphi_{1}^{\tau}+\varepsilon\right)>0\right\} .
\]

We assume by contradiction that $\varepsilon^{\star}>0$ and this
yields as before a contact point $\left(\xi^{\star},x^{\star}\right)\in\left(\zeta,+\infty\right)\times\overline{C}$. 

Now the main difficulty is that $u\mapsto uf_{1}\left[u\right]$ is
increasing near $0$, so that we really cannot hope to have:
\[
zf_{1}\left[z\right]\geq\left(z+\varepsilon\right)f_{1}\left[z+\varepsilon\right].
\]

Still, it is possible to assume without loss of generality that, during
the construction of $\tau$, $\zeta_{1}$ has also been chosen so
that:
\[
\frac{a_{2}}{2}\leq\varphi_{2}\left(\xi,x\right)\leq a_{2}\mbox{ for any }\left(\xi,x\right)\in[\zeta_{1},+\infty)\times\overline{C}.
\]

It follows that:
\[
\varphi_{1}^{\tau}\left(f_{1}\left[\varphi_{1}^{\tau}\right]-k\varphi_{2}^{\tau}\right)\leq\varphi_{1}^{\tau}\left(f_{1}\left[\varphi_{1}^{\tau}\right]-k\frac{a_{2}}{2}\right)\mbox{ in }[\zeta,+\infty)\times\overline{C}.
\]

By virtue of the hypotheses $\left(\mathcal{H}_{1}\right)$, $\left(\mathcal{H}_{2}\right)$
and $\left(\mathcal{H}_{3}\right)$, provided $k^{\star}$ is large
enough, for any $K>k^{\star}$, the following non-linearity:
\[
u\mapsto u\left(f_{1}\left[u\right]-K\frac{a_{2}}{2}\right)
\]
 is decreasing in a neighborhood of $0$ (in fact, it is decreasing
in $[0,+\infty)$). Then, in addition to this monotonicity, it suffices
to use:
\[
\varphi f_{1}\left[\varphi\right]\geq\varphi f_{1}\left[\varphi\right]-k\frac{a_{2}}{2}\varphi
\]
and the Lipschitz-continuity of $f_{1}$ to conclude this step.

\textbf{Step 6: thanks to the maximum principle again, the speeds
are equal and the profiles are equal up to some translation.}

Thus in fact:
\[
\Phi^{\tau}\gg0\text{ in }\mathbb{R}^{2}.
\]

Now, let:
\[
\tau^{\star}=\sup\left\{ \tau\in\mathbb{R}\ |\ \Phi^{\tau}\geq0\mbox{ in }\mathbb{R}^{2}\right\} .
\]

The limits as $\xi\to\pm\infty$ of $\varphi$ and $\varphi_{1}$
ensure that $\tau^{\star}<+\infty$. By continuity, 
\[
\Phi^{\tau^{\star}}\geq0.
\]

Let us verify quickly that, by virtue of the maximum principle, either
$\Phi^{\tau^{\star}}=0$ and $c=\underline{c}$, either $\Phi^{\tau^{\star}}\gg0$.
For instance, assume that $\left(\Phi^{\tau^{\star}}\right)^{-1}\left(\left\{ 0\right\} \right)$
is non-empty, so that $\Phi^{\tau^{\star}}\gg0$ does not hold. Then
there exists $\left(\xi^{\star},x^{\star}\right)\in\mathbb{R}^{2}$
such that $\Phi^{\tau^{\star}}\left(\xi^{\star},x^{\star}\right)=0$.
Once more, we introduce:

\[
t^{\star}=\frac{x^{\star}-\xi^{\star}}{\underline{c}},
\]
\[
v^{\tau^{\star}}\left(t,x\right)=\Phi^{\tau^{\star}}\left(x-\underline{c}t,x\right),
\]
 and using the parabolic linear inequality satisfied by $v^{\tau^{\star}}$,
we verify that $v^{\tau^{\star}}$ is a non-negative super-solution
which vanishes at $\left(t^{\star},x^{\star}\right)$. Then, by the
strong parabolic maximum principle and periodicity with respect to
$x$ of $\Phi^{\tau^{\star}}$, it is actually deduced that $\Phi^{\tau^{\star}}=0$,
which in turn implies (reinserting $v^{\tau^{\star}}=0$ into the
original non-linear equation satisfied by $v^{\tau^{\star}}$ and
considering the function $f$ which has been defined earlier) that
$c=\underline{c}$.

Finally, assume by contradiction that $\Phi^{\tau^{\star}}\gg\left(0,0\right)$,
i.e. assume that for any $B>0$,
\[
\min_{\left[-B,B\right]\times\overline{C}}\Phi^{\tau^{\star}}>\left(0,0\right).
\]

Fix $B>0$. By continuity, there exists $\epsilon>0$ such that:
\[
\min_{\left[-B,B\right]\times\overline{C}}\Phi^{\tau}>\left(0,0\right)\mbox{ for any }\tau\in[\tau^{\star},\tau^{\star}+\epsilon).
\]
We can now repeat Steps 2, 3, 4, 5 to show that, for any such $\tau$:
\[
\Phi^{\tau}\gg0\mbox{ in }\left(\mathbb{R}\backslash\left(-B,B\right)\right)\times\mathbb{R}.
\]

The maximality of $\tau^{\star}$ being contradicted, this ends this
step.

\textbf{Step 7: the contradiction.}

If $c=\underline{c}$ and $z=z_{1}$, then thanks to the equations
satisfied by $z$ and $z_{1}$, $z_{2}=0$ in $\mathbb{R}^{2}$. This
contradicts the limit of $\varphi_{2}$ as $\xi\to+\infty$.
\end{proof}
\begin{cor}
$\left(c_{k}\right)_{k>k^{\star}}$ has a limit point $c_{\infty}\in\left[-c^{\star}\left[d,2\right],c^{\star}\left[1,1\right]\right]$.
\end{cor}

\begin{rem*}
Similarly, we do expect that $c_{\infty}\notin\left\{ -c^{\star}\left[d,2\right],c^{\star}\left[1,1\right]\right\} $
but will not address this question for the sake of brevity. 
\end{rem*}

\subsection{Existence of a limiting density provided the speed converges}

In this subsection, we fix a sequence $\left(c_{k}\right)_{k>k^{\star}}$
such that it converges to $c_{\infty}$. 

Then we prove the relative compactness of the associated sequence
of pulsating front solutions $\left(\left(u_{1,k},u_{2,k}\right)\right)_{k>k^{\star}}$,
which will follow from classical parabolic estimates similar to those
used by Dancer and his collaborators (see for instance \cite{Dancer_Hilhors})
supplemented by some estimates specific to the pulsating front setting.
This supplement will lead indeed to a stronger compactness result
than the one presented in the aforementioned work. 

If $c_{\infty}\neq0$, we will see that $\left(\left(u_{1,k},u_{2,k}\right)\right)_{k>k^{\star}}$
is relatively compact if and only if $\left(\left(\varphi_{1,k},\varphi_{2,k}\right)\right)_{k>k^{\star}}$
is relatively compact. Moreover, we will show that the compactness
result can be improved further thanks to additional pulsating front
estimates.

\subsubsection{Normalization}

Before going any further, we point out that, at this point, for any
$k>k^{\star}$, $\left(\varphi_{1},\varphi_{2}\right)$ is fixed completely
arbitrarily among the one-dimensional family of translated profiles.
By monotonicity of the profiles with respect to $\xi$, this choice
can in fact be normalized. In the space-homogeneous problem \cite{Girardin_Nadin_2015},
the normalization was used to guarantee that the extracted limit point
had no null component. It should be clear that this part of the proof
will be strongly analogous. Therefore we choose now normalizations
reminiscent to the space-homogeneous ones.
\begin{itemize}
\item On one hand, if $c_{\infty}\leq0$, we fix without loss of generality
for any $k>k^{\star}$ the normalization: 
\[
0=\inf\left\{ \xi\in\mathbb{R}\ |\ \exists x\in\overline{C}\quad\varphi_{1,k}\left(\xi,x\right)<\frac{a_{1}}{2}\right\} .
\]
\item On the other hand, if $c_{\infty}>0$, we fix without loss of generality
for any $k>k^{\star}$ the normalization:
\[
0=\sup\left\{ \xi\in\mathbb{R}\ |\ \exists x\in\overline{C}\quad\varphi_{2,k}\left(\xi,x\right)<\frac{a_{2}}{2}\right\} .
\]
\end{itemize}
Remark also that $\left(\left(u_{1,k},u_{2,k}\right)\right)_{k>k^{\star}}$
is normalized (in the sense that its value at some arbitrary initial
time is entirely prescribed) if and only if $\left(\left(\varphi_{1,k},\varphi_{2,k}\right)\right)_{k>k^{\star}}$
is normalized.

\subsubsection{Compactness results}
\begin{prop}
\label{prop:compactness} The following collection of properties holds
independently of the sign of $c_{\infty}$.
\begin{enumerate}
\item {[}Segregation{]} $\left(\varphi_{1,k}\varphi_{2,k}\right)_{k>k^{\star}}$
converges to $0$ in $L_{loc}^{1}\left(\mathbb{R}\times C\right)$. 
\item {[}Persistence{]} $\left(0,0\right)$ is not a limit point of $\left(\left(\varphi_{1,k},\varphi_{2,k}\right)\right)_{k>k^{\star}}$
in $L_{loc}^{1}\left(\mathbb{R}^{2},\mathbb{R}^{2}\right)$.
\item {[}Uniform bound in the diagonal direction{]} For any $n\in\mathbb{N}$,
$\left(\left(\partial_{x}+\partial_{\xi}\right)\left(\varphi_{1,k},\varphi_{2,k}\right)\right)_{k>k^{\star}}$
is uniformly bounded with respect to $k$ in $L^{2}\left(\left(-n,n\right)\times C,\mathbb{R}^{2}\right)$.
\item {[}Uniform bound in the $\xi$ direction{]} For any $k>k^{\star}$
and any $x\in\overline{C}$, 
\[
\int_{\mathbb{R}}\partial_{\xi}\varphi_{1,k}\left(\zeta,x\right)\mbox{d}\zeta=-\int_{\mathbb{R}}\left|\partial_{\xi}\varphi_{1,k}\left(\zeta,x\right)\right|\mbox{d}\zeta=-a_{1}
\]
 and 
\[
\int_{\mathbb{R}}\partial_{\xi}\varphi_{2,k}\left(\zeta,x\right)\mbox{d}\zeta=\int_{\mathbb{R}}\left|\partial_{\xi}\varphi_{2,k}\left(\zeta,x\right)\right|\mbox{d}\zeta=a_{2}.
\]
\item {[}Uniform bound in the $x$ direction{]} For any $T>0$, $\left(\left(u_{1,k},u_{2,k}\right)\right)_{k>k^{\star}}$
is uniformly bounded with respect to $k$ in $L^{2}\left(\left(-T,T\right),H^{1}\left(C,\mathbb{R}^{2}\right)\right)$.
\item {[}Uniform bound in the $t$ direction{]} For any $T>0$, $\left(\partial_{t}v_{1,k}\right)_{k>k^{\star}}$
is uniformly bounded with respect to $k$ in $L^{2}\left(\left(-T,T\right),\left(H^{1}\left(C\right)\right)'\right)$. 
\item {[}Compactness in traveling coordinates{]} $\left(\left(\varphi_{1,k},\varphi_{2,k}\right)\right)_{k>k^{\star}}$
is relatively compact in the topology of $L_{loc}^{1}\left(\mathbb{R}^{2},\mathbb{R}^{2}\right)$. 
\item {[}Compactness in parabolic coordinates{]} There exists: 
\[
\left(u_{1,\infty},u_{2,\infty}\right)\in\left(L^{\infty}\left(\mathbb{R}^{2}\right)\cap L^{2}\left(\left(-T,T\right),H^{1}\left(\left(0,L\right)\right)\right)\right)^{2}
\]
 such that:
\begin{enumerate}
\item $\partial_{t}v_{1,\infty}\in L^{2}\left(\left(-T,T\right),\left(H^{1}\left(\left(0,L\right)\right)\right)'\right)$;
\item $\left(u_{1,\infty},u_{2,\infty}\right)$ is a limit point of $\left(\left(u_{1,k},u_{2,k}\right)\right)_{k>k^{\star}}$
in the topology of $L_{loc}^{1}\left(\mathbb{R}^{2},\mathbb{R}^{2}\right)$;
\item $u_{1,\infty}$ and $u_{2,\infty}$ are in $\mathcal{C}_{loc}^{0,\beta}\left(\mathbb{R}^{2}\right)$;
\item $u_{1,\infty}=\alpha^{-1}v_{d,\infty}^{+}=\alpha^{-1}v_{1,\infty}^{+}$
and $u_{2,\infty}=d^{-1}v_{d,\infty}^{-}=v_{1,\infty}^{-}$.
\end{enumerate}
\end{enumerate}
\end{prop}

\begin{proof}
The segregation property comes directly from an integration of, say,
the first equation of $\left(\mathcal{PF}_{sys,k}\right)$ over some
$\left(-n,n\right)\times C$. The persistence of at least one component
is a consequence of the choice of normalization: for instance, if
$c_{\infty}\leq0$, necessarily $\left(\varphi_{1,k}\right)_{k>k^{\star}}$
does not vanish.

To get the uniform bound in the diagonal direction, we introduce a
cut-off function. For any $n\in\mathbb{N}$, there exists a non-negative
non-zero function $\chi\in\mathcal{D}\left(\mathbb{R}^{2}\right)$
such that, for any $x\in\overline{C}$, $\chi\left(\xi,x\right)=0$
if $\xi\notin\left[-n-1,n+1\right]$ and $\chi\left(\xi,x\right)=1$
if $\xi\in\left[-n,n\right]$. 

Let $k>k^{\star}$. Multiplying the first equation of $\left(\mathcal{PF}_{sys}\right)$
by $\varphi_{1,k}\chi$ and integrating by parts in $\mathbb{R}\times C$,
we obtain:
\[
\int\left(\partial_{\xi}\varphi_{1,k}\right)^{2}\chi-\frac{1}{2}\int\varphi_{1,k}^{2}\partial_{\xi\xi}\chi+\int\chi\left(\partial_{x}\varphi_{1,k}\right)^{2}+2\int\chi\partial_{\xi}\varphi_{1,k}\partial_{x}\varphi_{1,k}\leq\int M_{1}\varphi_{1,k}^{2}\chi-c\int\frac{\varphi_{1,k}^{2}}{2}\partial_{\xi}\chi.
\]

(The integrals being implicitly over $\mathbb{R}\times C$.)

Using $\chi\geq\mathbf{1}_{\left[-n,n\right]}$, the $k$-uniform
$L^{\infty}$-bound for $\left(\varphi_{1,k}\right)_{k>k^{\star}}$
and: 
\[
\left|c\right|\leq\max\left\{ c^{\star}\left[d,2\right],c^{\star}\left[1,1\right]\right\} ,
\]
we deduce the existence of a constant $R_{n}$ independent on $k$
such that:
\[
\int_{\left[-n,n\right]\times C}\left|\partial_{\xi}\varphi_{1,k}+\partial_{x}\varphi_{1,k}\right|^{2}\leq R_{n}.
\]

The same proof holds for $\varphi_{2,k}$. Finally, the same computation
in parabolic coordinates gives immediately the uniform bound in the
$x$ direction. 

The uniform bound in the $\xi$ direction is a straightforward result.
Provided the uniform bound in the $x$ direction, the uniform bound
in the $t$ direction comes from an integration over $\left(0,T\right)\times C$
of $\left(\mathcal{PF}\right)$ multiplied by some test function in
$L^{2}\left(\left(0,T\right),H^{1}\left(C\right)\right)$. 

The relative compactness in both systems of coordinates follows from
the embedding: 
\[
L_{loc}^{2}\left(\mathbb{R}^{2}\right)\hookrightarrow L_{loc}^{1}\left(\mathbb{R}^{2}\right)
\]
 and the compact embedding: 
\[
W_{loc}^{1,1}\left(\mathbb{R}^{2},\mathbb{R}^{2}\right)\hookrightarrow L_{loc}^{1}\left(\mathbb{R}^{2}\right).
\]

To obtain the continuity of $u_{1,\infty}$ and $u_{2,\infty}$, we
consider a convergent subsequence. Since the convergence occurs a.e.
up to extraction, the limit point is actually in $L^{\infty}\left(\mathbb{R}^{2},\mathbb{R}^{2}\right)$,
whence: 
\[
v_{1,\infty}\in L^{\infty}\left(\mathbb{R}\times\left(0,L\right)\right)\cap L^{2}\left(\left(-T,T\right),H^{1}\left(\left(0,L\right)\right)\right),
\]
\[
\partial_{t}v_{1,\infty}\in L^{2}\left(\left(-T,T\right),\left(H^{1}\left(\left(0,L\right)\right)\right)'\right).
\]
It follows from a standard regularity result that $v_{1,\infty}\in\mathcal{C}\left(\left[-T,T\right],L^{2}\left(\left(0,L\right)\right)\right)$
(see for instance Evans \cite[5.9.2]{Evans_book}).

Then, we pass the parabolic version of $\left(\mathcal{PF}\right)$
to the limit in $\mathcal{D}'\left(\mathbb{R}\right)$ and we can
apply DiBenedetto\textquoteright s theory \cite{DiBenedetto_19}:
$v_{1,\infty}$ is a locally bounded weak solution of the following
parabolic equation: 
\[
\partial_{t}z-\partial_{x}\left(\left(\mathbf{1}_{z>0}+d\mathbf{1}_{z<0}\right)\partial_{x}z\right)=f_{1}\left[\frac{z}{\alpha}\right]z^{+}-f_{2}\left[-z\right]z^{-}.
\]

In a large class of degenerate parabolic equations which contains
in particular this equation, locally bounded weak solutions are, for
any $\delta\in\left(0,1\right)$, spatially $\mathcal{C}_{loc}^{0,\delta}$
and temporally $\mathcal{C}_{loc}^{0,\nicefrac{\delta}{2}}$, whence
\textit{a fortiori} $v_{1,\infty}\in\mathcal{C}_{loc}^{0,\beta}\left(\mathbb{R}^{2}\right)$
(with $\delta=2\beta\in\left(0,1\right)$).

Finally, by virtue of the segregation property: 
\[
u_{1,\infty}=\alpha^{-1}v_{1,\infty}^{+}\mbox{ a.e.},
\]
\[
u_{2,\infty}=v_{1,\infty}^{-}\mbox{ a.e.}.
\]

From this, it follows that $v_{1,\infty}$ is in $\mathcal{C}_{loc}^{0,\beta}\left(\mathbb{R}^{2}\right)$
if and only if $u_{1,\infty}$ and $u_{2,\infty}$ are themselves
in $\mathcal{C}_{loc}^{0,\beta}\left(\mathbb{R}^{2}\right)$, whence
\[
\left(u_{1,\infty},u_{2,\infty}\right)\in\mathcal{C}_{loc}^{0,\beta}\left(\mathbb{R}^{2},\mathbb{R}^{2}\right).
\]
\end{proof}
\begin{rem*}
At this point, we do not know if the limit points in parabolic coordinates
and in traveling coordinates are related. Yet, when $c_{\infty}\neq0$,
we can improve the preceding results and relate the limit points indeed.
\end{rem*}
\begin{prop}
\label{prop:improved_compactness} Assume $c_{\infty}\neq0$. The
following additional collection of properties holds.
\begin{enumerate}
\item {[}Improved uniform bound in the $\xi$ direction{]} Provided $k^{\star}$
is large enough, $\left(\partial_{\xi}\left(\varphi_{1,k},\varphi_{2,k}\right)\right)_{k>k^{\star}}$
is uniformly bounded with respect to $k$ in $L^{2}\left(\mathbb{R}\times C,\mathbb{R}^{2}\right)$. 
\item {[}Improved compactness{]} There exists $\left(\varphi_{1,seg},\varphi_{2,seg}\right)\in L^{\infty}\left(\mathbb{R}^{2},\mathbb{R}^{2}\right)\cap H_{loc}^{1}\left(\mathbb{R}^{2},\mathbb{R}^{2}\right)$
such that, up to extraction:
\begin{enumerate}
\item $\left(\left(\varphi_{1,k},\varphi_{2,k}\right)\right)_{k>k^{\star}}$
converges to $\left(\varphi_{1,seg},\varphi_{2,seg}\right)$ strongly
in $L_{loc}^{2}\left(\mathbb{R}^{2},\mathbb{R}^{2}\right)$ and a.e.;
\item $\left(\left(\nabla\varphi_{1,k},\nabla\varphi_{2,k}\right)\right)_{k>k^{\star}}$
converges to $\left(\nabla\varphi_{1,seg},\nabla\varphi_{2,seg}\right)$
weakly in $L_{loc}^{2}\left(\mathbb{R}^{2},\mathbb{R}^{4}\right)$;
\item $\left(\left(u_{1,k},u_{2,k}\right)\right)_{k>k^{\star}}$ converges
to: 
\[
\left(u_{1,seg},u_{2,seg}\right):\left(t,x\right)\mapsto\left(\varphi_{1,seg},\varphi_{2,seg}\right)\left(x-c_{\infty}t,x\right)
\]
 strongly in $L_{loc}^{2}\left(\mathbb{R}^{2},\mathbb{R}^{2}\right)$,
a.e., and $\left(\left(\nabla u_{1,k},\nabla u_{2,k}\right)\right)_{k>k^{\star}}$
converges weakly in $L_{loc}^{2}\left(\mathbb{R}^{2},\mathbb{R}^{4}\right)$. 
\end{enumerate}
\end{enumerate}
\end{prop}

\begin{proof}
Since $c_{\infty}\neq0$, we assume without loss of generality that
$k^{\star}$ is sufficiently large to ensure that $c_{k}\neq0$ for
any $k>k^{\star}$. 

We start by showing that the uniform boundedness in $L^{2}\left(\mathbb{R}\times C\right)$
of $\left(\partial_{\xi}\varphi_{1,k}\right)_{k>k^{\star}}$ is equivalent
to that of $\left(\partial_{\xi}\varphi_{2,k}\right)_{k>k^{\star}}$
and to that of $\left(\partial_{\xi}\psi_{d,k}\right)_{k>k^{\star}}$. 
\begin{itemize}
\item First step of the equivalence: assume that $\left(\|\partial_{\xi}\varphi_{1,k}\|_{L^{2}\left(\mathbb{R}\times C\right)}\right)_{k>k^{\star}}$
is uniformly bounded. Let $k>k^{\star}$. Multiply $\left(\mathcal{PF}_{k}\right)$
by $\partial_{\xi}\psi_{d,k}$, remark that: 
\[
\partial_{\xi}\psi_{1,k}=\frac{1}{d}\left(\alpha\left(d-1\right)\partial_{\xi}\varphi_{1,k}+\partial_{\xi}\psi_{d,k}\right)
\]
 and integrate by parts over $\left(-n,n\right)\times C$ with some
$n\in\mathbb{N}$. By classical parabolic estimates, the terms involving
$E$ vanish as $n\to+\infty$. By change of variable, for any $i\in\left\{ 1,2\right\} $,
\[
\int_{C}\int_{-n}^{n}\varphi_{i,k}f_{i}\left[\varphi_{i,k}\right]\partial_{\xi}\varphi_{i,k}=\int_{C}\int_{\varphi_{i}\left(-n,x\right)}^{\varphi_{i}\left(+n,x\right)}zf_{i}\left(z,x\right)\mbox{d}z\mbox{d}x,
\]
whence as $n\to+\infty$:
\[
\int_{C}\int_{-n}^{n}\varphi_{1,k}f_{1}\left[\varphi_{1,k}\right]\partial_{\xi}\varphi_{1,k}\to-\int_{C}\int_{0}^{a_{1}}zf_{1}\left(z,x\right)\mbox{d}z\mbox{d}x,
\]
\[
\int_{C}\int_{-n}^{n}\varphi_{2,k}f_{2}\left[\varphi_{2,k}\right]\partial_{\xi}\varphi_{2,k}\to\int_{C}\int_{0}^{a_{2}}zf_{1}\left(z,x\right)\mbox{d}z\mbox{d}x.
\]
It follows that:
\begin{eqnarray*}
\left(-\frac{c_{k}}{d}\right)\int_{\mathbb{R}\times C}\left(\alpha\left(d-1\right)\partial_{\xi}\varphi_{1,k}+\partial_{\xi}\psi_{d,k}\right)\partial_{\xi}\psi_{d,k} & = & -\alpha\int_{C}\int_{0}^{a_{1}}zf_{1}\left(z,x\right)\mbox{d}z\mbox{d}x\\
 &  & +\alpha\int_{\mathbb{R}\times C}\varphi_{1,k}f_{1}\left[\varphi_{1,k}\right]\left(-d\partial_{\xi}\varphi_{2,k}\right)\\
 &  & +\int_{\mathbb{R}\times C}\left(-\varphi_{2,k}\right)f_{2}\left[\varphi_{2,k}\right]\left(\alpha\partial_{\xi}\varphi_{1,k}\right)\\
 &  & +d\int_{C}\int_{0}^{a_{2}}zf_{1}\left(z,x\right)\mbox{d}z\mbox{d}x.
\end{eqnarray*}
Dividing by $-\frac{c_{k}}{d}$ which stays away from $0$, the result
reduces to:
\begin{eqnarray*}
\alpha\left(d-1\right)\int\partial_{\xi}\varphi_{1,k}\partial_{\xi}\psi_{d,k}+\int\left|\partial_{\xi}\psi_{d}\right|^{2} & = & \frac{\alpha d}{c_{k}}\int d\varphi_{1,k}f_{1}\left[\varphi_{1,k}\right]\partial_{\xi}\varphi_{2,k}\\
 &  & +\frac{\alpha d}{c_{k}}\int\varphi_{2,k}f_{2}\left[\varphi_{2,k}\right]\partial_{\xi}\varphi_{1,k}\\
 &  & +\frac{\alpha d}{c_{k}}\int_{C}\int_{0}^{a_{1}}zf_{1}\left(z,x\right)\mbox{d}z\mbox{d}x\\
 &  & -\frac{d^{2}}{c_{k}}\int_{C}\int_{0}^{a_{2}}zf_{1}\left(z,x\right)\mbox{d}z\mbox{d}x.
\end{eqnarray*}
Using the boundedness in $L^{\infty}$ of $\varphi_{i,k}f_{i}\left[\varphi_{i,k}\right]$
and the relations: 
\[
\int\left|\partial_{\xi}\varphi_{1,k}\right|=La_{1},
\]
\[
\int\left|\partial_{\xi}\varphi_{2,k}\right|=La_{2},
\]
we obtain that the right-hand side is uniformly bounded. Since $\partial_{\xi}\varphi_{1,k}$
and $\partial_{\xi}\psi_{d,k}$ are both non-positive non-zero, if
$d\geq1$, the uniform boundedness of $\left(\int\left|\partial_{\xi}\psi_{d,k}\right|^{2}\right)_{k>k^{\star}}$
follows. Otherwise, there exists $R>0$ such that:
\begin{eqnarray*}
\int\left|\partial_{\xi}\psi_{d,k}\right|^{2} & \leq & R+\left|\alpha\left(d-1\right)\right|\int\partial_{\xi}\varphi_{1,k}\partial_{\xi}\psi_{d,k}\\
 & \leq & R+\left|\alpha\left(d-1\right)\right|\left(\int\left|\partial_{\xi}\varphi_{1,k}\right|^{2}\right)^{\nicefrac{1}{2}}\left(\int\left|\partial_{\xi}\psi_{d,k}\right|^{2}\right)^{\nicefrac{1}{2}}.
\end{eqnarray*}
This shows that $\left(\int\left|\partial_{\xi}\psi_{d,k}\right|^{2}\right)^{\nicefrac{1}{2}}$,
which is positive, is also smaller than or equal to the largest zero
of the following polynomial: 
\[
X^{2}-\left|\alpha\left(d-1\right)\right|\|\partial_{\xi}\varphi_{1,k}\|_{L^{2}\left(\mathbb{R}\times C\right)}X-R
\]
 (which is itself positive and uniformly bounded).
\item Second step of the equivalence: assume that $\left(\|\partial_{\xi}\varphi_{2,k}\|_{L^{2}\left(\mathbb{R}\times C\right)}\right)_{k>k^{\star}}$
is uniformly bounded. A slight adaptation of the first step (using
$\partial_{\xi}\psi_{1}=\partial_{\xi}\psi_{d}+\left(d-1\right)\partial_{\xi}\varphi_{2}$)
shows that the third statement is implied indeed.
\item Third step of the equivalence: assume that $\left(\|\partial_{\xi}\psi_{d,k}\|_{L^{2}\left(\mathbb{R}\times C\right)}\right)_{k>k^{\star}}$
is uniformly bounded. Since, for any $k>k^{\star}$: 
\[
\|\partial_{\xi}\psi_{d}\|_{L^{2}}^{2}=\alpha^{2}\|\partial_{\xi}\varphi_{1}\|_{L^{2}}^{2}+d^{2}\|\partial_{\xi}\varphi_{2}\|_{L^{2}}^{2}-2\alpha d\left\langle \partial_{\xi}\varphi_{1},\partial_{\xi}\varphi_{2}\right\rangle _{L^{2}},
\]
 with a positive third term, the first and the second statements are
immediately implied. 
\end{itemize}
Now that the equivalence is established, we simply show that if $c_{\infty}>0$,
$\left(\|\partial_{\xi}\varphi_{1,k}\|_{L^{2}\left(\mathbb{R}\times C\right)}\right)_{k>k^{\star}}$
is uniformly bounded, and conversely if $c_{\infty}<0$, $\left(\|\partial_{\xi}\varphi_{2,k}\|_{L^{2}\left(\mathbb{R}\times C\right)}\right)_{k>k^{\star}}$
is uniformly bounded. Multiplying the first equation of $\left(\mathcal{PF}_{sys}\right)$
by $\partial_{\xi}\varphi_{1}$, integrating over $\mathbb{R}\times C$,
and using the sign of $\partial_{\xi}\varphi_{1}$ and classical parabolic
estimates at $\pm\infty$, the result reduces to:
\begin{eqnarray*}
c\int_{\mathbb{R}\times\left(0,L\right)}\left|\partial_{\xi}\varphi_{1}\right|^{2} & = & k\int_{\mathbb{R}\times\left(0,L\right)}\varphi_{1}\varphi_{2}\partial_{\xi}\varphi_{1}+\int_{0}^{L}\int_{0}^{a_{1}}zf_{1}\left(z,x\right)\mbox{d}z\mbox{d}x\\
 & \leq & \int_{0}^{L}\int_{0}^{a_{1}}zf_{1}\left(z,x\right)\mbox{d}z\mbox{d}x.
\end{eqnarray*}
Similarly, we obtain:
\begin{eqnarray*}
c\int_{\mathbb{R}\times\left(0,L\right)}\left|\partial_{\xi}\varphi_{2}\right|^{2} & = & \alpha k\int_{\mathbb{R}\times\left(0,L\right)}\varphi_{1}\varphi_{2}\partial_{\xi}\varphi_{2}-\int_{0}^{L}\int_{0}^{a_{2}}zf_{2}\left(z,x\right)\mbox{d}z\mbox{d}x\\
 & \geq & -\int_{0}^{L}\int_{0}^{a_{2}}zf_{2}\left(z,x\right)\mbox{d}z\mbox{d}x.
\end{eqnarray*}
The improved uniform bound in the $\xi$ direction immediately follows.

The improved relative compactness of $\left(\left(\varphi_{1,k},\varphi_{2,k}\right)\right)_{k>k^{\star}}$
is a straightforward consequence of the previous lemmas, of Sobolev\textquoteright s
embeddings and of Banach\textendash Alaoglu\textquoteright s theorem.
For the relative compactness of $\left(\left(u_{1,k},u_{2,k}\right)\right)_{k>k^{\star}}$,
let $\left[s\right]:\left(t,x\right)\mapsto\left(x-st,x\right)$,
so that for any $k>k^{\star}$ $\left(u_{1},u_{2}\right)=\left(\varphi_{1},\varphi_{2}\right)\circ\left[c\right]$.
For any $i\in\left\{ 1,2\right\} $: 
\[
\|u_{i}-u_{i,seg}\|_{L_{loc}^{2}}\leq\|\varphi_{i}\circ\left[c\right]-\varphi_{i}\circ\left[c_{\infty}\right]\|_{L_{loc}^{2}}+\|\varphi_{i}\circ\left[c_{\infty}\right]-\varphi_{i,seg}\circ\left[c_{\infty}\right]\|_{L_{loc}^{2}},
\]

Then, by virtue of Fréchet\textendash Kolmogorov\textquoteright s
theorem, the right-hand side vanishes as $k\to+\infty$. The same
argument holds for the weak convergence of the derivatives.
\end{proof}
\begin{rem*}
We point out that the preceding result is specific to the case of
constant $a_{1}$ and $a_{2}$ (without this assumption, one term
due to $E$ does not vanish after the integration by parts). In the
general case, we do not know if the bounds of Proposition \ref{prop:compactness}
can be improved.
\end{rem*}
\begin{cor}
If $c_{\infty}\neq0$, the parabolic limit point $\left(u_{1,seg},u_{2,seg}\right)$
obtained with the improved compactness result from Proposition \ref{prop:improved_compactness}
is also a limit point $\left(u_{1,\infty},u_{2,\infty}\right)$ in
the sense of Proposition \ref{prop:compactness}. In particular, $\left(u_{1,seg},u_{2,seg}\right)\in\mathcal{C}_{loc}^{0,\beta}\left(\mathbb{R}^{2},\mathbb{R}^{2}\right)$,
whence $\left(\varphi_{1,seg},\varphi_{2,seg}\right)\in\mathcal{C}_{loc}^{0,\beta}\left(\mathbb{R}^{2},\mathbb{R}^{2}\right)$
as well. 
\end{cor}

\begin{rem*}
The case $c_{\infty}=0$ is somehow degenerate and does not really
correspond to what intuition calls a \textquotedblleft pulsating\textquotedblright{}
front. Moreover, we will need quite different techniques to handle
the two cases and, even in the very end, there will be no clear common
framework. Therefore, hereafter, we call the case $c_{\infty}=0$
\textquotedblleft segregated stationary equilibrium\textquotedblright{}
whereas the case $c_{\infty}\neq0$ is referred to as \textquotedblleft segregated
pulsating front\textquotedblright . These terms will be precisely
defined in a moment.
\end{rem*}

\subsection{Characterization of the segregated stationary equilibrium}

In this subsection, we assume $c_{\infty}=0$ and we use Proposition
\ref{prop:compactness} to get an extracted convergent subsequence
of pulsating fronts, still denoted $\left(\left(u_{1,k},u_{2,k}\right)\right)_{k>k^{\star}}$,
with limit $\left(u_{1,\infty},u_{2,\infty}\right)$. Up to an additional
extraction, we assume a.e. convergence of $\left(u_{1,k},u_{2,k},u_{1,k}u_{2,k}\right)$
to $\left(u_{1,\infty},u_{2,\infty},0\right)$. 

Obviously, since $c_{\infty}=0$, we expect that $\left(u_{1,\infty},u_{2,\infty}\right)$
does not depend on $t$. This will be true indeed, so that it makes
sense to refer to this case as \textquotedblleft stationary equilibrium\textquotedblright .
To stress this particularity, we fix $t_{cv}$ such that $\left(\left(u_{1},u_{2}\right)_{|\left\{ t_{cv}\right\} \times\mathbb{R}}\right)_{k>k^{\star}}$
converges a.e. and we define $e=\left(v_{d,\infty}\right)_{|\left\{ t_{cv}\right\} \times\mathbb{R}}$,
so that if $\left(u_{1,\infty},u_{2,\infty}\right)$ is constant with
respect to $t$, $\left(\alpha u_{1,\infty},du_{2,\infty}\right)\left(t,x\right)=\left(e^{+},e^{-}\right)\left(x\right)$
for any $\left(t,x\right)\in\mathbb{R}^{2}$.

We start with an important particular case.
\begin{lem}
Assume that, provided $k^{\star}$ is large enough, $\left(c_{k}\right)_{k>k^{\star}}=0$.
Then:
\begin{itemize}
\item for any $k>k^{\star}$, $\left(u_{1},u_{2}\right)$ reduces to: 
\[
\left(t,x\right)\mapsto\left(\varphi_{1},\varphi_{2}\right)\left(x,x\right),
\]
\item for any $\left(t,x\right)\in\mathbb{R}^{2}$: 
\[
\left(\alpha u_{1,\infty},du_{2,\infty}\right)\left(t,x\right)=\left(e^{+},e^{-}\right)\left(x\right),
\]
\item the convergence of $\left(\left(\alpha u_{1},du_{2}\right)_{|\left\{ t_{cv}\right\} \times\mathbb{R}}\right)_{k>k^{\star}}$
to $\left(e^{+},e^{-}\right)$ actually occurs in $\mathcal{C}_{loc}^{0,\beta}\left(\mathbb{R}\right)$,
\item the convergence of $\left(\left(v_{d}\right)_{|\left\{ t_{cv}\right\} \times\mathbb{R}}\right)_{k>k^{\star}}$
to $e$ actually occurs in $\mathcal{C}_{loc}^{2,\beta}\left(\mathbb{R}\right)$,
\item $e$ satisfies:
\[
-e''=\eta\left[e\right].
\]
\end{itemize}
\end{lem}

\begin{proof}
The system $\left(\mathcal{P}\right)$ reduces to an elliptic system.
It is then easy to deduce the locally uniform convergence, the time-independence
and the limiting equation. We refer, for instance, to \cite{Girardin_2016}
for details.
\end{proof}
Some of the preceding results can be extended.
\begin{lem}
The properties:
\begin{itemize}
\item for any $\left(t,x\right)\in\mathbb{R}^{2}$, $\left(\alpha u_{1,\infty},du_{2,\infty}\right)\left(t,x\right)=\left(e^{+},e^{-}\right)\left(x\right)$
;
\item $e\in\mathcal{C}^{2}\left(\mathbb{R}\right)$ and $-e''=\eta\left[e\right]$;
\end{itemize}
hold true regardless of any sign assumption on the sequence $\left(c_{k}\right)_{k>k^{\star}}$.
\end{lem}

\begin{proof}
The two statements are actually quite easy to verify. Let $\left(t,t',x\right)\in\mathbb{R}^{3}$
such that, for any $i\in\left\{ 1,2\right\} $ and any $\tau\in\left\{ t,t'\right\} $,
$u_{i,k}\left(\tau,x\right)\to u_{i,\infty}\left(\tau,x\right)$ as
$k\to+\infty$. Recalling that: 
\[
\int_{\mathbb{R}}\partial_{t}u_{i,k}=-c_{k}\int_{\mathbb{R}}\partial_{\xi}\varphi_{i,k}\to0
\]
 as $k\to+\infty$ is sufficient to show that in the following inequality:
\begin{eqnarray*}
\left|u_{i,\infty}\left(t,x\right)-u_{i,\infty}\left(t',x\right)\right| & \leq & \left|u_{i,\infty}\left(t,x\right)-u_{i,k}\left(t,x\right)\right|\\
 &  & +\left|\int_{t}^{t'}\partial_{t}u_{i,k}\left(\tau,x\right)\mbox{d}\tau\right|\\
 &  & +\left|u_{i,k}\left(t',x\right)-u_{i,\infty}\left(t',x\right)\right|
\end{eqnarray*}
the right-hand side converges to $0$ as $k\to+\infty$. Therefore
the left-hand side is $0$, whence $u_{i,\infty}$ is constant with
respect to the time variable in a dense subset of $\mathbb{R}^{2}$,
and then by continuity, it holds \textit{a fortiori} everywhere in
$\mathbb{R}^{2}$.

As for the regularity and limiting equation, the equation is satisfied
\textit{a priori} in the distributional sense, then in the classical
sense by elliptic regularity.
\end{proof}
\begin{lem}
For any $x\in\mathbb{R}$, the sequence $\left(e\left(x+nL\right)\right)_{n\in\mathbb{N}}$
is non-increasing.
\end{lem}

\begin{proof}
By monotonicity with respect to $\xi$ and periodicity with respect
to $x$, for any $\left(t,x\right)\in\mathbb{R}^{2}$ and any $k>k^{\star}$:
\begin{eqnarray*}
v_{d}\left(t,x+L\right)-v_{d}\left(t,x\right) & \leq & \psi_{d}\left(x-ct+L,x+L\right)-\psi_{d}\left(x-ct,x\right)\\
 & \leq & \psi_{d}\left(x-ct+L,x\right)-\psi_{d}\left(x-ct,x\right)\\
 & \leq & 0.
\end{eqnarray*}

In particular, for any $\left(t,x\right)\in\mathbb{R}^{2}$ and any
$k>k^{\star}$, the sequence $\left(v_{d,k}\left(t,x+nL\right)\right)_{n\in\mathbb{N}}$
is non-increasing, and then, passing to the limit as $k\to+\infty$,
the sequence $\left(e\left(x+nL\right)\right)_{n\in\mathbb{N}}$ is
non-increasing. This holds for any $x$ in a dense subset of $\mathbb{R}$
and then for any $x\in\mathbb{R}$ by continuity of $e$.
\end{proof}
\begin{lem}
\label{lem:SSE_non-zero_sign-changing}$e$ is non-zero and sign-changing.
Moreover:
\[
\inf e^{-1}\left(\left(-\infty,0\right)\right)>-\infty.
\]
\end{lem}

\begin{proof}
The normalization:

\[
0=\inf\left\{ \xi\in\mathbb{R}\ |\ \exists x\in\overline{C}\quad\varphi_{1,k}\left(\xi,x\right)<\frac{a_{1}}{2}\right\} ,
\]
implies that $u_{1,\infty}\neq0$, whence $e\neq0$. It shows also
that the set:
\[
\left\{ n\in\mathbb{Z}\ |\ \exists x\in\overline{C}\quad\varphi_{1,k}\left(x+nL,x+nL\right)<\frac{a_{1}}{2}\right\} 
\]
 is uniformly bounded with respect to $k$ from below. In particular,
it has a minimum $\underline{n}_{k}\in\mathbb{Z}$. Then let: 
\[
x_{k}=\inf\left\{ x\in\overline{C}\ |\ \varphi_{1,k}\left(x+\underline{n}_{k}L,x+\underline{n}_{k}L\right)<\frac{a_{1}}{2}\right\} ,
\]
 so that: 
\[
\varphi_{1,k}\left(x,x\right)>\frac{a_{1}}{2}\mbox{ for any }x<x_{k}+\underline{n}_{k}L.
\]
By monotonicity, we deduce: 
\[
\varphi_{1,k}\left(\xi,x\right)>\frac{a_{1}}{2}\mbox{ for any }\xi<x<\underline{n}_{k}L.
\]
If (up to extraction) $\underline{n}_{k}\to+\infty$ as $k\to+\infty$,
then the definition of the normalization is contradicted by the preceding
inequality evaluated at $\xi=0$ and $x\in\left[L,2L\right]$, whence
$\left(\underline{n}_{k}\right)_{k>k^{\star}}$ is uniformly bounded
from above as well. In particular, up to extraction, $\left(\underline{n}_{k}\right)_{k>k^{\star}}$
converges to a finite limit. The finiteness of $\inf\left\{ x\in\mathbb{R}\ |\ e\left(x\right)<0\right\} $
follows immediately. 

By uniqueness, if $e>0$, $e=\alpha a_{1}$. This is discarded by
the finiteness of $\lim\limits _{k\to+\infty}\underline{n}_{k}$,
whence $e$ is sign-changing. 
\end{proof}
\begin{rem*}
If, instead of the normalization sequence: 
\[
0=\inf\left\{ \xi\in\mathbb{R}\ |\ \exists x\in\overline{C}\quad\varphi_{1}\left(\xi,x\right)<\frac{a_{1}}{2}\right\} \mbox{ for any }k>k^{\star},
\]
 we choose:

\[
0=\sup\left\{ \xi\in\mathbb{R}\ |\ \exists x\in\overline{C}\quad\varphi_{2}\left(\xi,x\right)<\frac{a_{2}}{2}\right\} \mbox{ for any }k>k^{\star},
\]
and if we consider once again the case $c_{\infty}=0$, the preceding
results hold apart from $\inf e^{-1}\left(\left(-\infty,0\right)\right)>-\infty$,
which is naturally replaced by: 
\[
\sup e^{-1}\left(\left(0,+\infty\right)\right)<+\infty.
\]
\end{rem*}
In view of these results, we state the following definition.
\begin{defn}
\label{def:seg_s_e}A function $z\in\mathcal{C}^{2}\left(\mathbb{R}\right)\cap L^{\infty}\left(\mathbb{R}\right)$
is called a \textit{segregated stationary equilibrium} if:
\begin{enumerate}
\item $-z''=\eta\left[z\right]$;
\item for any $x\in\overline{C}$, $\left(z\left(x+nL\right)\right)_{n\in\mathbb{N}}$
is non-increasing;
\item $z$ is non-zero and sign-changing;
\item $\inf z^{-1}\left(\left(-\infty,0\right)\right)>-\infty$ or $\sup z^{-1}\left(\left(0,+\infty\right)\right)<+\infty$.
\end{enumerate}
\end{defn}

\begin{cor}
$e$ is a segregated stationary equilibrium.
\end{cor}

Let us derive some properties necessarily satisfied by any segregated
stationary equilibrium. The first one is obvious but will be useful.
\begin{prop}
\label{prop:seg_s_e_invariant_shift} If $z$ is a segregated stationary
equilibrium, then for any $n\in\mathbb{Z}$, $x\mapsto z\left(x+nL\right)$
is a segregated stationary equilibrium as well. 
\end{prop}

The following one is easily derived from the second order necessary
conditions satisfied at a local extremum.
\begin{prop}
\label{prop:L_infty_estimates_SSE} Let $z$ be a segregated stationary
equilibrium. Then $-da_{2}<z<\alpha a_{1}$.
\end{prop}

The following one highlights some difficulties which are intrinsic
to the null speed limit.
\begin{prop}
Let $z$ be a segregated stationary equilibrium and 
\[
\mathcal{Z}\left(z\right)=z^{-1}\left(\left\{ 0\right\} \right).
\]
The set $\mathcal{Z}\left(z\right)$ is a discrete set. If it is a
finite set, its cardinal is odd. Moreover, it has a minimum or a maximum.
\end{prop}

\begin{proof}
The fact that $\mathcal{Z}\left(z\right)$ is a discrete set follows
easily from Hopf\textquoteright s lemma and the regularity of $z$.
Provided finiteness of the set, the monotonicity of $\left(z\left(x+nL\right)\right)_{n\in\mathbb{N}}$
for any $x\in C$ yields the parity of $\#\mathcal{Z}\left(z\right)$.
Finally, the existence of an extremum comes from the definition of
the segregated stationary equilibrium.
\end{proof}
\begin{rem*}
Under the more restrictive assumption $\left(\mathcal{H}_{freq}\right)$
presented by the first author in \cite{Girardin_2016}, it is possible
to prove that every segregated stationary equilibrium has a unique
zero. It is basically deduced from the fact that, when there are multiple
zeros, the segregated stationary equilibrium restricted to any interval
delimited by two consecutive zeros is the unique solution of a semi-linear
Dirichlet problem. The monotonicity of $\left(e\left(x+nL\right)\right)_{n\in\mathbb{N}}$
ensures that the distance between these consecutive zeros is smaller
than $L$ and then, considering the next zero and using $\left(\mathcal{H}_{freq}\right)$,
a contradiction arises. We do not detail this proof here. 
\end{rem*}
\begin{prop}
\label{prop:SSE_behavior_infinity} Let $z$ be a stationary segregated
equilibrium. 

If $z^{-1}\left(\left\{ 0\right\} \right)$ has a minimum, as $n\to+\infty$,
\[
\|z-\alpha a_{1}\|_{\mathcal{C}^{2}\left(\left[-\left(n+1\right)L,-nL\right]\right)}\to0.
\]

If $z^{-1}\left(\left\{ 0\right\} \right)$ has a maximum, as $n\to+\infty$,
\[
\|z-da_{2}\|_{\mathcal{C}^{2}\left(\left[nL,\left(n+1\right)L\right]\right)}\to0.
\]
\end{prop}

\begin{proof}
We assume that $z^{-1}\left(\left\{ 0\right\} \right)$ has a minimum,
the other case being similar. Since, for any $x\in[0,L)$, $\left(z\left(x-nL\right)\right)_{n\in\mathbb{N}}$
is bounded and non-decreasing, it converges to a limit $z_{-\infty}\left(x\right)$.
Using Lipschitz-continuity of $z$, we are able to prove that $z_{-\infty}$
is Lipschitz-continuous in $\overline{C}$. Using elliptic regularity,
the distributional equation: 
\[
-z_{-\infty}''=z_{-\infty}f_{1}\left[\frac{z_{-\infty}}{\alpha}\right]
\]
 and Arzela\textendash Ascoli\textquoteright s theorem, we are able
to prove in fact that $z_{-\infty}\in\mathcal{C}^{2,\beta}\left(\overline{C}\right)$
and that the convergence occurs in $\mathcal{C}^{2,\beta}\left(\overline{C}\right)$.
This proves that $z_{-\infty}$ also satisfies in the classical sense
the equation. Moreover, 
\[
\left|z\left(x-\left(n+1\right)L\right)-z\left(x-nL\right)\right|\to0
\]
 as $n\to+\infty$ and, this proves that $z_{-\infty}$ is periodic.
Since it is also positive, by uniqueness, $z_{-\infty}=\alpha a_{1}$.
\end{proof}

\subsection{Characterization of the segregated pulsating fronts}

In this subsection, we assume $c_{\infty}\neq0$ and we use Proposition
\ref{prop:improved_compactness} to get an extracted convergent subsequence
of profiles, still denoted $\left(\left(\varphi_{1,k},\varphi_{2,k}\right)\right)_{k>k^{\star}}$,
with limit $\left(\varphi_{1,seg},\varphi_{2,seg}\right)$. Up to
an additional extraction, we assume a.e. convergence of $\left(\varphi_{1,k},\varphi_{2,k},\varphi_{1,k}\varphi_{2,k}\right)$
to $\left(\varphi_{1,seg},\varphi_{2,seg},0\right)$. We define $\phi=\alpha\varphi_{1,seg}-d\varphi_{2,seg}$
and $w=\alpha u_{1,seg}-du_{2,seg}$ (that is, $\left(\phi,w\right)$
is the limit of $\left(\left(\psi_{d,k},v_{d,k}\right)\right)_{k>k^{\star}}$). 

Here, parabolic limit points and traveling limit points are naturally
related by the isomorphism $\left(t,x\right)\mapsto\left(x-c_{\infty}t,x\right)$.
Therefore we can freely use the more convenient system of variables. 

\subsubsection{Definitions and asymptotics}

Hereafter, 
\[
\sigma:z\mapsto\mathbf{1}_{z>0}+\frac{1}{d}\mathbf{1}_{z<0},
\]
\[
\hat{\sigma}:z\mapsto\mathbf{1}_{z>0}+d\mathbf{1}_{z<0}.
\]
\begin{rem*}
Clearly, for any $z\in\mathcal{C}\left(\mathbb{R}^{2}\right)$: 
\begin{itemize}
\item $\sigma\left[z\right]$ and $\hat{\sigma}\left[z\right]$ are in $L^{\infty}\left(\mathbb{R}^{2}\right)$;
\item $\sigma\left[z\right]$ and $\hat{\sigma}\left[z\right]$ vanish if
and only if $z$ vanish;
\item $\sigma\left[z\right]\hat{\sigma}\left[z\right]=1$ in $\mathbb{R}^{2}$
apart from the zero set of $z$;
\item $\sigma\left[z\right]z$ and $\hat{\sigma}\left[z\right]z$ are in
$\mathcal{C}\left(\mathbb{R}^{2}\right)$; furthermore, if $z\in W^{1,\infty}\left(\mathbb{R}^{2}\right)$,
then they are Lipschitz-continuous.
\end{itemize}
\end{rem*}
\begin{lem}
\label{lem:SPF_equalities_wd} The equalities: 
\[
\sigma\left[w\right]\left(t,x\right)=\sigma\left[\phi\right]\left(x-c_{\infty}t,x\right),
\]
\[
\hat{\sigma}\left[w\right]\left(t,x\right)=\hat{\sigma}\left[\phi\right]\left(x-c_{\infty}t,x\right),
\]
 hold for all $\left(t,x\right)\in\mathbb{R}^{2}$. 

Furthermore, the following equalities hold in $L_{loc}^{2}\left(\mathbb{R}^{2}\right)$:
\[
\partial_{t}\left(\sigma\left[w\right]w\right)=\sigma\left[w\right]\partial_{t}w,
\]
\[
\partial_{x}w=\hat{\sigma}\left[w\right]\partial_{x}\left(\sigma\left[w\right]w\right),
\]
\[
\partial_{\xi}\left(\sigma\left[\phi\right]\phi\right)=\sigma\left[\phi\right]\partial_{\xi}\phi,
\]
\[
\partial_{x}\phi=\hat{\sigma}\left[\phi\right]\partial_{x}\left(\sigma\left[\phi\right]\phi\right).
\]
\end{lem}

\begin{proof}
The equalities between the weak derivatives are derived easily from
the weak formulation of $\left(\mathcal{PF}\right)$ (recall the proof
of Proposition \ref{prop:compactness}). When passing to the limit
$k\to+\infty$, it is possible to obtain equivalently all these equations
(we restrict ourselves here to parabolic coordinates, the equalities
in traveling coordinates being obtained analogously):
\[
\sigma\left[w\right]\partial_{t}w-\partial_{xx}w=\eta\left[w\right],
\]
\[
\partial_{t}\left(\sigma\left[w\right]w\right)-\partial_{xx}w=\eta\left[w\right],
\]
\[
\partial_{t}\left(\sigma\left[w\right]w\right)-\partial_{x}\left(\hat{\sigma}\left[w\right]\partial_{x}\left(\sigma\left[w\right]w\right)\right)=\eta\left[w\right].
\]
\end{proof}
\begin{defn}
Let $s\in\mathbb{R}\backslash\left\{ 0\right\} $ and $\mathcal{C}_{0}^{1}\left(\mathbb{R}^{2}\right)$
be the subset of compactly supported elements of $\mathcal{C}^{1}\left(\mathbb{R}^{2}\right)$. 

We say that $\varphi\in\mathcal{C}\left(\mathbb{R}^{2}\right)\cap H_{loc}^{1}\left(\mathbb{R}^{2}\right)$
is a weak solution of: 
\[
-\mbox{div}\left(E\nabla\varphi\right)-s\partial_{\xi}\left(\sigma\left[\varphi\right]\varphi\right)=\eta\left[\varphi\right]\quad\left(\mathcal{SPF}\left[s\right]\right)
\]
 if, for any test function $\zeta\in\mathcal{C}_{0}^{1}\left(\mathbb{R}^{2}\right)$:
\[
\int E\nabla\varphi.\nabla\zeta+s\int\sigma\left[\varphi\right]\varphi\partial_{\xi}\zeta=\int\eta\left[\varphi\right]\zeta.
\]
\end{defn}

\begin{lem}
$\phi$ is a weak solution of $\left(\mathcal{SPF}\left[c_{\infty}\right]\right)$.
\end{lem}

\begin{proof}
This is merely the traveling formulation of the limiting equation
obtained \textit{a priori} in $\mathcal{D}'\left(\mathbb{R}^{2}\right)$
and \textit{a fortiori} holding in the weak sense.
\end{proof}
\begin{rem*}
Since $c_{\infty}\sigma\left[\phi\right]\partial_{\xi}\phi$ and $\eta\left[\phi\right]$
are in $L_{loc}^{2}\left(\mathbb{R}^{2}\right)$, $-\mbox{div}\left(E\nabla\phi\right)$
is actually in $L_{loc}^{2}\left(\mathbb{R}^{2}\right)$ as well and
we can also consider test functions in $L_{loc}^{2}\left(\mathbb{R}^{2}\right)$,
but then we cannot integrate by parts as in the equality above. 
\end{rem*}
\begin{prop}
Let $s\in\mathbb{R}\backslash\left\{ 0\right\} $. If $\varphi$ is
a weak solution of $\left(\mathcal{SPF}\left[s\right]\right)$, then
$z:\left(t,x\right)\mapsto\varphi\left(x-st,x\right)$ is a weak solution
of:

\[
\partial_{t}\left(\sigma\left[z\right]z\right)-\partial_{xx}z=\eta\left[z\right],
\]
 in the sense that for any $\zeta\in\mathcal{C}_{0}^{1}\left(\mathbb{R}^{2}\right)$,
the following holds:
\[
\int\left(\sigma\left[z\right]z\partial_{t}\zeta-\partial_{x}z\partial_{x}\zeta+\eta\left[z\right]\zeta\right)=0.
\]
\end{prop}

\begin{rem*}
Similarly, we can restrict ourselves regarding this weak parabolic
equation to test functions $\zeta\in L_{loc}^{2}\left(\mathbb{R}^{2}\right)$
but then we cannot integrate by parts.
\end{rem*}
\begin{lem}
$\phi$ is periodic with respect to $x$ and non-increasing with respect
to $\xi$. 
\end{lem}

\begin{proof}
Thanks to the a.e. convergence, periodicity with respect to $x$ and
monotonicity with respect to $\xi$ are preserved a.e., that is at
least in a dense subset of $\mathbb{R}^{2}$. Continuity extends these
behaviors everywhere.
\end{proof}
\begin{lem}
$\phi$ is non-zero and sign-changing. 
\end{lem}

\begin{rem*}
This statement holds if and only if both $\varphi_{1,seg}$ and $\varphi_{2,seg}$
are non-zero (or equivalently non-negative non-zero).
\end{rem*}
\begin{proof}
Assume for example $c_{\infty}<0$. The normalization gives immediately
$\varphi_{1,seg}\neq0$. If $\varphi_{2,seg}=0$, $u_{1,seg}$ is
a non-negative solution in $\mathbb{R}^{2}$ of:
\[
\partial_{t}z-\partial_{xx}z=zf_{1}\left[z\right].
\]

By the parabolic strong minimum principle, $u_{1,seg}\gg0$, and by
parabolic regularity, $u_{1,seg}$ is regular. By classical parabolic
estimates, as $\xi\to-\infty$, $\varphi_{1,seg}$ converges uniformly
in $x$ to a positive periodic solution of: 
\[
-\partial_{xx}z=zf_{1}\left[z\right],
\]
that is to $a_{1}$. Similarly, $\varphi_{1,seg}$ converges to $0$
as $\xi\to+\infty$.

Thus $\varphi_{1,seg}$ is a pulsating front connecting $a_{1}$ to
$0$ at speed $c_{\infty}<0$. This is a contradiction (see Theorem
\ref{thm:BHR2}). 

A symmetric proof discards the case $c_{\infty}>0$. 
\end{proof}
In view of these results, we state the following definition.
\begin{defn}
\label{def:seg_t_w} Let:
\[
s\in\mathbb{R}\backslash\left\{ 0\right\} ,
\]
\[
z\in\mathcal{C}_{loc}^{0,\beta}\left(\mathbb{R}^{2}\right)\cap H_{loc}^{1}\left(\mathbb{R}^{2}\right)\cap L^{\infty}\left(\mathbb{R}^{2}\right),
\]
\[
\varphi:\left(\xi,x\right)\mapsto z\left(\frac{x-\xi}{s},x\right).
\]

$z$ is called a \textit{segregated pulsating front with speed $s$
and profile $\varphi$} if:
\begin{enumerate}
\item $\varphi$ is a weak solution of $\left(\mathcal{SPF}\left[s\right]\right)$;
\item $\varphi$ is non-increasing with respect to $\xi$;
\item $\varphi$ is periodic with respect to $x$;
\item $\varphi$ is non-zero and sign-changing.
\end{enumerate}
\end{defn}

\begin{cor}
$w$ is a segregated pulsating front with speed $c_{\infty}$ and
profile $\phi$.
\end{cor}

\begin{prop}
\label{prop:behavior_at_infinity_SPF} Let $z$ be a segregated pulsating
front with profile $\varphi$. As $\xi\to+\infty$,
\[
\max_{x\in\overline{C}}\left|\varphi\left(-\xi,x\right)-\alpha a_{1}\right|+\max_{x\in\overline{C}}\left|\varphi\left(\xi,x\right)+da_{2}\right|\to0.
\]
\end{prop}

\begin{proof}
It follows from classical parabolic estimates and the monotonicity
of $\varphi$ with respect to $\xi$.
\end{proof}

\subsubsection{The intrinsic free boundary problem}

We intend to conclude the characterization of the segregated pulsating
front with a uniqueness result. Our proof will use a sliding argument
and the continuity of $\partial_{x}z$. Obviously, in $\mathbb{R}^{2}\backslash z^{-1}\left(\left\{ 0\right\} \right)$,
classical parabolic regularity applies and the regularity of a segregated
pulsating front is only limited by that of $\eta$. On the contrary,
the regularity of $z$ at the free boundary $z^{-1}\left(\left\{ 0\right\} \right)$
is a tough problem and, as usual in free boundary problems, requires
a detailed study of the regularity of the free boundary itself. This
study is the object of the following pages. 

Let us stress here that our interest does not lie in the most general
study of the free boundaries of the solutions of $\left(\mathcal{SPF}\left[s\right]\right)$.
To show that $\partial_{x}z$ is continuous, Lipschitz-continuity
of the free boundary is sufficient, and we are able to prove such
a regularity only using the monotonicity properties of the segregated
pulsating fronts as well as the parabolic maximum principle. We believe
that this proof has interest of its own. Yet, at the end of this subsection,
we will explain why we expect the free boundary to actually be $\mathcal{C}^{1}$
and $\partial_{t}z$ to be continuous without any additional assumption.

Up to the next subsection, let $z$ be a segregated pulsating front
with speed $s\neq0$ and profile $\varphi$ and let: 
\[
\Gamma=\left\{ \left(t,x\right)\in\mathbb{R}^{2}\ |\ z\left(t,x\right)=0\right\} ,
\]
\[
\Omega_{+}=\left\{ \left(t,x\right)\in\mathbb{R}^{2}\ |\ z\left(t,x\right)>0\right\} ,
\]
\[
\Omega_{-}=\left\{ \left(t,x\right)\in\mathbb{R}^{2}\ |\ z\left(t,x\right)<0\right\} .
\]

Before going any further, let us state precisely the results of this
subsection in the following proposition.
\begin{thm}
\label{thm:FB_reg_sum_up} There exists a continuous bijection $\Xi:\mathbb{R}\to\mathbb{R}$
such that $\Gamma$ is the graph of $\Xi$ and such that:
\[
\left\{ \begin{matrix}\Omega_{+}=\left\{ \left(t,x\right)\in\mathbb{R}^{2}\ |\ x<\Xi\left(t\right)\right\} \,\\
\Omega_{-}=\left\{ \left(t,x\right)\in\mathbb{R}^{2}\ |\ x>\Xi\left(t\right)\right\} .
\end{matrix}\right.
\]

Moreover, $\partial_{x}z\in\mathcal{C}^{0,\beta}\left(\mathbb{R}^{2}\right)$
and $\left(\partial_{x}z\right)_{|\Gamma}\ll0$.
\end{thm}

\begin{rem*}
Of course, this type of result is strongly reminiscent of the celebrated
paper by Angenent \cite{Angenent_1988} about the number of zeros
of a solution of a parabolic equation. We stress that this result
cannot be applied here because of the non-linearity due to $\sigma\left[z\right]$.
It will be clearly established during the proof that this lack of
regularity is compensated here by the monotonicity of $z$.
\end{rem*}
The proof of Theorem \ref{thm:FB_reg_sum_up} begins with a couple
of lemmas leading to the existence of $\Xi$.
\begin{lem}
\label{lem:def_Xi_-_Xi_+} The quantities:
\[
\Xi_{+}\left(t\right)=\sup\left\{ x\in\mathbb{R}\ |\ z\left(t,x\right)>0\right\} ,
\]
\[
\Xi_{-}\left(t\right)=\inf\left\{ x\in\mathbb{R}\ |\ z\left(t,x\right)<0\right\} ,
\]
are well-defined and finite. 
\end{lem}

\begin{proof}
By Proposition \ref{prop:behavior_at_infinity_SPF}, for any $\left(t,x\right)\in\mathbb{R}^{2}$:
\[
\lim_{n\to+\infty}\max_{x\in\overline{C}}\left|\varphi\left(x+nL-st,x\right)+da_{2}\right|=0,
\]
\[
\lim_{n\to+\infty}\max_{x\in\overline{C}}\left|\varphi\left(x-nL-st,x\right)-\alpha a_{1}\right|=0.
\]

By periodicity with respect to $x$:
\begin{eqnarray*}
\varphi\left(x\pm nL-st,x\right) & = & \varphi\left(x\pm nL-st,x\pm nL\right)\\
 & = & z\left(t,x\pm nL\right)
\end{eqnarray*}
 and thus $x\mapsto z\left(t,x\right)$ is negative at $+\infty$,
positive at $-\infty$, whence $\Xi_{+}\left(t\right)$ and $\Xi_{-}\left(t\right)$
are well-defined and finite. 
\end{proof}
\begin{lem}
\label{lem:sign_of_SPF_on_left_or_right} Let $\left(t,x\right)\in\mathbb{R}^{2}$. 
\begin{enumerate}
\item If $s>0$ and $z\left(t,x\right)\leq0$, then for any $y>x$, $z\left(t,y\right)<0$.
\item If $s<0$ and $z\left(t,x\right)\geq0$, then for any $y<x$, $z\left(t,y\right)>0$.
\end{enumerate}
\end{lem}

\begin{proof}
Let us show for instance the first statement, the other one being
symmetric.

By Lemma \ref{lem:def_Xi_-_Xi_+}, there exists $X>x$ such that $z\left(t,X\right)<0$.
Since $\varphi$ is non-increasing with respect to $\xi$, $z$ is
non-decreasing with respect to $t$, whence for any $t'<t$, $z\left(t',x\right)\leq0$
and $z\left(t',X\right)<0$. Moreover, by Proposition \ref{prop:behavior_at_infinity_SPF},
there exists $T>0$ such that:
\[
z\left(t-T,y\right)<0\mbox{ for any }y\in\left[x,X\right].
\]
By continuity of $z$, there exists $\tau>0$ such that:
\[
z\ll0\mbox{ in }\left[t-T,t-T+\tau\right]\times\left[x,X\right].
\]

Let:
\[
\tau^{\star}=\sup\left\{ \tau\in\left(0,T\right)\ |\ z\ll0\mbox{ in }\left[t-T,t-T+\tau\right]\times\left(x,X\right)\right\} 
\]
 and let us check that $\tau^{\star}=T$.

If $\tau^{\star}<T$, then there exists $y\in\left(x,X\right)$ such
that $z\left(t-T+\tau^{\star},y\right)=0$. But in the parabolic cylinder
$\left[t-T,t-T+\tau^{\star}\right]\times\left[x,X\right]$, $z<0$
satisfies a regular parabolic equation and satisfies also the strong
parabolic maximum principle, which immediately contradicts the strict
sign of $z$ at $t-T$.

Thus $\tau^{\star}=T$ and then, if there exists $y\in\left(x,X\right)$
such that $z\left(t,y\right)=0$, applying once more the strong parabolic
maximum principle gives the same contradiction. 

The proof is ended by passing to the limit $X\to+\infty$.
\end{proof}
\begin{cor}
For any $t\in\mathbb{R}$, the zero of $x\mapsto z\left(t,x\right)$
is unique, or equivalently, $\Xi_{+}\left(t\right)=\Xi_{-}\left(t\right)$.
\end{cor}

\begin{lem}
\label{lem:free_boundary_lem1} For any $t\in\mathbb{R}$, let $\Xi\left(t\right)$
be the unique zero of $x\mapsto z\left(t,x\right)$. 

Then $\Xi:\mathbb{R}\to\mathbb{R}$ is unbounded, non-decreasing if
$s>0$ and non-increasing if $s<0$, and continuous. 

Furthermore, $\Gamma$ is exactly the graph of $\Xi$,
\[
\Omega_{-}=\left\{ \left(t,x\right)\in\mathbb{R}^{2}\ |\ x>\Xi\left(t\right)\right\} ,
\]
\[
\Omega_{+}=\left\{ \left(t,x\right)\in\mathbb{R}^{2}\ |\ x<\Xi\left(t\right)\right\} .
\]
\end{lem}

\begin{proof}
Assume for instance and up to the end of the proof $s>0$ (the case
$s<0$ is similar). 

Since $\varphi$ is non-increasing with respect to $\xi$, $z$ is
non-decreasing with respect to $t$. Assume by contradiction that
there exists $t,t'\in\mathbb{R}$ such that $t'<t$ and $\Xi\left(t\right)<\Xi\left(t'\right)$.
By Lemma \ref{lem:sign_of_SPF_on_left_or_right}, for any $x>\Xi\left(t\right)$,
$z\left(t,x\right)<0$, whence in particular $z\left(t,\Xi\left(t'\right)\right)<0$,
whence by monotonicity of $z$, $z\left(t',\Xi\left(t'\right)\right)<0$,
which contradicts the definition of $\Xi\left(t'\right)$. Thus $\Xi$
is non-decreasing.

The unboundedness is straightforward: considering the limiting signs
of $t\mapsto z\left(t,x\right)$ shows by continuity that this function
has at least one zero for any $x\in\mathbb{R}$. But if $\Xi$ was
bounded, thanks to Lemma \ref{lem:sign_of_SPF_on_left_or_right} once
again, it would be possible to build a counter-example.

Finally, continuity is also straightforward, since it is well-known
that a monotonic function admits left-sided and right-sided limits
at every point and that every discontinuity it has is a jump discontinuity.
The existence of such a discontinuity, that is of a segment $\left\{ t^{\star}\right\} \times\left[x^{\star},x^{\star}+X\right]$
included in the free boundary, would immediately contradict Lemma
\ref{lem:sign_of_SPF_on_left_or_right}.
\end{proof}
\begin{cor}
Both $\Omega_{+}$ and $\Omega_{-}$ have a Lipschitz boundary.
\end{cor}

\begin{proof}
It suffices to recall that every point of the graph of a monotone
function satisfies an interior cone condition and that such a condition
characterizes Lipschitz boundaries. 
\end{proof}
In view of this regularity of $\Omega_{\pm}$ and by means of easy
integration by parts, we are now able to generalize to any segregated
pulsating front a property that was immediately satisfied by $w$
(Lemma \ref{lem:SPF_equalities_wd}).
\begin{cor}
\label{cor:SPF_weak_derivatives_bis} The following equalities hold
in $L_{loc}^{2}\left(\mathbb{R}^{2}\right)$:
\[
\partial_{t}\left(\sigma\left[z\right]z\right)=\sigma\left[z\right]\partial_{t}z,
\]
\[
\partial_{x}z=\hat{\sigma}\left[z\right]\partial_{x}\left(\sigma\left[z\right]z\right),
\]
\[
\partial_{\xi}\left(\sigma\left[\varphi\right]\varphi\right)=\sigma\left[\varphi\right]\partial_{\xi}\varphi,
\]
\[
\partial_{x}\varphi=\hat{\sigma}\left[\varphi\right]\partial_{x}\left(\sigma\left[\varphi\right]\varphi\right).
\]
\end{cor}

\begin{proof}
Let us show for instance the first one. Let $\left(\zeta_{n}\right)_{n\in\mathbb{N}}\in\left(\mathcal{D}\left(\mathbb{R}^{2}\right)\right)^{\mathbb{N}}$
such that $\left(\zeta_{n}\right)$ converges in $L_{loc}^{2}$ to
some test function $\zeta\in L_{loc}^{2}$. For any $n\in\mathbb{N}$,
we have: 
\begin{eqnarray*}
\int\partial_{t}\left(\sigma\left[z\right]z\right)\zeta_{n} & = & -\int\sigma\left[z\right]z\partial_{t}\zeta_{n}\\
 & = & -\int_{\Omega_{+}}z\partial_{t}\zeta_{n}-\int_{\Omega_{-}}\frac{1}{d}z\partial_{t}\zeta_{n}.
\end{eqnarray*}

Since $\Omega_{\pm}$ have a Lipschitz boundary, we can integrate
by parts once again (recalling that, by definition, $z_{|\Gamma}=0$):
\begin{eqnarray*}
\int\partial_{t}\left(\sigma\left[z\right]z\right)\zeta_{n} & = & \int_{\Omega_{+}}\partial_{t}z\zeta_{n}+\int_{\Omega_{-}}\frac{1}{d}\partial_{t}z\zeta_{n}\\
 & = & \int\sigma\left[z\right]\partial_{t}z\zeta_{n}.
\end{eqnarray*}

Passing to the limit $n\to+\infty$ ends the proof.
\end{proof}
More interestingly, we are now closer to an explicit free boundary
condition. The following three lemmas are dedicated to this question.
\begin{lem}
\label{lem:trace_SPF} Let $\Xi$ be defined as in Lemma \ref{lem:free_boundary_lem1}.

Then the traces $\left(\partial_{x}z^{+}\right)_{|\partial\Omega_{+}}$
and $\left(\partial_{x}z^{-}\right)_{|\partial\Omega_{-}}$ are well-defined
in $L_{loc}^{2}\left(\partial\Omega_{+}\right)$ and $L_{loc}^{2}\left(\partial\Omega_{-}\right)$
respectively. 
\end{lem}

\begin{proof}
Since $\partial\Omega_{+}$ (respectively $\partial\Omega_{-}$) is
a Lipschitz boundary, let us prove that $\left(\partial_{x}\left(z^{+}\right)\right)_{|\Omega_{+}}$
(resp.$\left(\partial_{x}\left(z^{-}\right)\right)_{|\Omega_{-}}$)
is in $H_{loc}^{1}\left(\Omega_{+}\right)$ (resp. $H_{loc}^{1}\left(\Omega_{-}\right)$).
It is already established that it is in $L_{loc}^{2}\left(\mathbb{R}^{2}\right)$.
Considering the equation satisfied by $z$ then shows immediately
that $\left(\partial_{xx}\left(z^{+}\right)\right)_{|\Omega_{+}}$
(resp. $\left(\partial_{xx}\left(z^{-}\right)\right)_{|\Omega_{-}}$)
is in $L_{loc}^{2}\left(\mathbb{R}^{2}\right)$ as well. To conclude,
it remains to prove that $\left(\partial_{tx}\left(z^{+}\right)\right)_{|\Omega_{+}}$
(resp. $\left(\partial_{tx}\left(z^{-}\right)\right)_{|\Omega_{-}}$)
is in $L_{loc}^{2}\left(\Omega_{+}\right)$ (resp. $L_{loc}^{2}\left(\Omega_{-}\right)$). 

Let $t_{1},t_{2},x_{1},x_{2}\in\mathbb{R}$ such that $t_{1}<t_{2}$,
$x_{1}<x_{2}$ and $\left[t_{1},t_{2}\right]\times\left[x_{1},x_{2}\right]\subset\Omega_{+}$.
Let $\chi\in\mathcal{D}\left(\mathbb{R}^{2}\right)$ be a non-negative
non-zero function identically equal to $1$ in $\left[t_{1},t_{2}\right]\times\left[x_{1},x_{2}\right]$.
From the following equation, satisfied in the classical sense in $\Omega_{+}$:
\[
\partial_{t}\left(\partial_{t}z\right)-\partial_{xx}\left(\partial_{t}z\right)=g_{1}\left[\frac{z}{\alpha}\right]\partial_{t}z,
\]
 multiplied by $\partial_{t}z\chi$ and integrated over $\mathbb{R}^{2}$,
we deduce:
\[
-\int\frac{1}{2}\left|\partial_{t}z\right|^{2}\partial_{t}\chi+\int\left|\partial_{xt}z\right|^{2}\chi-\frac{1}{2}\int\left|\partial_{t}z\right|^{2}\partial_{xx}\chi=\int g_{1}\left[\frac{z}{\alpha}\right]\left|\partial_{t}z\right|^{2}\chi.
\]

It follows that there exists a constant $R>0$ such that :
\[
\|\partial_{xt}z\|_{L^{2}\left(\left[t_{1},t_{2}\right]\times\left[x_{1},x_{2}\right]\right)}^{2}\leq R\|\partial_{t}z\|_{L^{2}\left(\left[t_{1},t_{2}\right]\times\left[x_{1},x_{2}\right]\right)}\|\chi\|_{H^{2}\left(\mathbb{R}^{2}\right)},
\]
 whence $\partial_{tx}z^{+}\in L_{loc}^{2}\left(\Omega_{+}\right)$
indeed.

Similarly, $\partial_{tx}\left(z^{-}\right)\in L_{loc}^{2}\left(\Omega_{-}\right)$.

In the end, $\Omega_{+}$ and $\Omega_{-}$ are Lipschitz domains,
$\left(\partial_{x}\left(z^{+}\right)\right)_{|\Omega_{+}}\in H_{loc}^{1}\left(\Omega_{+}\right)$
and $\left(\partial_{x}\left(z^{-}\right)\right)_{|\Omega_{-}}\in H_{loc}^{1}\left(\Omega_{-}\right)$,
whence their traces can be rigorously defined in $L_{loc}^{2}\left(\partial\Omega_{+}\right)$
and $L_{loc}^{2}\left(\partial\Omega_{-}\right)$ respectively.
\end{proof}
\begin{lem}
\label{lem:usable_weak_formulation}Let $\Xi$ be defined as in Lemma
\ref{lem:free_boundary_lem1}. 

For any non-negative test function with compact support $\zeta\in\mathcal{C}_{0}^{1}\left(\mathbb{R}^{2}\right)$,
the following equalities hold:
\[
\int_{\Omega_{+}}\left(\sigma\left[z\right]z\partial_{t}\zeta-\partial_{x}z\partial_{x}\zeta+\eta\left[z\right]\zeta\right)=\int_{\partial\Omega_{+}}\partial_{x}z\zeta,
\]
\[
\int_{\Omega_{-}}\left(\sigma\left[z\right]z\partial_{t}\zeta-\partial_{x}z\partial_{x}\zeta+\eta\left[z\right]\zeta\right)=\int_{\partial\Omega_{-}}\partial_{x}z\zeta.
\]
\end{lem}

\begin{proof}
We prove the equality concerning $\Omega_{+}$, the other one being
similar. 

First, it is straightforward that:
\[
\left(\sigma\left[z\right]\right)_{|\Omega_{+}}=1.
\]

Let $\varepsilon>0$ and:
\[
\Omega_{+}^{\varepsilon}=\left\{ \left(t,x\right)\in\mathbb{R}^{2}\ |\ \Xi\left(t\right)-\varepsilon\leq x<\Xi\left(t\right)\right\} .
\]

Then:
\begin{eqnarray*}
\int_{\Omega_{+}}\left(\sigma\left[z\right]z\partial_{t}\zeta-\partial_{x}z\partial_{x}\zeta\right) & = & \int_{\Omega_{+}^{\varepsilon}}\left(z\partial_{t}\zeta-\partial_{x}z\partial_{x}\zeta\right)+\int_{\Omega_{+}\backslash\Omega_{+}^{\varepsilon}}\left(z\partial_{t}\zeta-\partial_{x}z\partial_{x}\zeta\right)
\end{eqnarray*}

Let 
\[
\tau_{\varepsilon}:x\mapsto\inf\left\{ t\in\mathbb{R}\ |\ \Xi\left(t\right)=x+\varepsilon\right\} .
\]

This function is increasing, piecewise-continuous, measurable and
satisfies the following equality:
\[
\mathbf{1}_{\Omega_{+}\backslash\Omega_{+}^{\varepsilon}}=\mathbf{1}_{\left\{ \left(t,x\right)\in\mathbb{R}^{2}\ |\ \tau_{\varepsilon}\left(x\right)\leq t\right\} }.
\]

By integration by parts and using the equation satisfied by $z$ in
$\Omega_{+}\backslash\Omega_{+}^{\varepsilon}$:
\begin{eqnarray*}
\int_{\Omega_{+}\backslash\Omega_{+}^{\varepsilon}}\left(z\partial_{t}\zeta-\partial_{x}z\partial_{x}\zeta\right) & = & -\int_{\Omega_{+}\backslash\Omega_{+}^{\varepsilon}}\eta\left[z\right]\zeta\\
 &  & -\int_{\mathbb{R}}\partial_{x}z\left(t,\Xi\left(t\right)-\varepsilon\right)\zeta\left(t,\Xi\left(t\right)-\varepsilon\right)\mbox{d}t\\
 &  & -\int_{\mathbb{R}}z\left(\tau_{\varepsilon}\left(x\right),x\right)\zeta\left(\tau_{\varepsilon}\left(x\right),x\right)\mbox{d}x.
\end{eqnarray*}

By the Cauchy\textendash Schwarz inequality and dominated convergence,
as $\varepsilon\to0$:
\[
\int_{\Omega_{+}^{\varepsilon}}\left(z\partial_{t}\zeta-\partial_{x}z\partial_{x}\zeta\right)\to0,
\]
\[
\int_{\Omega_{+}\backslash\Omega_{+}^{\varepsilon}}\eta\left[z\right]\zeta\to\int_{\Omega_{+}}\eta\left[z\right]\zeta,
\]
\[
\int_{\mathbb{R}}z\left(\tau_{\varepsilon}\left(x\right),x\right)\zeta\left(\tau_{\varepsilon}\left(x\right),x\right)\mbox{d}x\to0.
\]

Therefore, the following convergence holds as $\varepsilon\to0$:
\[
-\int_{\mathbb{R}}\partial_{x}z\left(t,\Xi\left(t\right)-\varepsilon\right)\zeta\left(t,\Xi\left(t\right)-\varepsilon\right)\mbox{d}t\to\int_{\Omega_{+}}\left(\sigma\left[z\right]z\partial_{t}\zeta-\partial_{x}z\partial_{x}\zeta\right)+\int_{\Omega_{-}}\eta\left[z\right]\zeta.
\]

Lemma \ref{lem:trace_SPF} indicates that the trace of $\partial_{x}z\zeta$
at $\partial\Omega_{+}$ is well-defined in $L^{2}$. Therefore, it
remains to show that: 
\[
\lim_{\varepsilon\to0}\int_{\mathbb{R}}\partial_{x}z\left(t,\Xi\left(t\right)-\varepsilon\right)\zeta\left(t,\Xi\left(t\right)-\varepsilon\right)\mbox{d}t=-\int_{\partial\Omega_{+}}\partial_{x}z\zeta
\]

Define, for any $\varepsilon>0$: 
\[
z_{\varepsilon}:\left(t,x\right)\mapsto z\left(t,x-\varepsilon\right),
\]
\[
\zeta_{\varepsilon}:\left(t,x\right)\mapsto\zeta\left(t,x-\varepsilon\right).
\]

It is clear that the trace of $\partial_{x}z_{\varepsilon}\zeta_{\varepsilon}$
is well-defined in $L^{2}$ as well and satisfies:
\[
\int_{\mathbb{R}}\partial_{x}z\left(t,\Xi\left(t\right)-\varepsilon\right)\zeta\left(t,\Xi\left(t\right)-\varepsilon\right)\mbox{d}t=\int_{\partial\Omega_{+}}\left(-1\right)\partial_{x}z_{\varepsilon}\zeta_{\varepsilon}.
\]

Now, by virtue of the trace\textquoteright s theorem, there exists
a constant $R>0$ such that :
\[
\|\partial_{x}z_{\varepsilon}\zeta_{\varepsilon}-\partial_{x}z\zeta\|_{L^{2}\left(\partial\Omega_{+}\right)}\leq R\|\partial_{x}z_{\varepsilon}\zeta_{\varepsilon}-\partial_{x}z\zeta\|_{H^{1}\left(\Omega_{+}\right)}.
\]

Integrating by parts and using the continuity of $z$ and $\partial_{x}\zeta$,
it is easily deduced that the right-hand side converges to $0$ as
$\varepsilon\to0$. Hence the claimed result follows.
\end{proof}
We can now prove that $\Xi$ is bijective and that a free boundary
condition is satisfied in a weak sense.
\begin{lem}
\label{lem:free_boundary_part_I} Let $\Xi$ be defined as in Lemma
\ref{lem:free_boundary_lem1}. 

Then $\Xi$ is bijective and the functions:
\[
z_{x,-}:t\mapsto\left(\partial_{x}z\right)_{|\partial\Omega_{-}}\left(t,\Xi\left(t\right)\right),
\]
\[
z_{x,+}:t\mapsto\left(\partial_{x}z\right)_{|\partial\Omega_{+}}\left(t,\Xi\left(t\right)\right),
\]
where $\left(\partial_{x}z\right)_{|\partial\Omega_{\pm}}$ are the
traces of $\partial_{x}z$ at each side of $\Gamma$, are in $L_{loc}^{2}\left(\mathbb{R}\right)$
and are equal a.e.. 

Furthermore, if $s>0$, $z_{x,-}\ll0$, and if $s<0$, $z_{x,+}\ll0$.
\end{lem}

\begin{proof}
Assume for instance $s>0$, the other case being similar.

First, we prove the a.e. equality of $z_{x,+}$ and $z_{x,-}$, as
well as the sign of $z_{x,-}$.

Let $\zeta\in\mathcal{C}_{0}^{1}\left(\mathbb{R}^{2}\right)$ be any
non-negative test function and let $\zeta_{\Gamma}:t\mapsto\zeta\left(t,\Xi\left(t\right)\right)$.
By Lemma \ref{lem:usable_weak_formulation}:
\[
\int_{\partial\Omega_{+}}\partial_{x}z\zeta+\int_{\partial\Omega_{-}}\partial_{x}z\zeta=0
\]
 where the unit vector normal to $\partial\Omega_{+}$ is the opposite
of the one normal to $\partial\Omega_{-}$, whence we obtain:
\[
\int_{\mathbb{R}}z_{x,+}\zeta_{\Gamma}=\int_{\mathbb{R}}z_{x,-}\zeta_{\Gamma}.
\]

That is, for a.e. $t$, $z_{x,+}\left(t\right)=z_{x,-}\left(t\right)$,
or, in other words, for a.e. $t\in\mathbb{R}$, $x\mapsto\partial_{x}z\left(t,x\right)$
is continuous. The sign of $z_{x,-}\left(t\right)$ follows directly
from Hopf\textquoteright s lemma applied at the vertex $\left(t,\Xi\left(t\right)\right)$
of the smooth parabolic cylinder $\left(t-1,t\right)\times\left(\Xi\left(t\right),\Xi\left(t\right)+1\right)$. 

Then, it is clear that a continuous unbounded real-valued function
is necessarily surjective, whence $\Xi$ is bijective if and only
if it is injective (or equivalently if and only if it is strictly
monotonic). We are going to prove directly that $\Xi$ is injective.

Differentiating (firstly in the distributional sense) the equation
satisfied by $z$ with respect to $t$ in $\mathbb{R}^{2}\backslash\Gamma$
yields the following regular and linear parabolic equations:
\[
\left\{ \begin{matrix}\partial_{t}\left(\partial_{t}z\right)-\partial_{xx}\left(\partial_{t}z\right)-\alpha g_{1}\left[\frac{z}{\alpha}\right]\partial_{t}z=0 & \mbox{ in }\Omega_{+}\,\\
\partial_{t}\left(\partial_{t}z\right)-d\partial_{xx}\left(\partial_{t}z\right)+dg_{2}\left[-\frac{z}{d}\right]\partial_{t}z=0 & \mbox{ in }\Omega_{-}.
\end{matrix}\right.
\]

Let $x\in\mathbb{R}$. Assume that $\Xi^{-1}\left(\left\{ x\right\} \right)$
is not a singleton. By (large) monotonicity, it is then a segment,
say $\left[t_{1},t_{2}\right]$. Applying classical parabolic regularity
on this system of equations in $\left(t_{1},t_{2}\right)\times\left(x,x+1\right)$
shows that $\partial_{t}z$ is $\mathcal{C}^{1}$ with respect to
$t$ and $\mathcal{C}^{2}$ with respect to $x$ up to $\left(t_{1},t_{2}\right)\times\left\{ x\right\} $.
Moreover, $\partial_{t}z=0$ along $\left(t_{1},t_{2}\right)\times\left\{ x\right\} $.
By classical parabolic regularity and Hopf\textquoteright s lemma,
for any $t\in\left(t_{1},t_{2}\right)$, the right-sided and the left-sided
limit of $\partial_{x}\partial_{t}z\left(t,y\right)$ as $y\to x$
exists and have opposite sign.

Remark that, away from $\Gamma$, the equations satisfied by $z$,
$\partial_{t}z$ and $\partial_{x}z$ suffice to show that $z\in\mathcal{C}^{2}\left(\Omega_{+}\right)\cap\mathcal{C}^{2}\left(\Omega_{-}\right)$.
Therefore Schwarz\textquoteright{} theorem can be applied away from
$\Gamma$.

Thus, for any $t,t'\in\left(t_{1},t_{2}\right)$ and some $\varepsilon>0$
small enough, we get:
\begin{eqnarray*}
\partial_{x}z\left(t,x\pm\varepsilon\right)-\partial_{x}z\left(t',x\pm\varepsilon\right) & = & \int_{t'}^{t}\partial_{t}\partial_{x}z\left(\tau,x\pm\varepsilon\right)\mbox{d}\tau\\
 & = & \int_{t'}^{t}\partial_{x}\partial_{t}z\left(\tau,x\pm\varepsilon\right)\mbox{d}\tau.
\end{eqnarray*}

These two integrals have an opposite strict sign: with respect to
$t$, $\partial_{x}z$ is decreasing on one side of $\left(t_{1},t_{2}\right)\times\left\{ x\right\} $
and increasing on the other. This contradicts the fact that, for a.e.
$t\in\mathbb{R}$, $x\mapsto\partial_{x}z\left(t,x\right)$ is continuous
(see the first step of the proof). Therefore for any $x\in\mathbb{R}$,
$\mathbb{R}\times\left\{ x\right\} \cap\Gamma$ is a singleton, whence
$\Xi$ is bijective.
\end{proof}
\begin{cor}
The function $x\mapsto x-s\Xi^{-1}\left(x\right)$ is continuous and
periodic. Furthermore, 
\[
\left\{ \left(x-s\Xi^{-1}\left(x\right),x\right)\in\mathbb{R}^{2}\ |\ x\in\mathbb{R}\right\} =\varphi^{-1}\left(\left\{ 0\right\} \right).
\]
\end{cor}

\begin{proof}
The periodicity comes from the periodicity with respect to $x$ of
$\varphi$. 
\end{proof}
\begin{rem*}
This corollary confirms that, roughly speaking, the free boundary
is located near the straight line of equation $x=st+\Xi\left(0\right)$.
In other words, $\Xi$ can be represented as the sum of $t\mapsto st$
and a $\frac{L}{s}$-periodic function $\Xi_{per}$.
\end{rem*}
\begin{cor}
\label{cor:strict_monotonicity_SPF} The monotonicity of $z$ with
respect to $t$ is strict. Equivalently, $\varphi$ is decreasing
with respect to $\xi$.
\end{cor}

\begin{proof}
Just apply the strong maximum principle to the equations satisfied
by $\partial_{t}z$ in each component of $\mathbb{R}^{2}\backslash\Gamma$
to get that, in $\mathbb{R}^{2}\backslash\Gamma$, $\partial_{t}z\gg0$
if $s>0$ and $\partial_{t}z\ll0$ if $s<0$, which is sufficient
to obtain strict monotonicity since the measure of $\Gamma$ (as a
measurable subset of $\mathbb{R}^{2}$) is zero.
\end{proof}
Now, thanks to a technique developed by Aronson for the porous media
equation \cite{Aronson_1970}, we are able to prove the continuity
of $\partial_{x}z$.
\begin{lem}
\label{lem:trace_SPF_lem_2} Let $\Xi$ be defined as in Lemma \ref{lem:free_boundary_lem1}
and $z_{x,+}$ and $z_{x,-}$ be defined as in Lemma \ref{lem:free_boundary_part_I}.

If $s>0$ (respectively $s<0$), $z_{x,+}\left(t\right)$ (resp. $z_{x,-}\left(t\right)$)
is actually defined for any $t\in\mathbb{R}$. Moreover, the function
$z_{x,+}$ (resp. $z_{x,-}$) is non-positive and locally uniformly
bounded from below. 
\end{lem}

\begin{proof}
We only prove the result in the case $s>0$, the other one being symmetric. 

Let $t\in\mathbb{R}$ and $x,x'\in\mathbb{R}$ such that $x<x'<\Xi\left(t\right)$.
For any $\tilde{x}\in\left(x,x'\right)$, 
\[
\partial_{xx}z\left(t,\tilde{x}\right)=\partial_{t}z\left(t,\tilde{x}\right)-z\left(t,\tilde{x}\right)f_{1}\left(z\left(t,\tilde{x}\right),\tilde{x}\right).
\]

On one hand, the term $z\left(t,\tilde{x}\right)f_{1}\left(z\left(t,\tilde{x}\right),\tilde{x}\right)$
is bounded from below by $0$ and from above by a constant $R$ independent
on $\tilde{x}$. On the other hand, $\partial_{t}z\left(t,\tilde{x}\right)>0$.
Thus:
\[
\partial_{xx}z\left(t,\tilde{x}\right)\geq-R.
\]

Integrating this inequality, we obtain:
\[
\partial_{x}z\left(t,x'\right)\geq\partial_{x}z\left(t,x\right)-R\left(x'-x\right).
\]

It follows that:
\[
\liminf_{x'\to\Xi\left(t\right)}\partial_{x}z\left(t,x'\right)\geq\partial_{x}z\left(t,x\right)-R\left(\Xi\left(t\right)-x\right),
\]
 and then:
\[
\liminf_{x'\to\Xi\left(t\right)}\partial_{x}z\left(t,x'\right)\geq\limsup_{x\to\Xi\left(t\right)}\partial_{x}z\left(t,x\right).
\]

Hence: 
\[
\lim_{x\to0,x>0}\partial_{x}z\left(t,\Xi\left(t\right)-x\right)
\]
 exists. From the sign of $z$ in $\Omega_{+}$, it is clear that
it is non-positive. Using once more the inequality:
\[
\liminf_{x'\to\Xi\left(t\right)}\partial_{x}z\left(t,x'\right)\geq\partial_{x}z\left(t,x\right)-R\left(\Xi\left(t\right)-x\right)
\]
together with the local boundedness of $\partial_{x}z$ in $\Omega_{+}$,
it follows that the limit is locally uniformly bounded from below.
Finally, it necessarily coincides with $z_{x,+}\left(t\right)$.
\end{proof}
\begin{cor}
$\partial_{x}z\in L^{\infty}\left(\mathbb{R}^{2}\right)$.
\end{cor}

\begin{lem}
We have $\partial_{x}z\in\mathcal{C}_{loc}^{0,\beta}\left(\mathbb{R}^{2}\right)$. 
\end{lem}

\begin{proof}
Let $\zeta\in\mathcal{C}_{0}^{2}\left(\mathbb{R}^{2}\right)$. Choosing
as test functions in the weak formulation in $L_{loc}^{2}$ of:
\[
\sigma\left[z\right]\partial_{t}z-\partial_{xx}z=\eta\left[z\right]
\]
 a sequence of smooth functions converging in $L_{loc}^{2}\left(\mathbb{R}^{2}\right)$
to $\hat{\sigma}\left[z\right]\partial_{x}\zeta$, we obtain:
\[
\int\partial_{t}z\partial_{x}\zeta-\int\hat{\sigma}\left[z\right]\partial_{xx}z\partial_{x}\zeta=\int\hat{\sigma}\left[z\right]\eta\left[z\right]\partial_{x}\zeta.
\]

Remarking the following equalities:
\begin{eqnarray*}
\int\partial_{t}z\partial_{x}\zeta & = & -\int z\partial_{t}\left(\partial_{x}\zeta\right)\\
 & = & -\int z\partial_{x}\left(\partial_{t}\zeta\right)\\
 & = & \int\partial_{x}z\partial_{t}\zeta,
\end{eqnarray*}
\begin{eqnarray*}
\int\hat{\sigma}\left[z\right]\eta\left[z\right]\partial_{x}\zeta & = & -\int\partial_{x}\left(\hat{\sigma}\left[z\right]\eta\left[z\right]\right)\partial_{x}\zeta\\
 & = & -\int\hat{\sigma}\left[z\right]\partial_{x}\left(\eta\left[z\right]\right)\partial_{x}\zeta,
\end{eqnarray*}
(where, by virtue of $\left(\mathcal{H}_{1}\right)$, $\partial_{x}\left(\eta\left[z\right]\right)$
is piecewise-continuous and \textit{a fortiori} is in $L^{\infty}\left(\mathbb{R}^{2}\right)$),
we deduce:
\[
-\int\partial_{x}z\partial_{t}\zeta+\int\hat{\sigma}\left[z\right]\partial_{xx}z\partial_{x}\zeta=\int\hat{\sigma}\left[z\right]\partial_{x}\left(\eta\left[z\right]\right)\zeta.
\]

Hence we can once more apply DiBenedetto\textquoteright s theory \cite{DiBenedetto_19}:
$\partial_{x}z$, which is both in $L^{\infty}\left(\mathbb{R}^{2}\right)$
and in $\mathcal{C}_{loc}\left(\mathbb{R},L_{loc}^{2}\left(\mathbb{R}\right)\right)$
(by classical parabolic estimates similar to those detailed previously
in the proof of Proposition \ref{prop:compactness}), is a locally
bounded weak solution of:
\[
\partial_{t}Z-\partial_{x}\left(\hat{\sigma}\left[z\right]\partial_{x}Z\right)=\hat{\sigma}\left[z\right]\partial_{x}\left(\eta\left[z\right]\right)
\]
and therefore is locally Hölder-continuous indeed.
\end{proof}
\begin{rem*}
Let us explain here why $\partial_{t}z$ is very likely to be continuous
as well (equivalently, $\Xi$ is very likely to be continuously differentiable).
There are in fact some articles related to this free boundary problem
and although none of them is exactly what we need here, they strongly
lead to this conjecture (let us cite for instance Evans \cite{Evans_1982},
Cannon\textendash Yin \cite{Cannon_Yin_1990} and Jensen \cite{Jensen_1978}). 

Roughly speaking, the idea would be to regularize $\left(\mathcal{SPF}\left[s\right]\right)$,
to show the uniqueness of the weak solution of the problem written
in divergence form, to prove thanks to the maximum principle that
the regularization of $\|\left(\partial_{t}z\right)\left(\partial_{x}z\right)^{-1}\|_{L^{\infty}}$
is bounded uniformly with respect to the regularization, to obtain
consequently that $\Xi$ is Lipschitz-continuous, and then to deduce
from Caffarelli\textquoteright s classical results about one\textendash phase
Stefan problems \cite{Caffarelli_197} that $\Xi\in\mathcal{C}^{1}\left(\mathbb{R}\right)$,
whence finally $\partial_{t}z\in\mathcal{C}\left(\mathbb{R}^{2}\right)$. 

Since we do not need such results to conclude this study about pulsating
fronts, we choose not to investigate further in this direction. Nevertheless,
the rigorous proof of the continuity of $\partial_{t}z$ in the more
general framework of weak solutions of $\left(\mathcal{SPF}\left[s\right]\right)$
might be the object of a future follow-up to this article.
\end{rem*}
Let us conclude this subsection with the following corollary, which
takes into account the previous remark and gives an interesting formula.
\begin{cor}
If $d=1$, then $\partial_{t}z,\partial_{xx}z\in\mathcal{C}_{loc}^{0,\beta}\left(\mathbb{R}^{2}\right)$
and $\Xi\in\mathcal{C}^{1}\left(\mathbb{R}\right)$. 

If $d\neq1$ and if $\partial_{t}z\in L^{\infty}\left(\mathbb{R}^{2}\right)$,
then $\Xi\in\mathcal{C}^{1}\left(\mathbb{R}\right)$, $\partial_{t}z\in\mathcal{C}\left(\mathbb{R}^{2}\right)$,
$\hat{\sigma}\left[z\right]\partial_{xx}z\in\mathcal{C}\left(\mathbb{R}^{2}\right)$
and the following equality holds for any $t\in\mathbb{R}$:
\[
\Xi'\left(t\right)=\frac{d}{1-d}\frac{\lim\limits _{\varepsilon\to0,\varepsilon>0}\left(\partial_{xx}z\left(t,\Xi\left(t\right)-\varepsilon\right)-\partial_{xx}z\left(t,\Xi\left(t\right)+\varepsilon\right)\right)}{\partial_{x}z\left(t,\Xi\left(t\right)\right)}.
\]
\end{cor}

\begin{proof}
Regularity in the symmetrical case $d=1$ follows from classical parabolic
regularity. 

Provided $d\neq1$ and global boundedness of $\partial_{t}z$, let
$\varepsilon>0$ small enough so that the implicit function theorem
can be applied at the level set $z^{-1}\left(\left\{ \pm\varepsilon\right\} \right)$.
There exists $\Xi_{\pm\varepsilon}\in\mathcal{C}^{1}\left(\mathbb{R}\right)$
such that $\Xi_{+\varepsilon}\ll\Xi\ll\Xi_{-\varepsilon}$ and such
that:
\[
\Xi_{\pm\varepsilon}'\left(t\right)=-\frac{\partial_{t}z\left(t,\Xi_{\pm\varepsilon}\left(t\right)\right)}{\partial_{x}z\left(t,\Xi_{\pm\varepsilon}\left(t\right)\right)}.
\]

Passing to the limit $\varepsilon\to0$, we deduce that $\Xi$ is
Lipschitz-continuous. Then, by Caffarelli \cite{Caffarelli_197},
$\partial_{t}z,\partial_{xx}z\in\mathcal{C}\left(\overline{\Omega_{+}}\right)\cap\mathcal{C}\left(\overline{\Omega_{-}}\right)$
and $\Xi\in\mathcal{C}^{1}\left(\mathbb{R}\right)$. Thus $\Xi_{\pm\varepsilon}\to\Xi$
in $\mathcal{C}_{loc}^{1}\left(\mathbb{R}\right)$ as $\varepsilon\to0$,
whence $\partial_{t}z$ is moreover continuous at $\Gamma$. Then,
since $\hat{\sigma}\left[z\right]\partial_{xx}z=\partial_{t}z-\hat{\sigma}\left[z\right]\eta\left[z\right]$,
$\hat{\sigma}\left[z\right]\partial_{xx}z$ is continuous in $\mathbb{R}^{2}$
as well. Finally, the formula relating $\Xi'$ to the jump discontinuity
of $\partial_{xx}z$ is easily obtained:
\begin{eqnarray*}
\lim_{\varepsilon\to0,\varepsilon>0}\left(\partial_{xx}z\left(t,\Xi\left(t\right)-\varepsilon\right)-\partial_{xx}z\left(t,\Xi\left(t\right)+\varepsilon\right)\right) & = & \lim_{\varepsilon\to0,\varepsilon>0}\left(\partial_{t}z\left(t,\Xi\left(t\right)-\varepsilon\right)-\frac{1}{d}\partial_{t}z\left(t,\Xi\left(t\right)+\varepsilon\right)\right)\\
 &  & +\lim_{\varepsilon\to0,\varepsilon>0}\left(\eta\left[z\right]\left(t,\Xi\left(t\right)-\varepsilon\right)-\eta\left[z\right]\left(t,\Xi\left(t\right)+\varepsilon\right)\right)\\
 & = & \left(1-\frac{1}{d}\right)\partial_{t}z\left(t,\Xi\left(t\right)\right)\\
 & = & -\left(1-\frac{1}{d}\right)\Xi'\left(t\right)\partial_{x}z\left(t,\Xi\left(t\right)\right).
\end{eqnarray*}
\end{proof}

\subsubsection{Uniqueness}

We are now able to end our characterization.
\begin{thm}
\label{thm:uniqueness_SPF} Let $z_{1}$ and $z_{2}$ be segregated
pulsating fronts with respective speeds $s_{1}\neq0$ and $s_{2}\neq0$
and respective profiles $\varphi_{1}$ and $\varphi_{2}$. 

Then $s_{1}=s_{2}$ and there exists $\tau\in\mathbb{R}$ such that
$\varphi_{1}^{\tau}=\varphi_{2}$, where $\varphi_{1}^{\tau}:\left(\xi,x\right)\mapsto\varphi_{1}\left(\xi-\tau,x\right)$.

In other words, the speed is unique and the profile is unique up to
translation with respect to $\xi$.
\end{thm}

\begin{proof}
We are going to use once more the sliding method. Remark that, up
to the free boundary, this is the most simple case: bistable scalar
equation. Therefore we refer to the proof of Lemma \ref{lem:bounds_speed_pf}
for the details and only point out here some technical differences
due to the presence of the free boundary. 

\textbf{Step 1: existence of a translation of the profile associated
with the highest speed such that it is locally below the other profile.}

Here it is useful to additionally require that, at $\zeta$, the upper
profile is positive (uniformly with respect to $x$) whereas the lower
profile is negative (uniformly as well). This will simplify some arguments
in Steps 2, 3, 4 and 5 since it is now clear that the contact points
$\left(\xi^{\star},x^{\star}\right)$ are necessarily located away
from the free boundary, whence the arguments of the usual sliding
method for regular pulsating fronts (Berestycki\textendash Hamel \cite{Berestycki_Ham_3})
apply straightforwardly.

\textbf{Step 2: up to some extra term, this ordering is global on
the left.}

No new idea here: multiply the upper profile by some $\kappa\geq1$.

\textbf{Step 3: this extra term is actually unnecessary, thanks to
the maximum principle.}

Similarly, there is no new idea here as well and it follows easily
that $\kappa^{\star}=1$.

\textbf{Step 4: up to some (possibly different) extra term, this ordering
is global on the right.}

Thanks to the underlying symmetry due to the bistable structure, the
proof of this step is much simpler here: just change every profile
into its opposite and repeat straightforwardly Step 2.

\textbf{Step 5: this (possibly different) extra term is also unnecessary.}

Similarly, repeat Step 3 to prove that $\kappa^{\star}=1$.

\textbf{Step 6: thanks to the maximum principle again, the speeds
are equal and the profiles are equal up to some translation.}

This is the step which requires additional care because of the free
boundary. To this end, let us introduce some notations. 

We assume that $s_{1}\le s_{2}$. Let:

\[
v_{2}:\left(t,x\right)\mapsto\varphi_{2}\left(x-s_{1}t,x\right),
\]
\[
v_{1}^{\tau^{\star}}:\left(t,x\right)\mapsto\varphi_{1}\left(x-s_{1}t-\tau^{\star},x\right),
\]
\[
v=v_{2}-v_{1}^{\tau^{\star}},
\]
 where $\tau^{\star}$ is defined as in Lemma \ref{lem:bounds_speed_pf}.

At this step of the proof, it is established that $v\geq0$. Let $\mathcal{Z}=v^{-1}\left(\left\{ 0\right\} \right)$.
With the same argument as in Lemma \ref{lem:bounds_speed_pf}, we
can discard the possibility $\mathcal{Z}=\emptyset$. Now there are
basically three cases.
\begin{enumerate}
\item There exists $\left(t^{\star},x^{\star}\right)\in\mathcal{Z}$ such
that $v_{2}\left(t^{\star},x^{\star}\right)>0$. Then by virtue of
the usual parabolic strong maximum principle, $\left(v_{1}^{\tau^{\star}}\right)^{+}=\left(v_{2}\right)^{+}$
in some parabolic cylinder whose final time is $t^{\star}$ and whose
spatial center is $x^{\star}$. Thus $v$ is identically null in this
cylinder, whence by strict monotonicity (see Corollary \ref{cor:strict_monotonicity_SPF})
of $\varphi_{2}$ with respect to $\xi$, $s_{1}=s_{2}$, $v_{2}=z_{2}$
in this cylinder, and then by periodicity of $\varphi_{1}-\varphi_{2}$
with respect to $x$, $\left(v_{1}^{\tau^{\star}}\right)^{+}=v_{2}^{+}$
in $\mathbb{R}^{2}$ and their free boundaries (i.e. zero sets) coincide.
Thus there exists a unique bijection $\Xi$ such that this free boundary
is the graph of $\Xi$. By continuity of $\partial_{x}v_{1}^{\tau^{\star}}$
and $\partial_{x}v_{2}$ (see Proposition \ref{thm:FB_reg_sum_up}),
$\partial_{x}v=0$ on the other side of the free boundary, whence
by virtue of Hopf\textquoteright s lemma the equality $v_{1}^{\tau^{\star}}=v_{2}$
extends everywhere. 
\item There exists $\left(t^{\star},x^{\star}\right)\in\mathcal{Z}$ such
that $v_{2}\left(t^{\star},x^{\star}\right)<0$. Then, by the exact
same argument (this is once more due to the underlying symmetry),
$v_{1}^{\tau^{\star}}=v_{2}$ in $\mathbb{R}^{2}$.
\item Every $\left(t^{\star},x^{\star}\right)\in\mathcal{Z}$ is such that
$v_{1}^{\tau^{\star}}\left(t^{\star},x^{\star}\right)=v_{2}\left(t^{\star},x^{\star}\right)=0$.
Thanks to Hopf\textquoteright s lemma again, this case is actually
contradictory. On one hand, since $\partial_{x}v\in\mathcal{C}\left(\mathbb{R}^{2}\right)$
and $v$ is non-negative non-zero in $\mathbb{R}^{2}$, for any $\left(t^{\star},x^{\star}\right)\in\mathcal{Z}$,
$\partial_{x}v\left(t^{\star},x^{\star}\right)=0$. On the other hand,
although the free boundaries of $v_{1}^{\tau^{\star}}$ and $v_{2}$
are here \textit{a priori} distinct, we can still apply Hopf\textquoteright s
lemma at $\left(t^{\star},x^{\star}\right)$ in a suitable parabolic
cylinder and get a strict sign for $\partial_{x}v\left(t^{\star},x^{\star}\right)$.
\end{enumerate}
\end{proof}
\begin{rem*}
At this point, it would be tempting to notice that this kind of proof
can be easily generalized if one of the two speeds is zero (in this
case, the argument is usually referred to as a \textquotedblleft quenching\textquotedblright{}
or \textquotedblleft blocking\textquotedblright{} argument) and then
to use it to show that a segregated stationary equilibrium cannot
coexist with a segregated pulsating front. Unfortunately, this is
not possible. A segregated stationary equilibrium is \textit{a priori}
a much more general notion than what could be defined as a \textquotedblleft segregated
pulsating front with null speed\textquotedblright{} (the basic reason
being that, when $c_{\infty}=0$, the change of variables $\left(t,x\right)\mapsto\left(x-c_{\infty}t,x\right)$
is not an isomorphism anymore). 

Nevertheless, it is still possible to use some kind of more elaborated
quenching argument, as shows the following theorem.
\end{rem*}
\begin{thm}
\label{thm:SPF_SSE_exclusive} If there exists a segregated pulsating
front, there does not exist a segregated stationary equilibrium. 
\end{thm}

\begin{proof}
Assume that there exist both a segregated pulsating front $z$ with
speed $s\neq0$ and profile $\varphi$ and a segregated stationary
equilibrium $e$. 

Assume for instance that $s>0$ and that $e$ has a smallest zero:
\[
x_{1}=\min e^{-1}\left(\left\{ 0\right\} \right)\in\mathbb{R}.
\]

As in the usual sliding method, we construct (and do not detail these
constructions) $\tau\in\mathbb{R}$ and $\kappa>1$ such that:
\[
\left(\xi,x\right)\mapsto\kappa e\left(\xi\right)-\varphi\left(\xi-\tau,x\right)
\]
 is positive everywhere in $\left(-\infty,x_{1}\right)\times\mathbb{R}$,
with a fixed gap at $\left\{ x_{1}\right\} \times\mathbb{R}$ (constructing
for instance $\tau$ such that $\max\limits _{x\in\overline{C}}\varphi\left(x_{1}-\tau,x\right)=-\frac{da_{2}}{2}$).
Then we define $\kappa^{\star}$ as the infimum of these $\kappa$,
we assume by contradiction that $\kappa^{\star}>1$ and we construct
consequently the first contact point $\left(\xi^{\star},x^{\star}\right)$
with $\xi^{\star}<x_{1}$. By virtue of Proposition \ref{prop:L_infty_estimates_SSE},
$\xi^{\star}>-\infty$. Let $t^{\star}=\frac{x^{\star}-\xi^{\star}}{s}$. 

Notice that there exists a neighborhood of $\left(\xi^{\star},x^{\star}\right)$
such that $\varphi\gg0$ in this neighborhood. Consequently, there
exists $\varepsilon>0$ such that both functions:
\[
x\mapsto\varphi\left(x-st^{\star}-\tau,x\right),
\]
\[
v_{\tau,\kappa^{\star}}:x\mapsto\kappa^{\star}e\left(x+\xi^{\star}-x^{\star}\right)-\varphi\left(x-st^{\star}-\tau,x\right),
\]
are non-negative non-zero everywhere in $\left[x^{\star}-\varepsilon,x^{\star}+\varepsilon\right]$.
Moreover, $v_{\tau,\kappa^{\star}}\left(x^{\star}\right)=0$. Thanks
to the inequality: 
\[
\kappa\eta\left[e\right]\geq\kappa\eta\left[\kappa e\right]\mbox{ in }\left(x^{\star}-\varepsilon,x^{\star}+\varepsilon\right),
\]
 we get:
\[
-\kappa e''\left(x+\xi^{\star}-x^{\star}\right)\geq\kappa\eta\left(\kappa e\left(x+\xi^{\star}-x^{\star}\right),x+\xi^{\star}-x^{\star}\right)\mbox{ for any }x\in\left(x^{\star}-\varepsilon,x^{\star}+\varepsilon\right),
\]
 whence, since $\partial_{t}z>0$, $v_{\tau,\kappa^{\star}}$ satisfies:
\[
-v_{\tau,\kappa^{\star}}''\left(x\right)>q_{\kappa^{\star}}\left(x\right)v_{\tau,\kappa^{\star}}\left(x\right)\mbox{ for any }x\in\left(x^{\star}-\varepsilon,x^{\star}+\varepsilon\right),
\]
 where $q_{\kappa^{\star}}\in L^{\infty}\left(\mathbb{R}\right)$
is defined as:
\[
q_{\kappa^{\star}}:x\mapsto\left\{ \begin{matrix}\frac{\eta\left(\kappa^{\star}e\left(x+\xi^{\star}-x^{\star}\right),x+\xi^{\star}-x^{\star}\right)-\eta\left(\varphi\left(x-st^{\star}-\tau,x\right),x\right)}{v_{\tau,\kappa^{\star}}} & \mbox{ if }v_{\tau,\kappa^{\star}}\left(x\right)\neq0\,\\
1 & \mbox{ if }v_{\tau,\kappa^{\star}}\left(x\right)=0,
\end{matrix}\right.
\]

The function $v_{\tau,\kappa^{\star}}$ is a non-negative non-zero
super-solution of some elliptic problem. Since the elliptic strong
maximum principle contradicts the existence of $\xi^{\star}$, $\kappa^{\star}=1$
indeed.

Repeating the argument near $\xi=+\infty$ with some $\kappa\leq1$
then proves that (up to some increase of $\tau$) $e\left(\xi\right)-\varphi\left(\xi-\tau,x\right)\gg0$
actually holds in $\mathbb{R}^{2}$. Note that in this case, the proof
is simpler, since the negativity of $\varphi$ in $\left(\xi^{\star},+\infty\right)\times\mathbb{R}$
follows from its normalization and monotonicity. We point out that,
\textit{a priori}, there are two cases, depending on the existence
of $\max e^{-1}\left(\left\{ 0\right\} \right)$. But in fact these
two cases do not require different arguments.

Now, just as usual, we can define:
\[
\tau^{\star}=\sup\left\{ \tau\in\mathbb{R}\ |\ e\left(\xi\right)-\varphi\left(\xi-\tau,x\right)\geq0\mbox{ for any }\left(\xi,x\right)\in\mathbb{R}^{2}\right\} .
\]

Assume by contradiction that: 
\[
\min_{[-B,B]\times\mathbb{R}}\left(e\left(\xi\right)-\varphi\left(\xi-\tau^{\star},x\right)\right)>0
\]
 for any $B>0$ such that: 
\[
e\left(B\right)<0,
\]
\[
\min_{x\in\mathbb{R}}\varphi\left(-B-\tau^{\star},x\right)>0.
\]

By continuity, we then obtain for $\tau>\tau^{*}$ close enough, 
\[
\min_{[-B,B]\times\mathbb{R}}\left(e\left(\xi\right)-\varphi\left(\xi-\tau,x\right)\right)>0,
\]
\[
\min_{x\in\mathbb{R}}\varphi\left(-B-\tau,x\right)>0.
\]
It follows from the same type of arguments as those presented at the
beginning of this proof that: 
\[
e\left(\xi\right)-\varphi\left(\xi-\tau,x\right)\gg0\mbox{ in }\left(\mathbb{R}\backslash\left(-B,B\right)\right)\times\mathbb{R},
\]
thus contradicting the maximality of $\tau^{*}$. 

Hence, there exists $B>0$ such that: 
\[
\min_{[-B,B]\times\mathbb{R}}\left(e\left(\xi\right)-\varphi\left(\xi-\tau^{\star},x\right)\right)=0,
\]
 i.e. there exists $\left(\xi^{\star},x^{\star}\right)\in\left[-B,B\right]\times\mathbb{R}$
such that: 
\[
e\left(\xi^{\star}\right)-\varphi\left(\xi^{\star}-\tau^{\star},x^{\star}\right)=0.
\]

Let: 
\[
t^{\star}=\frac{x^{\star}-\xi^{\star}}{s},
\]
\[
v:\left(t,x\right)\mapsto e\left(x+\xi^{\star}-x^{\star}\right)-\varphi\left(x-st-\tau^{\star},x\right)
\]
 and notice that:
\[
v\left(t,x\right)>0\mbox{ for any }\left(t,x\right)\in[t^{\star}-1,t^{\star})\times\mathbb{R},
\]
\[
v\left(t^{\star},x^{\star}\right)=0.
\]

Now, we need to distinguish two cases, as in the proof of Theorem
\ref{thm:uniqueness_SPF}:
\begin{itemize}
\item if $\xi^{\star}\notin e^{-1}\left(\left\{ 0\right\} \right)$, using
the continuity of $v$ and the strong parabolic maximum principle
in some parabolic cylinder $\left[t^{\star}-\varepsilon,t^{\star}\right]\times\left[x^{\star}-\varepsilon,x^{\star}+\varepsilon\right]$
(with a small enough $\varepsilon$ so that the signs of $e\left(x+\xi^{\star}-x^{\star}\right)$
and of $\varphi\left(x-st-\tau^{\star},x\right)$ do not change in
this cylinder), we get a contradiction;
\item if $x^{\star}\in e^{-1}\left(\left\{ 0\right\} \right)$, using the
continuity of $e'$ and $\partial_{x}z$ and Hopf\textquoteright s
lemma at the vertex $\left(t^{\star},x^{\star}\right)$ of the parabolic
cylinder $\left[t^{\star}-1,t^{\star}\right]\times\left[x^{\star},x^{\star}+1\right]$,
we get a contradiction as well.
\end{itemize}
The pair $\left(z,e\right)$ cannot exist.

If $s<0$, we change $v_{\tau,\kappa^{\star}}$ into $-v_{\tau,\kappa^{\star}}$
so that $\partial_{t}z<0$ yields a negative sub-solution and we deduce
similarly $e\left(\xi\right)-\varphi\left(\xi-\tau,x\right)\gg0$.
The end of the proof is carried on similarly.

If $\min e^{-1}\left(\left\{ 0\right\} \right)$ does not exist, then
$\max e^{-1}\left(\left\{ 0\right\} \right)$ does: it suffices to
change the roles of $e$ and $\varphi$, in the sense that now we
have to show that $\varphi\left(\xi-\tau,x\right)-e\left(\xi\right)\gg0$.
Near $\xi=-\infty$, the studied quantity is $\kappa e-\varphi$ with
$\kappa\leq1$, and near $\xi=+\infty$, the studied quantity is $\kappa e-\varphi$
with $\kappa\geq1$. Once $\kappa^{\star}=1$ is established, the
end of the proof is exactly the same.
\end{proof}
\begin{rem*}
The preceding proof only works in the case of constant $a_{1}$ and
$a_{2}$. In the case of non-constant extinction states, this type
of quenching argument does not hold anymore because Proposition \ref{prop:L_infty_estimates_SSE}
is not true anymore and therefore we cannot prove that $\xi^{\star}<-\infty$
when trying to prove that $\kappa^{\star}=1$. We do not know how
to prove the theorem in such a case and we stress that this is really
unsatisfying. Still, we think it is natural to make the following
conjecture.
\end{rem*}
\begin{conjecture}
\label{conj:uniqueness_limit_speed} Theorem \ref{thm:SPF_SSE_exclusive}
still holds true in the non-constant case.
\end{conjecture}

\subsection{Uniqueness of the asymptotic speed}

From now on, $\left(c_{k}\right)_{k>k^{\star}}$ refers to the general
family indexed on $\left(k^{\star},+\infty\right)$ instead of an
\textit{a priori} extracted convergent sequence. In the following,
we will prove that $\left(c_{k}\right)_{k>k^{\star}}$ converges indeed
to $c_{\infty}$ as $k\to+\infty$.
\begin{defn}
We say that \textit{$s\in\mathbb{R}$ satisfies Property $\left(\mathcal{E}\left(d,\alpha,f_{1},f_{2}\right)\right)$}
if one of the following holds: 
\begin{itemize}
\item $s=0$ and there exists a segregated stationary equilibrium;
\item $s\neq0$ and there exists a segregated pulsating front with speed
$s$.
\end{itemize}
The set of all $s\in\mathbb{R}$ satisfying Property $\left(\mathcal{E}\left(d,\alpha,f_{1},f_{2}\right)\right)$
is referred to as $\Sigma_{\left(d,\alpha,f_{1},f_{2}\right)}$. 
\end{defn}

\begin{rem*}
This set does not depend at all on $k^{\star}$. 
\end{rem*}
Following Theorems \ref{thm:uniqueness_SPF} and \ref{thm:SPF_SSE_exclusive},
we deduce the following uniqueness result. 
\begin{cor}
\label{cor:c_infty_element_Sigma} There is at most one $s\in\mathbb{R}$
satisfying Property $\left(\mathcal{E}\left(d,\alpha,f_{1},f_{2}\right)\right)$.
\end{cor}

To conclude about the convergence of the speeds, it suffices to recall
that $c_{\infty}$ satisfies of course Property $\left(\mathcal{E}\left(d,\alpha,f_{1},f_{2}\right)\right)$.
\begin{prop}
\label{prop:limit_c_k_well-defined} The limit at $+\infty$ of the
function $k\mapsto c_{k}$ is well-defined. 
\end{prop}

\begin{rem*}
If $a_{1}$ and $a_{2}$ are non-constant, as explained before, the
quenching argument cannot be used and we do not have the uniqueness
in $\mathbb{R}$ of the elements satisfying Property $\left(\mathcal{E}\left(d,\alpha,f_{1},f_{2}\right)\right)$.
Still, we have the uniqueness in $\mathbb{R}\backslash\left\{ 0\right\} $,
whence in particular the countability of the limit points of $k\mapsto c_{k}$
as $k\to+\infty$. Therefore, using the intermediate value theorem,
we can still prove that the limit of the continuous function $k\mapsto c_{k}$
as $k\to+\infty$ is well-defined. In other words, the convergence
of $\left(c_{k}\right)$ can be proved even without proving Conjecture
\ref{conj:uniqueness_limit_speed}.
\end{rem*}

\subsection{Conclusion of this section}

The function $k\mapsto c_{k}$ converges at $+\infty$. 

If its limit $c_{\infty}$ is non-zero, then both families $\left(\left(u_{1,k},u_{2,k}\right)\right)_{k>k^{\star}}$
and $\left(\left(\varphi_{1,k},\varphi_{2,k}\right)\right)_{k>k^{\star}}$
have a unique limit point (which are respectively the segregated pulsating
front $w$ traveling with speed $c_{\infty}$ and its profile $\phi$),
and therefore the functions $k\mapsto\left(\varphi_{1,k},\varphi_{2,k}\right)$
and $k\mapsto\left(u_{1,k},u_{2,k}\right)$ converge as well as $k\to+\infty$. 

If $c_{\infty}=0$, then $\left(\left(u_{1,k},u_{2,k}\right)\right)_{k>k^{\star}}$
might have multiple limit points, each one of them being a segregated
stationary equilibrium.

\section{\label{sec:Sign} Sign of the asymptotic speed depending on the parameters}

In this final section, we investigate the sign of $c_{\infty}$ as
a function of $\left(d,\alpha\right)$, which is consequently not
considered as fixed anymore ($L>0$ and $\left(f_{1},f_{2}\right)$
are still fixed nevertheless). 

We assume the existence of $D_{exis}\geq0$ such that, for any $d>D_{exis}$
and any $\alpha>0$, $\left(\mathcal{H}_{exis}\right)$ is satisfied. 

Once $\left(d,\alpha\right)\in\left(D_{exis},+\infty\right)\times\left(0,+\infty\right)$
is given, $c_{\infty}$ is naturally defined. If $c_{\infty}\neq0$,
$\phi$ and $w$ are well-defined as well. 
\begin{rem*}
These assumptions are natural in view of the existence result under
the hypothesis $\left(\mathcal{H}_{freq}\right)$ exhibited by the
first author \cite{Girardin_2016}. Indeed, if $\left(\mathcal{H}_{freq}\right)$
is assumed, then it implies $\left(\mathcal{H}_{exis}\right)$ and
the existence of an explicit $D_{exis}$:
\[
D_{exis}=\left\{ \begin{matrix}M_{2}\left(\frac{L}{\pi}-\frac{1}{\sqrt{M_{1}}}\right)^{2} & \mbox{ if }L\sqrt{M_{1}}>\pi\\
0 & \mbox{ if }L\sqrt{M_{1}}\leq\pi.
\end{matrix}\right.
\]
\end{rem*}

\subsection{\label{subsec:Necessary-conditions-zero-speed} Necessary and sufficient
conditions on the parameters for the asymptotic speed to be zero}

Here the idea is to follow what we did in the space-homogeneous case
\cite{Girardin_Nadin_2015} to deduce a free boundary condition satisfied
by any segregated stationary equilibrium. To this end, we need the
following result, which shares some similarities with Proposition
4.1. of Du\textendash Lin \cite{Du_Lin_2010,Du_Lin_2010_er} but is,
on one hand, restricted to the null speeds and, on the other hand,
extended to the space-periodic non-linearities.
\begin{prop}
\label{prop:Du-Lin_extension}Let $x_{0}\in\mathbb{R}$ and $f:[0,+\infty)\times\mathbb{R}\to\mathbb{R}$,
periodic with respect to $x$ and satisfying $\left(\mathcal{H}_{1}\right)$,
$\left(\mathcal{H}_{2}\right)$ and $\left(\mathcal{H}_{3}\right)$.
The following problem:
\[
\left\{ \begin{matrix}-z''=zf\left[z\right] & \mbox{ in }\left(x_{0},+\infty\right)\\
z\left(x_{0}\right)=0
\end{matrix}\right.
\]
 admits a unique non-negative non-zero solution $z_{x_{0},f}\in\mathcal{C}^{2}\left([x_{0},+\infty)\right)$. 

Furthermore, the function 
\[
\Theta:\left(x_{0},f\right)\mapsto z_{x_{0},f}'\left(x_{0}\right)
\]
 (that is the right-sided derivative of $z_{x_{0},f}$ at $x_{0}$)
satisfies:
\begin{enumerate}
\item $\Theta\gg0$;
\item $\Theta$ is continuous with respect to the canonical topology of
$\mathbb{R}\times\mathcal{C}^{1}\left(\mathbb{R}^{2},\mathbb{R}\right)$;
\item $\Theta$ is periodic with respect to its first variable;
\item for any $\kappa>0$, 
\[
\Theta\left(x_{0},\left(z,x\right)\mapsto f\left(\frac{z}{\kappa},x\right)\right)=\kappa\Theta\left(x_{0},f\right).
\]
\end{enumerate}
\end{prop}

\begin{proof}
Firstly, let us point out that Du\textendash Lin\textquoteright s
proposition \cite{Du_Lin_2010,Du_Lin_2010_er} is readily extended
to generic ``KPP''-type non-linearities which do not depend on the
spatial variable. We do not detail this extension here.

Thus, let $\overline{f}:z\mapsto\max\limits _{y\in\overline{C}}f\left(z,y\right)$.
It can be checked that $z\mapsto z\overline{f}\left[z\right]$ is
indeed a KPP-type non-linearity (mostly, it reduces to the proof of
the fact that $\overline{f}$ is decreasing and negative after some
fixed value). Then, let $\overline{z}$ be the solution given by (the
aforementioned extension of) Du\textendash Lin\textquoteright s proposition
of:
\[
\left\{ \begin{matrix}-z''\left(x\right)=z\overline{f}\left[z\right] & \mbox{in }\left(x_{0},+\infty\right)\\
z\left(x_{0}\right)=0.
\end{matrix}\right.
\]

Similarly, let $\underline{f}:z\mapsto\min\limits _{y\in\overline{C}}f\left(z,y\right)$
and $\underline{z}$ be the solution of:
\[
\left\{ \begin{matrix}-z''\left(x\right)=z\underline{f}\left[z\right] & \mbox{in }\left(x_{0},+\infty\right)\\
z\left(x_{0}\right)=0.
\end{matrix}\right.
\]

We intend to prove that $\overline{z}$ and $\underline{z}$ form
an ordered pair of super- and sub-solution for the problem at hand. 

Let $a$ be the positive constant given by $\left(\mathcal{H}_{3}\right)$
such that $f\left(a,x\right)=0$ for all $x\in\overline{C}$. By standard
elliptic estimates, 
\[
\lim_{+\infty}\overline{z}=\lim_{+\infty}\underline{z}=a.
\]

By Du\textendash Lin\textquoteright s proposition, we know that $\overline{z}'\left(x_{0}\right)$
and $\underline{z}'\left(x_{0}\right)$ (understood as right-sided
derivatives) are finite, whence there exists $\kappa>0$ such that:
\[
\kappa\overline{z}-\underline{z}\geq0\mbox{ in }\left(x_{0},+\infty\right).
\]

Let:
\[
\kappa^{\star}=\inf\left\{ \kappa>0\ |\ \kappa\overline{z}-\underline{z}\gg0\mbox{ in }\left(x_{0},+\infty\right)\right\} 
\]
 and assume by contradiction that $\kappa^{\star}>1$. We can fix
a sequence $\left(\kappa_{n}\right)_{n\in\mathbb{N}}\in\left(1,\kappa^{\star}\right)^{\mathbb{N}}$
which converges to $\kappa^{\star}$ from below. There exists a sequence
$\left(x_{n}\right)_{n\in\mathbb{N}}\in\left(x_{0},+\infty\right)^{\mathbb{N}}$
such that: 
\[
\left(\kappa_{n}\overline{z}-\underline{z}\right)\left(x_{n}\right)<0.
\]
Since $\lim\limits _{+\infty}\left(\kappa_{n}\overline{z}-\underline{z}\right)=\left(\kappa_{n}-1\right)a>0$,
the sequence $\left(x_{n}\right)_{n\in\mathbb{N}}$ is bounded and
then convergent up to extraction.

If $x_{\infty}$ is the limit of $\left(x_{n}\right)$, then by continuity:
\[
\kappa^{\star}\overline{z}\left(x_{\infty}\right)=\underline{z}\left(x_{\infty}\right).
\]

Now, remarking that:
\[
\kappa^{\star}\overline{z}\overline{f}\left[\overline{z}\right]\geq\kappa^{\star}\overline{z}\overline{f}\left[\kappa^{\star}\overline{z}\right]
\]
 by monotonicity of $\overline{f}$, it follows by Lipschitz-continuity
of $\overline{f}$ that $\kappa^{\star}\overline{z}-\underline{z}$
is a positive super-solution of some linear elliptic problem which
vanishes at $x_{\infty}$. Provided $x_{\infty}\neq x_{0}$, this
contradicts the elliptic strong minimum principle and the strict ordering
at $+\infty$.

But if $x_{\infty}=x_{0}$, then Hopf\textquoteright s lemma implies
that: 
\[
\left(\kappa^{\star}\overline{z}-\underline{z}\right)'\left(x_{0}\right)>0.
\]

From this inequality, the optimality of $\kappa^{\star}$ is easily
contradicted. 

Hence $\kappa^{\star}=1$, that is $\overline{z}$ and $\underline{z}$
are indeed a pair of ordered super- and sub-solution of the problem.
Since $f$ depends on $x$ (the special case of $f$ constant with
respect to $x$, that is Du\textendash Lin\textquoteright s case,
can be discarded here without loss of generality), they are not solutions
themselves, whence their ordering is strict:
\[
\underline{z}\ll\overline{z}\mbox{ in }\left(x_{0},+\infty\right).
\]

Finally, by virtue of classical existence\textendash comparison results
for semi-linear elliptic problems, there exists a solution of the
problem $z_{x_{0},f}$ satisfying furthermore: 
\[
\underline{z}\ll z_{x_{0},f}\ll\overline{z}.
\]

The uniqueness of $z_{x_{0},f}$ follows from similar arguments. 

The positivity of $\Theta$ easily follows from $z_{x_{0},f}\gg\underline{z}$.
Its continuity comes from the uniqueness of $z_{x_{0},f}$ and classical
compactness arguments. Its periodicity with respect to $x$ comes
from the uniqueness of $z_{x_{0},f}$ and the periodicity of $f$
with respect to $x$. The last property comes from the following easy
fact. Let $\kappa>0$ and $Z=\kappa z_{x_{0},f}$. It is easily verified
that:
\[
-Z''=Zf\left[\frac{Z}{\kappa}\right]\mbox{ in }\left(x_{0},+\infty\right)
\]
 and then by uniqueness $Z=z_{x_{0},f_{\kappa}}$ where $f_{\kappa}:\left(z,x\right)\mapsto f\left(\frac{z}{\kappa},x\right)$.
\end{proof}
Before going any further, we recall that it suffices to choose different
normalization sequences to deduce that, if $c_{\infty}=0$, there
exists at least one segregated stationary equilibrium $e_{1}$ satisfying:
\[
\inf e_{1}^{-1}\left(\left(-\infty,0\right)\right)>-\infty
\]
 and at least one segregated stationary equilibrium $e_{2}$ satisfying:
\[
\sup e_{2}^{-1}\left(\left(0,+\infty\right)\right)<+\infty.
\]

If $c_{\infty}=0$, we define consequently $x_{1}=\min e_{1}^{-1}\left(\left\{ 0\right\} \right)$
and $x_{2}=\max e_{2}^{-1}\left(\left\{ 0\right\} \right)$. Recall
that, without loss of generality, we can assume that $\left(x_{1},x_{2}\right)\in[0,L)^{2}$.
\begin{lem}
Let $\left(d,\alpha\right)\in\left(D_{exis},+\infty\right)\times\left(0,+\infty\right)$,
$f_{1,x_{1}}:\left(z,x\right)\mapsto f_{1}\left(z,2x_{1}-x\right)$
and $\Theta$ be defined as in Proposition \ref{prop:Du-Lin_extension}.
Assume $c_{\infty}=0$. 

Then:
\[
\alpha\Theta\left(x_{1},f_{1,x_{1}}\right)\geq d\Theta\left(x_{1},\frac{1}{d}f_{2}\right),
\]
\[
\alpha\Theta\left(x_{2},f_{1,x_{2}}\right)\leq d\Theta\left(x_{2},\frac{1}{d}f_{2}\right).
\]
\end{lem}

\begin{proof}
We prove the first inequality, the second one being proved similarly
(using $e_{2}$ instead of $e_{1}$). 

First, if: 
\[
e_{1}^{-1}\left(\left\{ 0\right\} \right)\backslash\left\{ x_{1}\right\} =\emptyset,
\]
 then $e_{1}$ has a unique zero. Now, consider the problems satisfied
by the functions: 
\[
z_{1}:x\mapsto e_{1}^{+}\left(2x_{1}-x\right),
\]
\[
z_{2}:x\mapsto e_{1}^{-}\left(x\right).
\]

It is clear that: 
\[
\left(z_{1},z_{2}\right)=\left(z_{x_{1},\left(z,x\right)\mapsto f_{1}\left(\frac{z}{\alpha},2x_{1}-x\right)},z_{x_{1},\left(z,x\right)\mapsto\frac{1}{d}f_{2}\left(\frac{z}{d},x\right)}\right).
\]

Since $e_{1}\in\mathcal{C}^{2}\left(\mathbb{R}\right)$, $z_{1}'\left(x_{1}^{+}\right)=z_{2}'\left(x_{1}^{+}\right)$
is necessary. From the relations:
\[
\Theta\left(x_{1},\left(z,x\right)\mapsto f_{1}\left(\frac{z}{\alpha},2x_{1}-x\right)\right)=\alpha\Theta\left(x_{1},f_{1,x_{1}}\right),
\]
\[
\Theta\left(x_{1},\left(z,x\right)\mapsto\frac{1}{d}f_{2}\left(\frac{z}{d},x\right)\right)=d\Theta\left(x_{1},\frac{1}{d}f_{2}\right),
\]
 we see that we are in the case of equality.

Next, if:
\[
e_{1}^{-1}\left(\left\{ 0\right\} \right)\backslash\left\{ x_{1}\right\} \neq\emptyset,
\]
 then let: 
\[
y_{1}=\min e_{1}^{-1}\left(\left\{ 0\right\} \right)\backslash\left\{ x_{1}\right\} .
\]

Clearly, $z_{3}=\left(e_{1}^{-}\right)_{|\left(x_{1},y_{1}\right)}$
is the unique non-negative non-zero solution of:
\[
\left\{ \begin{matrix}-dz''=zf_{2}\left[\frac{z}{d}\right] & \mbox{in }\left(x_{1},y_{1}\right)\,\\
z\left(x_{1}\right)=z\left(y_{1}\right)=0.
\end{matrix}\right.
\]

Now it can be easily verified that $z_{3}$ is a sub-solution for
the problem satisfied by $z_{x_{1},\left(z,x\right)\mapsto\frac{1}{d}f_{2}\left(\frac{z}{d},x\right)}$.
The inequality follows.
\end{proof}
\begin{rem*}
We explained previously that, if $\left(\mathcal{H}_{freq}\right)$
\cite{Girardin_2016} is assumed, each segregated stationary equilibrium
has a unique zero $x_{e}$. In such a case, we have equality:
\[
\alpha\Theta\left(x_{e},f_{1,x_{e}}\right)=d\Theta\left(x_{e},\frac{1}{d}f_{2}\right).
\]
\end{rem*}
Let $\left(d,\alpha\right)\in\left(0,+\infty\right)^{2}$. With the
same notations as before, we define the following sets: 
\[
X_{\left(d,\alpha\right)}^{+}=\left\{ x\in[0,L)\ |\ \alpha\Theta\left(x,f_{1,x}\right)\geq d\Theta\left(x,\frac{1}{d}f_{2}\right)\right\} ,
\]
\[
X_{\left(d,\alpha\right)}^{-}=\left\{ x\in[0,L)\ |\ \alpha\Theta\left(x,f_{1,x}\right)\leq d\Theta\left(x,\frac{1}{d}f_{2}\right)\right\} .
\]
Clearly, from the preceding corollary, if $c_{\infty}=0$,
\[
\left\{ \begin{matrix}X_{\left(d,\alpha\right)}^{+}\neq\emptyset\,\\
X_{\left(d,\alpha\right)}^{-}\neq\emptyset.
\end{matrix}\right.
\]
\begin{prop}
Let $\left(d,\alpha\right)\in\left(0,+\infty\right)^{2}$, $f_{1,x}:\left(z,y\right)\mapsto f_{1}\left(z,2x-y\right)$
, $\Theta$ be defined as in Proposition \ref{prop:Du-Lin_extension}
and :
\[
A_{d}:x\mapsto\frac{d\Theta\left(x,\frac{1}{d}f_{2}\right)}{\Theta\left(x,f_{1,x}\right)}.
\]

The function $A_{d}$ is continuous, positive and periodic, does not
depend on $\alpha$ and satisfies the following properties.
\begin{itemize}
\item If there exists $x\in X_{\left(d,\alpha\right)}^{+}$, then $\alpha\geq A_{d}\left(x\right)$.
\item If there exists $x\in X_{\left(d,\alpha\right)}^{-}$, then $\alpha\leq A_{d}\left(x\right)$.
\item It has a global minimum and a global maximum .
\end{itemize}
Consequently, provided $d>D_{exis}$, $\alpha\in\left[\min A_{d},\max A_{d}\right]$
if and only if $c_{\infty}=0$. 
\end{prop}

\begin{proof}
Everything is straightforward apart maybe the following implication:
if $\alpha\in\left[\min A_{d},\max A_{d}\right]$, then $c_{\infty}=0$.
In fact, if there exists $x_{e}\in[0,L)$ such that $\alpha=A\left(x_{e}\right)$,
then the following function:
\[
z:y\mapsto\left\{ \begin{matrix}z_{x_{e},\left(z,x\right)\mapsto f_{1}\left(\frac{z}{\alpha},2x_{e}-x\right)}\left(2x_{e}-y\right) & \mbox{if }y<x_{e},\\
-z_{x_{e},\left(z,x\right)\mapsto\frac{1}{d}f_{2}\left(\frac{z}{d},x\right)}\left(y\right) & \mbox{if }y\geq x_{e},
\end{matrix}\right.
\]
 is a segregated stationary equilibrium, which implies by uniqueness
(see Theorem \ref{thm:SPF_SSE_exclusive}) that $c_{\infty}=0$.
\end{proof}
\begin{rem*}
The preceding proposition characterizes sharply $\left\{ \alpha>0\ |\ c_{\infty}=0\right\} $.
Moreover, it also gives an implicit characterization of the diffusion
rates such that $c_{\infty}=0$. With this in mind, understanding
whether $A_{d}$ is constant or not would be of great interest.

Let us recall that if $a_{1}$ and $a_{2}$ are not constant, we do
not know how to prove Theorem \ref{thm:SPF_SSE_exclusive}. Therefore
in such a case the preceding sharpness is lost and we might still
have a non-zero $c_{\infty}$ for some $\alpha\in\left[\min A_{d},\max A_{d}\right]$.
This pathological situation seems highly unlikely (recall Conjecture
\ref{conj:uniqueness_limit_speed}).
\end{rem*}
From this result, we can also deduce an explicit estimate for the
range of parameters $\left(d,\alpha\right)$, as indicated by the
following statement.
\begin{prop}
\label{prop:alpha_squared_over_d_SSE_estimate}Let $\Lambda\subset\mathbb{R}^{2}$
be the following set: 
\[
\left\{ \left(d,\alpha\right)\in\left(0,+\infty\right)^{2}\ |\ X_{\left(d,\alpha\right)}^{+}\neq\emptyset\mbox{ and }X_{\left(d,\alpha\right)}^{-}\neq\emptyset\right\} .
\]

There exists $\underline{r}>0$ and $\overline{r}\geq\underline{r}$,
defined by formulas $\left(\mathfrak{F}_{\underline{r}}\right)$ and
$\left(\mathfrak{F}_{\overline{r}}\right)$ which only depend on $\left(f_{1},f_{2}\right)$,
such that, for any $\left(d,\alpha\right)\in\Lambda$, 
\[
\underline{r}\leq\frac{\alpha^{2}}{d}\leq\overline{r}.
\]
\end{prop}

\begin{rem*}
Although these estimates do not depend on $d$, they are also less
precise than the previous statement. Indeed, we will see in the course
of the proof that, for any $d>0$:
\[
\sqrt{\underline{r}d}\leq\min A_{d},
\]
\[
\max A_{d}\leq\sqrt{\overline{r}d},
\]
 and furthermore it should be expected that these inequalities are
actually strict. Thus the interest of this proposition lies mostly
in the fact that $\underline{r}$ and $\overline{r}$ do not depend
on $d$.
\end{rem*}
\begin{proof}
Recalling from Proposition \ref{prop:Du-Lin_extension} the definition
of $z_{x_{0},f}$, we define for any $d>0$ and any $y\in\overline{C}$
the following functions: 
\[
z_{1,y}=z_{y,f_{1,y}},
\]
\[
z_{2,y}:x\mapsto z_{y,\frac{1}{d}f_{2}}\left(\sqrt{d}x+y\right).
\]

Most importantly, $z_{2,y}$ satisfies:
\[
\left\{ \begin{matrix}-z_{2,y}''\left(x\right)=z_{2,y}\left(x\right)f_{2}\left(z_{2,y}\left(x\right),\sqrt{d}x+y\right) & \mbox{for any }x\in\left(0,+\infty\right),\\
z_{2,y}\left(0\right)=0.
\end{matrix}\right.
\]

Let $\overline{f_{2}}:z\mapsto\max\limits _{x\in\overline{C}}f_{2}\left(z,x\right)$
and $\overline{z}$ be the solution of:
\[
\left\{ \begin{matrix}-z''=z\overline{f_{2}}\left[z\right] & \mbox{in }\left(0,+\infty\right)\\
z\left(0\right)=0.
\end{matrix}\right.
\]

Similarly, let $\underline{f_{2}}:z\mapsto\min\limits _{x\in\overline{C}}f_{2}\left(z,x\right)$
and $\underline{z}$ be the solution of:
\[
\left\{ \begin{matrix}-z''=z\underline{f_{2}}\left[z\right] & \mbox{in }\left(0,+\infty\right)\\
z\left(0\right)=0.
\end{matrix}\right.
\]

It can easily be checked (see the proof of Proposition \ref{prop:Du-Lin_extension})
that the solutions $\underline{z}$ and $\overline{z}$ form a pair
of sub-solution and super-solution for the problem satisfied by $z_{2,y}$.
By uniqueness, $\underline{z}\leq z_{2,y}\leq\overline{z}$. Since
$\sqrt{d}\Theta\left(y,\frac{1}{d}f_{2}\right)=z_{2,y}'\left(0\right)$,
consequently: 
\[
\underline{z}'\left(0\right)\leq\sqrt{d}\Theta\left(y,\frac{1}{d}f_{2}\right)\leq\overline{z}'\left(0\right).
\]

Then, for any $\left(d,\alpha\right)\in\Lambda$, we deduce from the
preceding estimate and from the definitions of $X_{\left(d,\alpha\right)}^{+}$
and $X_{\left(d,\alpha\right)}^{-}$ that there exists $\left(x_{1},x_{2}\right)\in[0,L)^{2}$
such that:
\[
\alpha\Theta\left(x_{1},f_{1,x_{1}}\right)\geq\sqrt{d}\underline{z}'\left(0\right),
\]
\[
\alpha\Theta\left(x_{2},f_{1,x_{2}}\right)\leq\sqrt{d}\overline{z}'\left(0\right).
\]

The conclusion follows from the following definitions\label{formulas}:
\[
\underline{r}=\left(\frac{\underline{z}'\left(0\right)}{\max\limits _{x\in\overline{C}}\Theta\left(x,f_{1,x}\right)}\right)^{2},\quad\left(\mathfrak{F}_{\underline{r}}\right)
\]
\[
\overline{r}=\left(\frac{\overline{z}'\left(0\right)}{\min\limits _{x\in\overline{C}}\Theta\left(x,f_{1,x}\right)}\right)^{2}.\quad\left(\mathfrak{F}_{\overline{r}}\right)
\]
\end{proof}
\begin{cor}
Assume that, for any $i\in\left\{ 1,2\right\} $, $f_{i}$ has the
particular form $\left(u,x\right)\mapsto\mu_{i}\left(x\right)\left(1-u\right)$
with $\mu_{i}\in\mathcal{C}_{per}^{1}\left(\mathbb{R}\right)$, $\mu_{i}\gg0$. 

Then:
\[
\frac{\min\limits _{\overline{C}}\left(\mu_{2}\right)}{\max\limits _{\overline{C}}\left(\mu_{1}\right)}\leq\underline{r}\leq\overline{r}\leq\frac{\max\limits _{\overline{C}}\left(\mu_{2}\right)}{\min\limits _{\overline{C}}\left(\mu_{1}\right)}.
\]
\end{cor}

\begin{proof}
In such a case, the functions $\overline{f_{2}}$ and $\underline{f_{2}}$
defined in the proof of Proposition \ref{prop:alpha_squared_over_d_SSE_estimate}
reduce to:
\[
\overline{f_{2}}:z\mapsto\max\limits _{\overline{C}}\left(\mu_{2}\right)\left(1-z\right),
\]
\[
\underline{f_{2}}:z\mapsto\min\limits _{\overline{C}}\left(\mu_{2}\right)\left(1-z\right).
\]
 Define analogously:
\[
\overline{f_{1}}:z\mapsto\max\limits _{\overline{C}}\left(\mu_{1}\right)\left(1-z\right),
\]
\[
\underline{f_{1}}:z\mapsto\min\limits _{\overline{C}}\left(\mu_{1}\right)\left(1-z\right).
\]

Denoting the functions $\overline{z}$ and $\underline{z}$ defined
in the proof of Proposition \ref{prop:alpha_squared_over_d_SSE_estimate}
as $\overline{z}_{2}$ and $\underline{z}_{2}$, the definitions of
$\underline{r}$ and $\overline{r}$ read: 
\[
\underline{r}=\left(\frac{\underline{z}_{2}'\left(0\right)}{\max\limits _{x\in\overline{C}}\Theta\left(x,f_{1,x}\right)}\right)^{2},
\]
\[
\overline{r}=\left(\frac{\overline{z}_{2}'\left(0\right)}{\min\limits _{x\in\overline{C}}\Theta\left(x,f_{1,x}\right)}\right)^{2}.
\]

Defining analogously the functions $\overline{z}_{1}$ and $\underline{z}_{1}$,
we obtain by a super- and sub-solution argument similar to that of
Proposition \ref{prop:alpha_squared_over_d_SSE_estimate} the following
estimates:
\[
\underline{z}_{1}'\left(0\right)\leq\min\limits _{x\in\overline{C}}\Theta\left(x,f_{1,x}\right)\leq\max\limits _{x\in\overline{C}}\Theta\left(x,f_{1,x}\right)\leq\overline{z}_{1}'\left(0\right),
\]
which lead subsequently to:
\[
\underline{r}\geq\left(\frac{\underline{z}_{2}'\left(0\right)}{\overline{z}_{1}'\left(0\right)}\right)^{2},
\]
\[
\overline{r}\leq\left(\frac{\overline{z}_{2}'\left(0\right)}{\underline{z}_{1}'\left(0\right)}\right)^{2}.
\]

Now let us determine $\Theta\left(0,z\mapsto r\left(1-z\right)\right)$
for any constant $r>0$. Multiplying the equality satisfied by $z=z_{0,z\mapsto r\left(1-z\right)}$
by $z'$, we find:
\[
-\left(\frac{\left(z'\right)^{2}}{2}\right)'=r\left(\frac{z^{2}}{2}\right)'-r\left(\frac{z^{3}}{3}\right)'.
\]
Integrating between $0$ and $+\infty$, it follows $\left(z'\left(0\right)\right)^{2}=\frac{r}{6}$,
that is:
\[
\Theta\left(0,z\mapsto r\left(1-z\right)\right)=\sqrt{\frac{r}{6}}.
\]

Applying this equality with $r=\max\limits _{\overline{C}}\left(\mu_{2}\right)$,
$r=\min\limits _{\overline{C}}\left(\mu_{2}\right)$, $r=\max\limits _{\overline{C}}\left(\mu_{1}\right)$
and $r=\min\limits _{\overline{C}}\left(\mu_{1}\right)$, the claimed
estimates for $\underline{r}$ and $\overline{r}$ follow directly.
\end{proof}
Thanks to the existence of $\underline{r}$ and $\overline{r}$, we
now know that the quantity $\frac{\alpha^{2}}{d}$ plays a particular
role (and this is obviously reminiscent of the space-homogeneous case
\cite{Girardin_Nadin_2015}). Therefore, we also state the following
(immediate) proposition.
\begin{prop}
For any $d\in\left(0,+\infty\right)$, let:
\[
\mathcal{R}_{d}^{0}=\left[\frac{\left(\min A_{d}\right)^{2}}{d},\frac{\left(\max A_{d}\right)^{2}}{d}\right].\quad\left(\mathfrak{F}_{\mathcal{R}^{0}}\right)
\]

The set $\mathcal{R}_{d}^{0}$ is a non-empty, closed, subinterval
of $\left[\underline{r},\overline{r}\right]$.

Assume moreover that $d>D_{exis}$. Then $c_{\infty}=0$ if and only
if $\frac{\alpha^{2}}{d}\in\mathcal{R}_{d}^{0}$.
\end{prop}

\begin{rem*}
Once more, in the case of non-constant $a_{1}$ and $a_{2}$, one
implication is lacking, but proving Conjecture \ref{conj:uniqueness_limit_speed}
would be sufficient to recover it. 

The length of $\mathcal{R}_{d}^{0}$ is a very interesting open question
(which is obviously equivalent to that of the constancy of $A_{d}$).
Recall that in the space-homogeneous case \cite{Girardin_Nadin_2015},
$\mathcal{R}_{d}^{0}=\left\{ \frac{f_{2}\left[0\right]}{f_{1}\left[0\right]}\right\} $
is a singleton which does not depend on $d$.
\end{rem*}

\subsection{Sign of a non-zero asymptotic speed}
\begin{prop}
\label{prop:sign_non-zero_speed} Let $\left(d,\alpha\right)\in\left(0,+\infty\right)^{2}$.
Let $z$ be a segregated pulsating front with speed $s\neq0$ and
profile $\varphi$. 

Then $s$ has the sign of: 
\[
\int_{0}^{L}\int_{-da_{2}}^{\alpha a_{1}}\eta\left(z,x\right)\mbox{d}z\mbox{d}x=\int_{0}^{L}\left(\alpha^{2}\int_{0}^{a_{1}}zf_{1}\left(z,x\right)\mbox{d}z-d\int_{0}^{a_{2}}zf_{2}\left(z,x\right)\mbox{d}z\right)\mbox{d}x.
\]
\end{prop}

\begin{rem*}
In view of well-known results about bistable scalar traveling waves,
and more recently pulsating fronts (see for instance Ding\textendash Hamel\textendash Zhao
\cite{Ding_Hamel_Zhao}), such a result was to be expected. 

It could be tempting to try to get rid of the \textit{a priori} condition
$s\neq0$ and to show that the existence of a segregated stationary
equilibrium implies: 
\[
\int_{C}\int_{-da_{2}}^{\alpha a_{1}}\eta\left(z,x\right)\mbox{d}z\mbox{d}x=0.
\]
But Zlatos \cite{Zlatos_2015} showed on the contrary that it is possible
to build counter-examples of pure bistable non-linearities $F$ of
positive integral such that: 
\[
\partial_{t}z-\partial_{xx}z=F\left[z\right]
\]
 does not admit any transition front with non-zero speed. Therefore
we do not investigate further in this direction.
\end{rem*}
\begin{proof}
We have justified previously that in the equation $\left(\mathcal{SPF}\left[s\right]\right)$,
every term ($\mbox{div}\left(E\nabla\varphi\right)$, $\partial_{\xi}\varphi$
and $\eta\left[\varphi\right]$) is well-defined in $L_{loc}^{2}\left(\mathbb{R}^{2}\right)$.
Thus we consider the test function $\partial_{\xi}\varphi\mathbf{1}_{\left[-B,B\right]\times\overline{C}}\in L_{loc}^{2}\left(\mathbb{R}^{2}\right)$
for some large enough $B>0$. By large, we mean here that we assume
the following: 
\[
\min_{x\in\overline{C}}\varphi\left(\xi,x\right)>0\mbox{ for any }\xi<-B,
\]
\[
\max_{x\in\overline{C}}\varphi\left(\xi,x\right)<0\mbox{ for any }\xi>B.
\]

Hence the subset of the free boundary $\left\{ \left(\xi,x\right)\in\mathbb{R}\times\overline{C}\ |\ \varphi\left(\xi,x\right)=0\right\} $
is included in $\left(-B,B\right)\times\overline{C}$. 

Multiplying $\left(\mathcal{SPF}\left[s\right]\right)$ by $\partial_{\xi}\varphi$
and integrating over $\left(-B,B\right)\times\overline{C}$ yield:
\[
\int_{-B}^{B}\int_{0}^{L}\mbox{div}\left(E\nabla\varphi\right)\partial_{\xi}\varphi+s\int_{-B}^{B}\int_{0}^{L}\sigma\left[\varphi\right]\left(\partial_{\xi}\varphi\right)^{2}=-\int_{-B}^{B}\int_{0}^{L}\eta\left[\varphi\right]\partial_{\xi}\varphi.
\]

First, by change of variable, Lipschitz-continuity of the free boundary
(see Proposition \ref{thm:FB_reg_sum_up}) and definition of $\eta$:
\begin{eqnarray*}
-\int_{0}^{L}\int_{-B}^{B}\eta\left[\varphi\right]\partial_{\xi}\varphi & = & \int_{0}^{L}\int_{\varphi\left(B,x\right)}^{\varphi\left(-B,x\right)}\eta\left(z,x\right)\mbox{d}z\mbox{d}x\\
 & = & \int_{0}^{L}\left(\alpha^{2}\int_{0}^{\nicefrac{\varphi\left(-B,x\right)}{\alpha}}zf_{1}\left(z,x\right)\mbox{d}z-d\int_{0}^{-\nicefrac{\varphi\left(B,x\right)}{d}}zf_{2}\left(z,x\right)\mbox{d}z\right)\mbox{d}x.
\end{eqnarray*}

Then, since we do not know that $\partial_{\xi}\varphi$ is continuous,
the term $\int_{-B}^{B}\int_{0}^{L}\mbox{div}\left(E\nabla\varphi\right)\partial_{\xi}\varphi$
is dealt with a standard mollification procedure. There exists a sequence
of non-negative non-zero mollifiers $\left(\theta_{n}\right)_{n\in\mathbb{N}}\in\mathcal{D}\left(\mathbb{R}\right)$.
For any $n\in\mathbb{N}$, let: 
\[
\varphi_{n}:\left(\xi,x\right)\mapsto\int\varphi\left(\xi-\zeta,x\right)\theta_{n}\left(\zeta\right)\mbox{d}\zeta.
\]

On one hand, for any $n\in\mathbb{N}$, it is clear that all the terms
$\partial_{\xi\xi}\varphi_{n}$, $\partial_{xx}\varphi_{n}$, $\partial_{\xi x}\varphi_{n}$
are classically defined. By periodicity and integration by parts,
we easily obtain:
\[
\int_{-B}^{B}\int_{0}^{L}\mbox{div}\left(E\nabla\varphi_{n}\right)\partial_{\xi}\varphi_{n}=\frac{1}{2}\int_{0}^{L}\left(\left[\left(\partial_{\xi}\varphi_{n}\right)^{2}\left(\xi,x\right)\right]_{-B}^{B}-\left[\left(\partial_{x}\varphi_{n}\right)^{2}\left(\xi,x\right)\right]_{-B}^{B}\right)\mbox{d}x.
\]
 It can be easily verified that if both sets: 
\[
\pm B+2\mbox{supp}\theta_{1}=\pm B+2\bigcup_{n\in\mathbb{N}}\mbox{supp}\theta_{n}
\]
 do not intersect the free boundary, that is if $B$ is large enough
indeed, then as $n\to+\infty$: 
\[
\max_{x\in\overline{C}}\left|\partial_{\xi}\varphi_{n}\left(\pm B,x\right)-\partial_{\xi}\varphi\left(\pm B,x\right)\right|+\max_{x\in\overline{C}}\left|\partial_{x}\varphi_{n}\left(\pm B,x\right)-\partial_{x}\varphi\left(\pm B,x\right)\right|\to0.
\]

It follows that:
\[
\frac{1}{2}\int_{0}^{L}\left(\left[\left(\partial_{\xi}\varphi_{n}\right)^{2}\left(\xi,x\right)\right]_{-B}^{B}-\left[\left(\partial_{x}\varphi_{n}\right)^{2}\left(\xi,x\right)\right]_{-B}^{B}\right)\mbox{d}x\to\frac{1}{2}\int_{0}^{L}\left(\left[\left(\partial_{\xi}\varphi\right)^{2}\left(\xi,x\right)\right]_{-B}^{B}-\left[\left(\partial_{x}\varphi\right)^{2}\left(\xi,x\right)\right]_{-B}^{B}\right)\mbox{d}x.
\]

On the other hand:

\[
\int_{-B}^{B}\int_{0}^{L}\mbox{div}\left(E\nabla\varphi\right)\partial_{\xi}\left(\varphi-\varphi_{n}\right)\leq\|\mbox{div}\left(E\nabla\varphi\right)\|_{L^{2}\left(\left(-B,B\right)\times C\right)}\|\partial_{\xi}\left(\varphi-\varphi_{n}\right)\|_{L^{2}\left(\left(-B,B\right)\times C\right)},
\]
\[
\int_{-B}^{B}\int_{0}^{L}\mbox{div}\left(E\nabla\left(\varphi-\varphi_{n}\right)\right)\partial_{\xi}\varphi_{n}\leq\|\mbox{div}\left(E\nabla\left(\varphi-\varphi_{n}\right)\right)\|_{L^{2}\left(\left(-B,B\right)\times C\right)}\sup_{n\in\mathbb{N}}\|\partial_{\xi}\varphi_{n}\|_{L^{2}\left(\left(-B,B\right)\times C\right)},
\]
 and, once more by standard mollification theory, $\|\partial_{\xi}\left(\varphi-\varphi_{n}\right)\|_{L^{2}\left(\left(-B,B\right)\times C\right)}$
and $\|\mbox{div}\left(E\nabla\left(\varphi-\varphi_{n}\right)\right)\|_{L^{2}\left(\left(-B,B\right)\times C\right)}$
converge to $0$ as $n\to+\infty$. 

Therefore, passing to the limit $n\to+\infty$, we obtain the expected
equality:
\[
\int_{-B}^{B}\int_{0}^{L}\mbox{div}\left(E\nabla\varphi\right)\partial_{\xi}\varphi=\frac{1}{2}\int_{0}^{L}\left(\left[\left(\partial_{\xi}\varphi\right)^{2}\left(\xi,x\right)\right]_{-B}^{B}-\left[\left(\partial_{x}\varphi\right)^{2}\left(\xi,x\right)\right]_{-B}^{B}\right)\mbox{d}x.
\]

Finally, using these computations to pass to the limit $B\to+\infty$
in the equality:
\[
\int_{-B}^{B}\int_{0}^{L}\mbox{div}\left(E\nabla\varphi\right)\partial_{\xi}\varphi+s\int_{-B}^{B}\int_{0}^{L}\sigma\left[\varphi\right]\left(\partial_{\xi}\varphi\right)^{2}=-\int_{-B}^{B}\int_{0}^{L}\eta\left[\varphi\right]\partial_{\xi}\varphi.
\]
 it follows:
\[
s\int_{\mathbb{R}\times C}\sigma\left[\varphi\right]\left(\partial_{\xi}\varphi\right)^{2}=\int_{0}^{L}\left(\alpha^{2}\int_{0}^{a_{1}}zf_{1}\left(z,x\right)\mbox{d}z-d\int_{0}^{a_{2}}zf_{2}\left(z,x\right)\mbox{d}z\right)\mbox{d}x,
\]
 and since: 
\[
0<\min\left\{ 1,\frac{1}{d}\right\} \|\partial_{\xi}\varphi\|_{L^{2}\left(\mathbb{R}\times C\right)}^{2}\leq\int_{\mathbb{R}\times C}\sigma\left[\varphi\right]\left(\partial_{\xi}\varphi\right)^{2},
\]
the claimed relationship between $s$ and $\int_{0}^{L}\int_{-da_{2}}^{\alpha a_{1}}\eta\left(z,x\right)\mbox{d}z\mbox{d}x$
follows.
\end{proof}
\begin{cor}
\label{cor:sign_s} Let $\left(d,\alpha\right)\in\left(D_{exis},+\infty\right)\times\left(0,+\infty\right)$.
Then:
\begin{enumerate}
\item if $\frac{\alpha^{2}}{d}>\max\mathcal{R}_{d}^{0}$, $c_{\infty}>0$;
\item if $\frac{\alpha^{2}}{d}<\min\mathcal{R}_{d}^{0}$, $c_{\infty}<0$.
\end{enumerate}
\end{cor}

\begin{proof}
It suffices to remark that, for any $i\in\left\{ 1,2\right\} $, $\int_{0}^{L}\int_{0}^{a_{i}}zf_{i}\left(z,x\right)\mbox{d}z\mbox{d}x>0$.
\end{proof}
\begin{rem*}
We recall that, in the proof of Proposition \ref{prop:sign_non-zero_speed},
the fact that $a_{1}$ and $a_{2}$ are constant is crucial. This
issue has already been encountered (see the remark following Proposition
\ref{prop:improved_compactness}). Therefore, in the general setting,
it is not possible to obtain such an explicit formula for the sign
of $c_{\infty}$. Nevertheless, let us point out that the results
of Corollary \ref{cor:sign_s} should still hold in this case:
\begin{itemize}
\item there still exists $\overline{r}\geq\underline{r}>0$ such that $0\notin\Sigma_{\left(d,\alpha,f_{1},f_{2}\right)}$
if $\left(d,\alpha\right)$ does not satisfy $\underline{r}\leq\frac{\alpha^{2}}{d}\leq\overline{r}$,
since the whole subsection \ref{subsec:Necessary-conditions-zero-speed}
can be easily generalized (even though:
\begin{itemize}
\item we cannot prove that $c_{\infty}=0$ if $\alpha\in\left[\min A_{d},\max A_{d}\right]$,
i.e. if $\frac{\alpha^{2}}{d}\in\mathcal{R}_{d}^{0}$ (but recall
Conjecture \ref{conj:uniqueness_limit_speed});
\item additional care is needed since a non-constant $a_{2}$ would \textit{a
priori} depend on $d$);
\end{itemize}
\item we will prove in the next section that $\left(d,\alpha\right)\mapsto c_{\infty}$
is continuous at least in $\left\{ \left(d,\alpha\right)\in\left(D_{exis},+\infty\right)\times\left(0,+\infty\right)\ |\ \frac{\alpha^{2}}{d}\notin\mathcal{R}_{d}^{0}\right\} $;
\item the study of the limit of the segregated pulsating front as $\alpha\to0$
or $\alpha\to+\infty$ (which can be rigorously done since $D_{exis}$
does not depend on $\alpha$) should easily yield the sign of the
speed at such limits: 
\begin{itemize}
\item formally, as $\alpha\to0$, the positive part of $w$ vanishes and
we are left with a Fisher\textendash KPP pulsating front connecting
$0$ to $-da_{2}$, consequently with a negative speed;
\item formally, as $\alpha\to+\infty$, the negative part of $\frac{w}{\alpha}$
vanishes and we are left with a Fisher\textendash KPP pulsating front
connecting $a_{1}$ to $0$, consequently with a positive speed;
\end{itemize}
\item hence, by connectedness and continuity, Corollary \ref{cor:sign_s}
would be recovered indeed. 
\end{itemize}
\end{rem*}
To conclude, let us highlight an important particular case.
\begin{cor}
Assume that, for any $i\in\left\{ 1,2\right\} $, $f_{i}$ has the
particular form $\left(u,x\right)\mapsto\mu_{i}\left(x\right)\left(1-u\right)$
with $\mu_{i}\in\mathcal{C}_{per}^{1}\left(\mathbb{R}\right)$, $\mu_{i}\gg0$. 

Let: 
\[
r=\frac{\|\mu_{2}\|_{L^{1}\left(C\right)}}{\|\mu_{1}\|_{L^{1}\left(C\right)}}.
\]

If $c_{\infty}\neq0$, then it has the sign of $\alpha^{2}r-d$.
\end{cor}

\begin{proof}
In such a case, for any $i\in\left\{ 1,2\right\} $, $a_{i}=1$ and:
\[
\int_{0}^{L}\int_{0}^{1}zf_{i}\left(z,x\right)\mbox{d}z\mbox{d}x=\frac{1}{6}\int_{0}^{L}\mu_{i}\left(x\right)\mbox{d}x.
\]
\end{proof}

\subsection{Continuity of the asymptotic speed with respect to the parameters}

In this final subsection, we even allow $\left(f_{1},f_{2}\right)$
to vary in the set $\mathcal{F}$ of all $L$-periodic $f:[0,+\infty)\times\mathbb{R}\to\mathbb{R}$
satisfying $\left(\mathcal{H}_{1}\right)$, $\left(\mathcal{H}_{2}\right)$
and $\left(\mathcal{H}_{3}\right)$, equipped with the canonical topology
of $\mathcal{C}^{1}\left(\mathbb{R}^{2},\mathbb{R}\right)$. 
\begin{prop}
Assume that for any $\left(f_{1},f_{2}\right)\in\mathcal{F}^{2}$,
there exists a non-negative $D_{exis}=D_{exis}^{f_{1},f_{2}}$ as
defined before. 

Let: 
\[
\mathfrak{P}=\left\{ \left(d,\alpha,f_{1},f_{2}\right)\in\left(0,+\infty\right)^{2}\times\mathcal{F}^{2}\ |\ d>D_{exis}^{f_{1},f_{2}}\right\} .
\]

The function:
\[
\begin{matrix}\mathfrak{P} & \to & \mathbb{R}\\
\left(d,\alpha,f_{1},f_{2}\right) & \mapsto & c_{\infty}
\end{matrix}
\]
 is well-defined and continuous. 

Assume moreover that the function $\left(d,\alpha,f_{1},f_{2}\right)\in\mathfrak{P}\mapsto k^{\star}$
is locally bounded. Then the convergence of $\left(\left(d,\alpha,f_{1},f_{2}\right)\in\mathfrak{P}\mapsto c_{k}\right)_{k>k^{\star}}$
to $\left(d,\alpha,f_{1},f_{2}\right)\in\mathfrak{P}\mapsto c_{\infty}$
is locally uniform. 
\end{prop}

\begin{rem*}
If $\left(\mathcal{H}_{exis}\right)$ follows from $\left(\mathcal{H}_{freq}\right)$
\cite{Girardin_2016} and if: 
\[
D_{exis}=\left\{ \begin{matrix}M_{2}\left(\frac{L}{\pi}-\frac{1}{\sqrt{M_{1}}}\right)^{2} & \mbox{ if }L\sqrt{M_{1}}>\pi,\\
0 & \mbox{ if }L\sqrt{M_{1}}\leq\pi,
\end{matrix}\right.
\]
then $\left(f_{1,}f_{2}\right)\mapsto D_{exis}^{f_{1},f_{2}}$ is
indeed well-defined (and actually continuous) in $\mathcal{F}^{2}$.
\end{rem*}
\begin{proof}
Just verify (with the same integrations by parts than those used in
the course of the proofs of Propositions \ref{prop:compactness} and
\ref{prop:improved_compactness}) that: 
\begin{itemize}
\item all families of segregated pulsating fronts satisfy some locally uniform
estimates (with respect to $\left(d,\alpha,f_{1},f_{2}\right)$) in
$\mathcal{C}_{loc}\left(\mathbb{R},L_{loc}^{2}\left(\mathbb{R}\right)\right)\cap L_{loc}^{2}\left(\mathbb{R},H_{loc}^{1}\left(\mathbb{R}\right)\right)$
and therefore, by virtue of DiBenedetto\textquoteright s theory \cite{DiBenedetto_19},
in $\mathcal{C}_{loc}^{0,\beta}\left(\mathbb{R}^{2}\right)$;
\item all families of segregated stationary equilibrium satisfy some locally
uniform estimates (with respect to $\left(d,\alpha,f_{1},f_{2}\right)$)
in $\mathcal{C}_{loc}^{2,\beta}\left(\mathbb{R}\right)$.
\end{itemize}
The continuity of $c_{\infty}$ is then a classical consequence of
Theorems \ref{thm:uniqueness_SPF} and \ref{thm:SPF_SSE_exclusive}
and of compactness arguments. 

The locally uniform convergence is proved with similar compactness
arguments, this time using the fact that the compactness estimates
of Propositions \ref{prop:compactness} and \ref{prop:improved_compactness}
are locally uniform.
\end{proof}
\begin{rem*}
We recall that in the case of non-constant $a_{1}$ and $a_{2}$,
we cannot prove Theorem \ref{thm:SPF_SSE_exclusive}. Therefore it
is not possible to prove complete continuity of $c_{\infty}$. In
the whole subset:
\[
\left\{ \left(d,\alpha,f_{1},f_{2}\right)\in\mathfrak{P}\ |\ \frac{\alpha^{2}}{d}\in\mathcal{R}_{d,f_{1},f_{2}}^{0}\right\} ,
\]
 $c_{\infty}$ might not be continuous and jump between $0$ and some
non-zero values. Still, it is not possible to jump directly from a
positive value to a negative one, whence the zero set is in any case
non-empty. Moreover, we recall that these issues are completely subordinated
to Conjecture \ref{conj:uniqueness_limit_speed}. 
\end{rem*}

\subsubsection{As a conclusion: what about monotonicity?}

\paragraph*{Regarding the monotonicity of $\alpha\protect\mapsto c_{\infty}$:}

it should be easily established, via super- and sub-solutions, that
$\alpha\mapsto c_{\infty}$ is non-decreasing (a proof that we do
not detail here for the sake of brevity). Recall moreover that we
already suggested in the previous subsection that $c_{\infty}\to-c^{\star}\left[d,2\right]$
as $\alpha\to0$ and $c_{\infty}\to c^{\star}\left[1,1\right]$ as
$\alpha\to+\infty$, whence $\alpha\mapsto c_{\infty}$ would in fact
be from $\left(0,+\infty\right)$ onto $\left(-c^{\star}\left[d,2\right],c^{\star}\left[1,1\right]\right)$. 

\paragraph*{Regarding the monotonicity of $d\protect\mapsto c_{\infty}$: }

on the contrary, such a result should in general not be expected.
We recall that:
\begin{itemize}
\item the dependency of the speed of a bistable front on its diffusion coefficient
is in general unclear;
\item even for the Fisher\textendash KPP equation, as long as heterogeneity
is introduced, the monotonicity of the minimal speed as a function
of the diffusion coefficient is in general lost (for instance, in
space-time periodic media, a counter example has been exhibited by
the second author \cite{Nadin_2011}).
\end{itemize}
\bibliographystyle{amsplain}
\bibliography{ref}

\end{document}